\documentclass{aic}

\aicAUTHORdetails{%
  title = {A Completion of the Proof of the Edge-statistics Conjecture}, 
  author = {Jacob Fox, Lisa Sauermann},
  plaintextauthor = {Jacob Fox, Lisa Sauermann},
  keywords = {edge-statistics, inducibility}
}   

\aicEDITORdetails{%
   year={2020},
   number={4},
   received={13 March 2019},   
   published={28 February 2020},  
   doi={10.19086/aic.12047},      
}   

\usepackage{mathrsfs}

\newtheorem{theorem}{Theorem}[section]

\newtheorem{lemma}[theorem]{Lemma}
\newtheorem{conjecture}[theorem]{Conjecture}

\newtheorem{claim}[theorem]{Claim}
\newtheorem{corollary}[theorem]{Corollary}

\theoremstyle{definition}
\newtheorem{definition}[theorem]{Definition}

\newcommand{\eps}{\varepsilon}
\newcommand{\su}{\subseteq}
\newcommand{\sm}{\setminus}
\newcommand{\ind}{\operatorname{ind}}
\renewcommand{\l}{\ell}
\newcommand{\E}{\mathbb{E}}

\newcommand{\Vh}{V_{\textnormal{high}}}
\newcommand{\Vm}{V_{\textnormal{med}}}
\newcommand{\Vl}{V_{\textnormal{low}}}
\newcommand{\PR}{\mathcal{P}\!_{<r}}
\newcommand{\U}{\operatorname{U}}

\begin{document}

\begin{frontmatter}[classification=text]


\author[jacob]{Jacob Fox\thanks{Research supported by a Packard Fellowship and by NSF Career Award DMS-1352121.}}
\author[lisa]{Lisa Sauermann}

\begin{abstract}
For given integers $k$ and $\l$ with $0<\ell< {k \choose 2}$, Alon, Hefetz, Krivelevich and Tyomkyn formulated the following conjecture: When sampling a $k$-vertex subset uniformly at random from a very large graph $G$, then the probability to have exactly $\ell$ edges within the sampled $k$-vertex subset is at most $e^{-1}+o_k(1)$. This conjecture was proved in the case $\Omega(k)\leq \ell\leq {k \choose 2}-\Omega(k)$ by Kwan, Sudakov and Tran. In this paper, we complete the proof of the conjecture by resolving the remaining cases. We furthermore give nearly tight upper bounds for the probability described above in the case $\omega(1)\leq \ell\leq o(k)$. We also extend some of our results to hypergraphs with bounded edge size.
\end{abstract}
\end{frontmatter}

\section{Introduction}

Given integers $k\geq 1$ and $0\leq \l\leq {k\choose 2}$, Alon, Hefetz, Krivelevich and Tyomkyn \cite{alon-et-al} asked the following question: If we choose a $k$-vertex subset uniformly at random from a very large graph, what is the maximum possible probability to obtain exactly $\l$ edges within the random $k$-vertex subset? More precisely, for $n\geq k$ they defined $I(n,k,\l)$ to be the maximum possible probability of the event to have exactly $\l$ edges within a uniformly random $k$-vertex subset sampled from an $n$-vertex graph $G$ (where the maximum is taken over all $n$-vertex graphs $G$). Observing that $I(n,k,\l)$ is a monotone decreasing function of $n$, they defined the \emph{edge-inducibility} $\ind(k,\l)=\lim_{n\to\infty} I(n,k,\l)$.

Clearly, $0\leq \ind(k,\l)\leq 1$. If $\l=0$ or $\l={k\choose 2}$, it is easy to see that $\ind(k,\l)=1$ by taking $G$ to be an empty or a complete graph, respectively. However, Alon et al.\ \cite{alon-et-al} showed that for $0< \l< {k\choose 2}$ the edge-inducibility $\ind(k,\l)$ is bounded away from 1. Furthermore, they formulated the following stronger conjecture, which they called the Edge-statistics Conjecture:

\begin{conjecture}[The Edge-statistics Conjecture, \cite{alon-et-al}]\label{conj-1-e}For all integers $k$ and $\l$ with $0< \l<{k\choose 2}$, we have
\[\ind(k,\l)\leq \frac{1}{e}+o_k(1).\]
\end{conjecture}

The following observations of Alon et al.\ \cite{alon-et-al} motivate this conjecture and show that the constant $1/e$ in the conjecture is the smallest possible: By considering a random graph $G(n,p)$ with $p=1/{k\choose 2}$, one can see that $\ind(k,1)\geq  e^{-1}+o_k(1)$. Furthermore, by considering a complete bipartite graph with parts of sizes $n/k$ and $(k-1)n/k$, one obtains $\ind(k,k-1)\geq  e^{-1}+o_k(1)$.

It is easy to see by considering graph complements that $\ind(k,\l)=\ind(k,{k\choose 2}-\l)$. Hence, when studying the edge-inducibility $\ind(k,\l)$, one may assume that $\l\leq {k\choose 2}/2$.

As mentioned above, Alon et al.\ \cite{alon-et-al} proved that $\ind(k,\l)<1-\eps$ for all $0< \l<{k\choose 2}$ with some absolute constant $\eps$. They furthermore established Conjecture \ref{conj-1-e} for $\Omega(k^2)\leq \l\leq  {k\choose 2}/2$ by proving that $\ind(k,\l)\leq O(k^{-0.1})$ in this range. They made further partial progress towards their conjecture by showing that $\ind(k,\l)\leq \frac{1}{2}+o_k(1)$ if $\omega(k)\leq \l\leq  {k\choose 2}/2$ as well as $\ind(k,\l)\leq \frac{3}{4}+o_k(1)$ for fixed $\l>1$ and $\ind(k,1)\leq \frac{1}{2}+o_k(1)$.

Kwan, Sudakov and Tran \cite{kwan-et-al} subsequently proved that $\ind(k,\l)\leq \sqrt{(k/\l)}\cdot \log^{O(1)}(3\l/k)$ for $k\leq \l\leq  {k\choose 2}/2$. From considering a complete bipartite graph with appropriately sized parts, one can see that their bound is best-possible up to the logarithmic factors. Clearly, their result implies Conjecture \ref{conj-1-e} for $C\cdot k\leq \l\leq  {k\choose 2}/2$, where $C$ is a sufficiently large absolute constant.

Our main contribution in this paper is to resolve Conjecture \ref{conj-1-e} in the remaining cases, namely for $1\leq \l\leq C\cdot k$, where $C$ is an appropriately chosen constant such that the result of Kwan et al.\  \cite{kwan-et-al} yields Conjecture \ref{conj-1-e} for $C\cdot k\leq \l\leq  {k\choose 2}/2$. More precisely, we prove the following:

\begin{theorem}\label{thm-graph-1-e}For any constant $C$, the following is true: For all integers $k$ and $\l$ with $1\leq \l\leq C\cdot k$, we have
\[\ind(k,\l)\leq \frac{1}{e}+o_k(1).\]
\end{theorem}
Here, the $o_k(1)$-term converges to zero as $k\to\infty$ while $C$ is fixed (the term may depend on $C$, but not on $\l$).

As mentioned above, together with \cite[Theorem 1.1]{kwan-et-al}, this theorem implies Conjecture \ref{conj-1-e}. It furthermore improves \cite[Theorem 1.6]{alon-et-al}, which states the above-mentioned bounds for $\ind(k,\l)$ while $\l$ is fixed.

Theorem \ref{thm-graph-1-e} implies that $\ind(k,1)=  e^{-1}+o_k(1)$ and $\ind(k,k-1)=  e^{-1}+o_k(1)$ (the correponding lower bounds were discussed above). It is an interesting question to study $\ind(k,\l)$ in the ranges away from $\l=1$ and $\l=k-1$. Alon et al.\ \cite[Conjecture 6.1]{alon-et-al} conjectured that $\ind(k,\l)=o_k(1)$ if $\omega(k)\leq \l\leq {k\choose 2}/2$. Kwan et al.\  \cite{kwan-et-al} proved this conjecture as a consequence of their result that $\ind(k,\l)\leq \sqrt{(k/\l)}\cdot \log^{O(1)}(3\l/k)$ for $k\leq \l\leq  {k\choose 2}/2$. Our next theorem shows that $\ind(k,\l)=o_k(1)$ if  $\omega(1)\leq \l\leq o(k)$, which can be seen as an analogue of \cite[Conjecture 6.1]{alon-et-al} in the sublinear range.

\begin{theorem}\label{thm-graph-o1}Let us fix a real number $\eps>0$. Then for all sufficiently large integers $k$, the following holds: For all integers $\l$ with $1\leq \l\leq k/\log^4 k$, we have
\[\ind(k,\l)\leq 90\cdot \l^{-1/4}\]
and for all integers $\l$ with $k/\log^4 k \leq \l\leq (1-\eps)k/2$ we have
\[\ind(k,\l)\leq 90\cdot k^{-1/4} \cdot \log k.\]
\end{theorem}

It is not hard to see that the exponent $-1/4$ in Theorem \ref{thm-graph-o1} is best-possible. Indeed, if $1\leq \l\leq k$ and $\l={m\choose 2}$ for some positive integer $m$, then considering an $n$-vertex graph $G$ consisting of a clique on $(m/k)\cdot n$ vertices (and no additional edges) shows that $\ind(k,\l)\geq \Omega(m^{-1/2})=\Omega(\l^{-1/4})$.

Kwan et al.\ \cite{kwan-et-al} started studying the edge-inducibility for $r$-uniform hypergraphs, a direction suggested by Alon et al.\ \cite{alon-et-al}. For a hypergraph $G$ with at least $k$ vertices, let $I(G,k,\l)$ be the probability that a uniformly random $k$-vertex subset of $G$ contains exactly $\l$ edges. Now, for $r\geq 1$, $0\leq \l\leq {k\choose r}$ and $n\geq k$, let $I_r(n,k,\l)$ be the maximum of $I(G,k,\l)$ over all  $r$-uniform hypergraphs $G$ on $n$ vertices. Again, $I_r(n,k,\l)$ is a monotone decreasing function of $n$, and one can define $\ind_r(k,\l)=\lim_{n\to\infty} I_r(n,k,\l)$. With this notation, we have $\ind(k,\l)=\ind_2(k,\l)$. A natural analogue of Conjecture \ref{conj-1-e}, suggested by Alon et al.\ \cite{alon-et-al}, is that for fixed $r$ one has $\ind_r(k,\l)\leq e^{-1}+o_k(1)$ for all $0< \l< {k\choose r}$. Note that $\ind_r(k,\l)\geq e^{-1}+o_k(1)$ whenever $\l={k-s\choose r-s}$ for some $1\leq s\leq r$. This can be seen by taking $H$ to be an auxiliary random $s$-uniform hypergraph $G_s(n,p)$ with $p=1/{k\choose s}$, and defining the edges of $G$ to be all $r$-sets that contain one of the edges of $H$. For $s=r-1$, we in particular see that $\ind_r(k,k-r+1)\geq e^{-1}+o_k(1)$.

Our results for graphs for the sublinear range $1\leq \l\leq o(k)$ also extend to $r$-uniform hypergraphs. In fact, they also extend to not necessarily uniform hypergraphs all of whose edges have size bounded by $r$. Here, we insist that all edges of a hypergraph are non-empty. For $r\geq 1$, $0\leq \l\leq {k\choose r}$ and $n\geq k$, let $I_{\leq r}(n,k,\l)$ be the maximum of $I(G,k,\l)$ over all hypergraphs $G$ on $n$ vertices all of whose edges have size at most $r$. Again, $I_{\leq r}(n,k,\l)$ is a monotone decreasing function of $n$, and one can define $\ind_{\leq r}(k,\l)=\lim_{n\to\infty} I_{\leq r}(n,k,\l)$. We clearly have $\ind_r(k,\l)\leq \ind_{\leq r}(k,\l)$. However, note that one cannot hope for $\ind_{\leq r}(k,\l)\leq e^{-1}+o_k(1)$ for all $0<\l<{k\choose r}$, as $\ind_{\leq r}(k,\l)=1$ whenever $\l={k\choose t}$ for some $1\leq t\leq r-1$ (this can be seen by considering a complete $t$-uniform hypergraph).

Note that the case $r=1$ is not interesting, as $I_1(n,k,\l)=I_{\leq 1}(n,k,\l)$ is just given by maximizing a certain probability in a hypergeometric distribution. In particular, it is not hard to check that $\ind_1(k,\l)= \ind_{\leq 1}(k,\l)\leq e^{-1}+o_k(1)$ for all $1\leq \l\leq k-1$.

Our first result for hypergraphs is that $\ind_r(k,\l)\leq \ind_{\leq r}(k,\l)\leq e^{-1}+o_k(1)$ if $1\leq \l\leq o(k)$. This is an immediate consequence of the following theorem.

\begin{theorem}\label{thm-hyper-1-e}For all positive integers $r$, $k$ and $\l$ with $1\leq \l< k/r$, we have
\[\ind_r(k,\l)\leq \ind_{\leq r}(k,\l)\leq \frac{k}{k-r\l}\cdot \frac{1}{e}.\]
\end{theorem}

The following theorem extends Theorem \ref{thm-graph-o1} to the hypergraph setting.

\begin{theorem}\label{thm-hyper-o1}Let us fix a positive integer $r\geq 3$ and a real number $\eps>0$. Then for all sufficiently large integers $k$, we have
\[\ind_r(k,\l)\leq\ind_{\leq r}(k,\l)\leq 100\cdot \l^{-1/(2r)}\]
for all $1\leq \l\leq (1-\eps)k/r$.
\end{theorem}

Again, the exponent $-1/(2r)$ is best-possible. Indeed, if $1\leq \l\leq k$ and $\l={m\choose r}$ for some positive integer $m$, then considering an $n$-vertex hypergraph $G$ consisting of an $r$-uniform clique on $(m/k)\cdot n$ vertices (and no additional edges) shows that $\ind(k,\l)\geq \Omega(m^{-1/2})=\Omega(\l^{-1/(2r)})$.

Another possible direction to extend the study of the edge-inducibility $\ind(k,\l)$ for graphs, is to consider more restricted families of induced subgraphs on $k$-vertices with $\l$ edges. One initial motivation for Alon et al.\ \cite{alon-et-al} to introduce and study edge-inducibilities was the close relationship to graph inducibilities, as introduced by Pippenger and Golumbic \cite{pippenger-golumbic}. The inducibility of a $k$-vertex graph $H$ is the maximum possible probability of the event to have an induced subgraph isomorphic to $H$ when sampling a $k$-vertex subset uniformly at random from a large graph $G$ (where again, we consider the limit as the number of vertices of $G$ tends to infinity). The graph inducibility problem has attracted a lot of attention recently, see for example \cite{balogh-hu-lidicky-pfender, hefetz-tyomkyn, kral-norin-volec, yuster}. If $H$ has $k$ vertices and $\l$ edges, then the graph inducibility of $H$ is clearly bounded by $\ind(k,\l)$. We think that it is an interesting question to study notions of inducibility that lie between graph inducibility and edge-inducibility.

For example, one can ask the following question: What is the maximum possible probability of the event to have an induced forest with precisely $\l$ edges when sampling a $k$-vertex subset uniformly at random from a large graph $G$? The following theorem implies an upper bound for this probability if $1\leq \l\leq \sqrt{k}/4$.

\begin{theorem}\label{thm-graph-forest}Let $k$ and $\l$ be positive integers with $1\leq \l\leq \sqrt{k}/4$ and let $G$ be a graph on $n \geq k$ vertices. Then the number of $k$-vertex subsets $A\su V(G)$ that induce a forest with exactly $\l$ edges is at most
\[50\cdot \l^{-1/2}\cdot \frac{n^k}{k!}.\]
\end{theorem}

Note that for sufficiently large $n$, Theorem \ref{thm-graph-forest} is tight up to a constant factor. Indeed, consider a graph $G$ consisting of a complete bipartite graph with parts of sizes $n/k$ and $\l\cdot n/k$ and with $(k-\l-1)\cdot n/k$ additional isolated vertices. When sampling a $k$-vertex subset uniformly at random from this graph, with probability $\Omega(\l^{-1/2})$ we obtain an induced star with $\l$ edges together with $k-\l-1$ additional isolated vertices. Thus, Theorem \ref{thm-graph-forest} is tight up to a constant factor for all $1\leq \l\leq \sqrt{k}/4$. If $\l=o(\sqrt{k})$ one can also consider a random graph $G(n,p)$ with $p=\l/{k \choose 2}$ and note that when sampling a $k$-vertex subset, one obtains an induced matching of exactly $\l$ edges with probability $\Omega(\l^{-1/2})$.  

We remark that we did not attempt to optimize the constants in Theorems \ref{thm-graph-o1}, \ref{thm-hyper-o1} and \ref{thm-graph-forest}.

\noindent\textit{Organization.} We will first prove Theorem \ref{thm-hyper-1-e} in Section \ref{sect-hyper-1-e}. Note that applying this theorem to $r=2$ yields Theorem \ref{thm-graph-1-e} in the case where $1\leq \l\leq o(k)$. In Section \ref{sect-proof-o1}, we will prove Theorems \ref{thm-graph-o1}, \ref{thm-hyper-o1} and \ref{thm-graph-forest}. The proofs will rely on three lemmas, which we will prove in Sections \ref{sect-proof-propo-3} and \ref{sect-proof-propo-1-2}. A corollary of these lemmas will also be an important ingredient in the proof of Theorem \ref{thm-graph-1-e}. This proof can be found in Section \ref{sect-proof-graph-1-e}, apart from the proofs of several other lemmas which we will postpone to Sections \ref{sect-proofs-two-lemmas} to \ref{sect-A-interesting}.

\noindent\textit{Notation.} All logarithms are to base 2. For integers $m\geq 1$ and $0\leq t\leq m$, let $(m)_t$ denote the falling factorial ${m\choose t}\cdot t!=m\cdot (m-1)\dotsm (m-t+1)=\prod_{i=0}^{t-1}(m-i)$. Note that $(m)_0=1$.

All edges of all hypergraphs are assumed to be non-empty.

For a hypergraph $G$, let $V(G)$ denote its vertex set. For a vertex $v\in V(G)$, let $\deg_G(v)$ be its degree (which is the number of edges of $G$ that contain $v$). Similarly, for a subset $W\su V(G)$ and a vertex $v\in W$, let $\deg_W(v)$ denote the degree of $v$ in the sub-hypergraph induced by $W$.

Furthermore, for a hypergraph $G$ and a subset $W\su V(G)$, let $e(W)$ be the number of edges of $G$ that are subsets of $W$. Let us call a vertex $v\in W$ \emph{non-isolated in $W$} if there is an edge $e\su W$ with $v\in e$, and let us call $v$ \emph{isolated in $W$} otherwise. Finally, let $m(W)$ denote the number of vertices of $W$ that are non-isolated in $W$. If all edges of $G$ have size at most $r$, then we clearly have $m(W)\leq r\cdot e(W)$.

For a graph $G$ and a vertex $v\in V(G)$, let $N(v)\su V(G)\sm \lbrace v\rbrace$ denote the neighborhood of $v$ in $G$.

All the $o_k(1)$-terms are supposed to only depend on $k$ and $C$ (and on none of the other variables or objects), and to converge to zero as $k\to \infty$ for every fixed $C$.

\section{Proof of Theorem \ref{thm-hyper-1-e}}\label{sect-hyper-1-e}

In this section, we will prove Theorem \ref{thm-hyper-1-e}. So let us fix positive integers $r$, $k$ and $\l$ satisfying $1\leq \l<k/r$. In particular, $k\geq r+1\geq 2$.

Furthermore, let $G$ be a hypergraph on $n$ vertices all of whose edges have size at most $r$, and let us assume that $n$ is large with respect to $k$. We will show that the number of $k$-vertex subsets $A\su V(G)$ such that the set $A$ contains exactly $\l$ edges is at most
\[\frac{k}{k-r\l}\cdot \frac{1}{e}\cdot \frac{n^k}{k!}=(1+o_n(1))\cdot\frac{k}{k-r\l}\cdot \frac{1}{e}\cdot {n\choose k},\]
thus proving
\[\ind_{r}(k,\l)\leq \ind_{\leq r}(k,\l)\leq \frac{k}{k-r\l}\cdot \frac{1}{e}\]
as desired.

We may assume that there is at least one $k$-vertex subset $A\su V(G)$ that contains exactly $\l$ edges.

Let us introduce some more notation. For a subset $W\su V(G)$, let
\[\PR(W)=\lbrace X\su W\mid \vert X\vert<r\rbrace\]
be the family of all subsets of $W$ of size smaller than $r$. Furthermore, for a subfamily $\mathcal{X}\su \PR(W)$, let
\[\U(\mathcal{X})=\bigcup_{X\in \mathcal{X}}X\su W\]
be the union of all members of $\mathcal{X}$.

For a vertex $v\in V(G)$ and a subset $W\su V(G)\sm \lbrace v\rbrace$, let
\[\mathcal{N}(v,W)=\lbrace X \su W\rbrace\mid X\cup \lbrace v\rbrace\in E(G)\rbrace.\]
This means $\mathcal{N}(v,W)$ is the family of all subsets $X\su W$ such that $X\cup \lbrace v\rbrace$ is an edge of the hypergraph $G$. In other words, $\mathcal{N}(v,W)$ consists of the sets $e\sm \lbrace v\rbrace$ for all edges $e\in E(G)$ with $e\su W\cup \lbrace v\rbrace$ and $v\in e$. As every edge of $G$ has size at most $r$, we have $\mathcal{N}(v,W)\su \PR(W)$. Also note that $\vert \mathcal{N}(v,W)\vert$ is the number of edges inside $W\cup \lbrace v\rbrace$ that contain $v$, hence
\[\vert \mathcal{N}(v,W)\vert=\deg_{W\cup \lbrace v\rbrace}(v).\]

The following definition will be crucial for our proof of Theorem \ref{thm-hyper-1-e}.

\begin{definition}A sequence $v_1,\dots,v_k$ of distinct vertices of $G$ is called \emph{good} if $e(\lbrace v_1,\dots,v_k\rbrace)=\l$ and $v_k$ is non-isolated in $\lbrace v_1,\dots,v_k\rbrace$. For $0\leq j\leq k-1$ a sequence $v_1,\dots,v_j$ of distinct vertices of $G$ is called \emph{good} if it can be extended to a good sequence $v_1,\dots,v_k$.
\end{definition}

Recall that the condition that $v_k$ is non-isolated in $\lbrace v_1,\dots,v_k\rbrace$ means that there is at least one edge $e\su \lbrace v_1,\dots,v_k\rbrace$ with $v_k\in e$. If $r=2$ and $G$ is a graph, this simply means that $v_k$ has an edge to at least one of the vertices $v_1,\dots,v_{k-1}$.

Note that any good sequence $v_1,\dots,v_k$ must satisfy $e(\lbrace v_1,\dots,v_{k-1}\rbrace)<e(\lbrace v_1,\dots,v_{k}\rbrace)=\l$ and consequently $e(\lbrace v_1,\dots,v_j\rbrace)<\l$ for all $0\leq j\leq k-1$. Furthermore note that for any good sequence $v_1,\dots,v_k$, we have $m(\lbrace v_1,\dots,v_k\rbrace)\geq 1$ as $e(\lbrace v_1,\dots,v_k\rbrace)=\l\geq 1$.

For $0\leq j\leq k-1$ and any good sequence $v_1,\dots,v_j$, set $\lambda(v_1,\dots,v_j)=1$. On the other hand, for a good sequence $v_1,\dots,v_k$, set
\[\lambda(v_1,\dots,v_k)=\frac{1}{m(\lbrace v_1,\dots,v_k\rbrace)}.\]

Now, for $0\leq j\leq k-1$ and a good sequence $v_1,\dots,v_j$, set
\[\Lambda(v_1,\dots,v_j)=\sum_{\substack{v_{j+1}\text{ s.t. }\\v_1,\dots,v_{j+1}\text{ good}}}\lambda(v_1,\dots,v_{j+1}).\]
Note that if $0\leq j\leq k-2$, then $\Lambda(v_1,\dots,v_j)$ simply counts the number of choices for $v_{j+1}\in V(G)$ such that $v_1,\dots,v_{j+1}$ is a good sequence (as $\lambda(v_1,\dots,v_{j+1})=1$ for all such $v_{j+1}$).

Furthermore note that $\Lambda(v_1,\dots,v_j)>0$ for any $0\leq j\leq k-1$ and any good sequence $v_1,\dots,v_j$ (because there is at least one way to extend $v_1,\dots,v_j$ to a good sequence $v_1,\dots,v_k$).

For $1\leq j\leq k$ and a good sequence $v_1,\dots,v_j$, set
\[\rho(v_1,\dots,v_j)=\prod_{i=1}^{j}\frac{\lambda(v_1,\dots,v_i)}{\Lambda(v_1,\dots,v_{i-1})}=\frac{\lambda(v_1)\cdot \lambda(v_1,v_2)\cdot \lambda(v_1,v_2,v_3)\dotsm \lambda(v_1,\dots,v_j)}{\Lambda(\emptyset)\cdot \Lambda(v_1)\cdot \Lambda(v_1, v_2)\dotsm \Lambda(v_1,\dots,v_{j-1})}.\]
Here, by slight abuse of notation, $\emptyset$ stands for the empty sequence.

\begin{lemma}\label{lemma1}For every $1\leq j\leq k$, we have
\[\sum_{v_1,\dots,v_{j} \textnormal { good}}\rho(v_1,\dots,v_j)= 1.\]
\end{lemma}
\begin{proof}
Let us prove the lemma by induction on $j$. Note that for $j=1$, by definition, $\Lambda(\emptyset)$ is precisely the number of $v_1\in V(G)$  such that $v_1$ is a good sequence. Thus, 
\[\sum_{v_1 \text { good}}\rho(v_1)=\sum_{v_1 \text { good}}\frac{\lambda(v_1)}{\Lambda(\emptyset)}=\sum_{v_1 \text { good}}\frac{1}{\Lambda(\emptyset)}=1.\]
Now assume $2\leq j\leq k$ and that we have already proved the lemma for $j-1$. For every good sequence $v_1,\dots,v_{j}$, the sequence $v_1,\dots,v_{j-1}$ is also good. Hence
\begin{multline*}
\sum_{v_1,\dots,v_{j} \text { good}}\rho(v_1,\dots,v_j)=\sum_{v_1,\dots,v_{j} \text { good}}\rho(v_1,\dots,v_{j-1})\cdot  \frac{\lambda(v_1,\dots,v_{j})}{\Lambda(v_1,\dots,v_{j-1})}\\
=\sum_{v_1,\dots,v_{j-1} \text { good}}\ \sum_{\substack{v_j\text{ s.t.\ }\\v_1,\dots,v_{j} \text { good}}} \rho(v_1,\dots,v_{j-1})\cdot  \frac{\lambda(v_1,\dots,v_{j})}{\Lambda(v_1,\dots,v_{j-1})}\\
=\sum_{v_1,\dots,v_{j-1} \text { good}}\left(\frac{\rho(v_1,\dots,v_{j-1})}{\Lambda(v_1,\dots,v_{j-1})}\cdot \sum_{\substack{v_j\text{ s.t.\ }\\v_1,\dots,v_{j} \text { good}}} \lambda(v_1,\dots,v_{j})\right)\\
=\sum_{v_1,\dots,v_{j-1} \text { good}} \frac{\rho(v_1,\dots,v_{j-1})}{\Lambda(v_1,\dots,v_{j-1})}\cdot \Lambda(v_1,\dots,v_{j-1})=\sum_{v_1,\dots,v_{j-1} \text { good}} \rho(v_1,\dots,v_{j-1})=1,
\end{multline*}
using the induction hypothesis in the last step.
\end{proof}

Recall that we want to show that the number of $k$-vertex subsets $A\su V(G)$ with $e(A)=\l$ is at most
\[\frac{k}{k-r\l}\cdot \frac{1}{e}\cdot \frac{n^k}{k!}.\]

Note that for every good sequence $v_1,\dots,v_k$, the set $A=\lbrace v_1,\dots,v_k\rbrace$ satisfies $A\su V(G)$ and $e(A)=\l$.  Conversely, for every $k$-vertex subset $A\su V(G)$ with $e(A)=\l$ we can find certain labelings of the elements of $A$ as $v_1,\dots, v_k$ such that $v_1,\dots, v_k$ is a good sequence (we just need to ensure that $v_k$ is one of the $m(A)$ non-isolated vertices in $A$).

By Lemma \ref{lemma1} applied to $j=k$, we have
\[\sum_{\substack{A\su V(G)\\\vert A\vert=k,\, e(A)=\l}}\ \sum_{\substack{\text{labelings }v_1,\dots, v_k\text{ of } A\\\text{s.t.\ }v_1,\dots, v_k\text{ is good}}}\rho(v_1,\dots,v_k)= 1.\]
In order to prove the desired bound on the number of $k$-vertex subsets $A$ with $e(A)=\l$, it therefore suffices to show that
\[\sum_{\substack{\text{labelings }v_1,\dots, v_k\text{ of } A\\\text{s.t.\ }v_1,\dots, v_k\text{ is good}}}\rho(v_1,\dots,v_k)\geq \frac{k-r\l}{k}\cdot e\cdot \frac{k!}{n^k}\]
for every $k$-vertex subset $A\su V(G)$ with $e(A)=\l$.

So from now on, let us fix a $k$-vertex subset $A\su V(G)$ with $e(A)=\l$. 

Note that there are exactly $m(A)\cdot (k-1)!$ ways to label the elements of $A$ by $v_1,\dots,v_k$ such that $v_1,\dots,v_k$ is a good sequence. This is because there are exactly $m(A)$ choices for $v_k\in A$ (as $v_k$ needs to be non-isolated in $A$), and then any labeling $v_1,\dots,v_{k-1}$ of the remaining $k-1$ vertices yields a good sequence $v_1,\dots,v_k$.

It suffices to prove that for each of the $m(A)$ choices for $v_k\in A$, we have
\[\sum_{\substack{\text{labelings }v_1,\dots, v_{k-1}\text{ of } A\sm \lbrace v_k\rbrace\\\text{s.t.\ }v_1,\dots, v_k\text{ is good}}}\rho(v_1,\dots,v_k)\geq \frac{1}{m(A)} \cdot \frac{k-r\l}{k}\cdot e\cdot \frac{k!}{n^k}.\]
So fix $v_k\in A$ such that $v_k$ is non-isolated in $A$. Set $A'=A\sm \lbrace v_k\rbrace$ and note that for any labeling $v_1,\dots, v_{k-1}$ of $A'$, the sequence $v_1,\dots, v_{k}$ is good. So it suffices to prove
\begin{equation}\label{inequality}
\sum_{\text{labelings }v_1,\dots, v_{k-1}\text{ of } A'}\rho(v_1,\dots,v_k)\geq \frac{1}{m(A)} \cdot \frac{k-r\l}{k}\cdot e\cdot \frac{k!}{n^k}.
\end{equation}

Let $d=\deg_A(v_k)$. Note that $1\leq d\leq e(A)=\l$ and
\[e(A')=e(A\sm \lbrace v_k\rbrace)=e(A)-\deg_A(v_k)=\l-d<\l.\]

Furthermore, let $B\su A'$ consist of those vertices in $A'$ that are non-isolated in $A'$. Then we have $\vert B\vert=m(A')\leq r\cdot e(A')=r\cdot (\l-d)$ and $e(B)=e(A')=\l-d$.

Note that for any labeling $v_1,\dots,v_{k-1}$ of $A'$ we have
\begin{multline}\label{eq-w}
\rho(v_1,\dots,v_k)=\frac{\lambda(v_1)\cdot \lambda(v_1,v_2)\cdot \lambda(v_1,v_2,v_3)\dotsm \lambda(v_1,\dots,v_k)}{\Lambda(\emptyset)\cdot \Lambda(v_1)\cdot \Lambda(v_1, v_2)\dotsm \Lambda(v_1,\dots,v_{k-1})}\\
=\frac{1\dotsm 1\cdot (1/m(\lbrace v_1,\dots,v_{k}\rbrace))}{\Lambda(\emptyset)\cdot \Lambda(v_1)\cdot \Lambda(v_1, v_2)\dotsm \Lambda(v_1,\dots,v_{k-1})}\\
=\frac{1}{m(A)}\cdot \frac{1}{\Lambda(\emptyset)\cdot \Lambda(v_1)\cdot \Lambda(v_1, v_2)\dotsm \Lambda(v_1,\dots,v_{k-1})}.
\end{multline}
Let us first analyze the quantity $\Lambda(v_1,\dots,v_{k-1})$. In particular, we will see that this quantity is independent of the labeling $v_1,\dots,v_{k-1}$ of $A'$.

For every labeling $v_1,\dots,v_{k-1}$ of $A'$, the possible extensions of $v_1,\dots,v_{k-1}$ to a good sequence $v_1,\dots,v_{k-1}, v_k'$ are given by those vertices $v_k'\in V(G)\sm A'$ that satisfy $e(A'\cup \lbrace v_k'\rbrace)=\l$ (note that then $v_k'$ will automatically be non-isolated in $\lbrace v_1,\dots,v_{k-1}, v_k'\rbrace=A'\cup \lbrace v_k'\rbrace$, because $e(A')=\l-d<\l$). As we always have
\[\vert \mathcal{N}(v_k',A')\vert=\deg_{A'\cup  \lbrace v_k'\rbrace}(v_k')=e(A'\cup \lbrace v_k'\rbrace)-e(A')=e(A'\cup \lbrace v_k'\rbrace)-(\l-d),\]
the condition $e(A'\cup \lbrace v_k'\rbrace)=\l$ is equivalent to $\vert \mathcal{N}(v_k',A')\vert=d$. Hence the extensions of $v_1,\dots,v_{k-1}$ to a good sequence $v_1,\dots,v_{k-1}, v_k'$ are given by those $v_k'\in V(G)\sm A'$  that satisfy $\vert \mathcal{N}(v_k',A')\vert=d$. Thus,
\begin{multline}\label{eq-last-Q}
\Lambda(v_1,\dots,v_{k-1})=\sum_{\substack{v_{k}'\text{ s.t. }\\v_1,\dots,v_{k-1},v_k' \text{ good}}}\lambda(v_1,\dots,v_{k-1}, v_k')\\
=\sum_{\substack{v_{k}'\in V(G)\sm A'\\\vert \mathcal{N}(v_k',A')\vert=d}}\lambda(v_1,\dots,v_{k-1}, v_k')=\sum_{\substack{\mathcal{X}\su \PR(A')\\ \vert \mathcal{X}\vert=d}}\ \sum_{\substack{v_{k}'\in V(G)\sm A'\\ \mathcal{N}(v_k',A')=\mathcal{X}}}\lambda(v_1,\dots,v_{k-1}, v_k')
\end{multline}
for every labeling $v_1,\dots,v_{k-1}$ of $A'$.

For every subfamily $\mathcal{X}\su \PR(A')$ with $\vert \mathcal{X}\vert=d$, let $0\leq c(\mathcal{X})\leq 1$ be such that the number of vertices $v\in V(G)\sm A'$ with $\mathcal{N}(v,A')=\mathcal{X}$ is precisely $c(\mathcal{X})\cdot n$. Then for every such $\mathcal{X}$ and every labeling $v_1,\dots,v_{k-1}$ of $A'$, the second sum on the right-hand side of (\ref{eq-last-Q}) has exactly $c(\mathcal{X})\cdot n$ summands. Note that $\mathcal{N}(v,A')=\mathcal{X}$ means that $X\cup \lbrace v\rbrace\in E(G)$ for all $X\in \mathcal{X}$, but for no other subsets $X\su A'$.

Now, for every labeling $v_1,\dots,v_{k-1}$ of $A'$, for every subfamily $\mathcal{X}\su \PR(A')$ with $\vert \mathcal{X}\vert=d$ and for every vertex $v_k'\in V(G)\sm A'$ with $\mathcal{N}(v_k',A')=\mathcal{X}$, we claim that we have $m(\lbrace v_1,\dots,v_{k-1}, v_k'\rbrace)=\vert B\cup \U(\mathcal{X})\vert +1$. Indeed, $\lbrace v_1,\dots,v_{k-1}, v_k'\rbrace=A'\cup \lbrace v_k'\rbrace$. The edges inside $A'\cup \lbrace v_k'\rbrace$ are precisely the edges inside $A'$ and the sets of the form $X\cup \lbrace v_k'\rbrace$ for $X\in \mathcal{N}(v_k',A')=\mathcal{X}$. Hence the non-isolated vertices in $A'\cup \lbrace v_k'\rbrace$ are precisely the vertices in the set 
\[\bigcup_{\substack{e\in E(G)\\e\su A'}}e\ \cup \bigcup_{X\in \mathcal{X}}(X\cup \lbrace v_k'\rbrace)=B\cup \U(\mathcal{X})\cup \lbrace v_k'\rbrace.\]
Here we used that $\vert \mathcal{X}\vert=d\geq 1$ and that $B$ was defined as the set of non-isolated vertices in $A'$. Thus, we indeed have $m(\lbrace v_1,\dots,v_{k-1}, v_k'\rbrace)=m(A'\cup \lbrace v_k'\rbrace)=\vert B\cup \U(\mathcal{X})\vert +1$. Hence
\[\lambda(v_1,\dots,v_{k-1}, v_k')=\frac{1}{m(\lbrace v_1,\dots,v_{k-1}, v_k'\rbrace)}=\frac{1}{\vert B\cup \U(\mathcal{X})\vert +1}.\]

Plugging this into (\ref{eq-last-Q}), for every labeling $v_1,\dots,v_{k-1}$ of $A'$ we obtain
\begin{multline}\label{eq-Q-C}
\Lambda(v_1,\dots,v_{k-1})=\sum_{\substack{\mathcal{X}\su \PR(A')\\ \vert \mathcal{X}\vert=d}}\ \sum_{\substack{v_{k}'\in V(G)\sm A'\\ \mathcal{N}(v_k',A')=\mathcal{X}}}\frac{1}{\vert B\cup \U(\mathcal{X})\vert +1}\\
=\sum_{\substack{\mathcal{X}\su \PR(A')\\ \vert \mathcal{X}\vert=d}}c(\mathcal{X})\cdot n\cdot \frac{1}{\vert B\cup \U(\mathcal{X})\vert +1}=n\cdot \sum_{\substack{\mathcal{X}\su \PR(A')\\ \vert \mathcal{X}\vert=d}}\frac{c(\mathcal{X})}{\vert B\cup \U(\mathcal{X})\vert +1}=C\cdot n.
\end{multline}
Here, we set
\[C=\sum_{\substack{\mathcal{X}\su \PR(A')\\ \vert \mathcal{X}\vert=d}}\frac{c(\mathcal{X})}{\vert B\cup \U(\mathcal{X})\vert +1}.\]
In particular, we see that $\Lambda(v_1,\dots,v_{k-1})$ is independent of the labeling $v_1,\dots,v_{k-1}$ of $A'$. Note that $C>0$ as $\Lambda(v_1,\dots,v_{k-1})>0$.

Now, let us find an upper bound for the terms $\Lambda(v_1,\dots,v_{j})$ for $0\leq j\leq k-2$ in the denominator on the right-hand side of (\ref{eq-w}). Recall that for any labeling $v_1,\dots,v_{k-1}$ of $A'$ and any $0\leq j\leq k-2$, the quantity  $\Lambda(v_1,\dots,v_{j})$ counts the number of choices for $v_{j+1}'\in V(G)$ such that $v_1,\dots,v_j,v_{j+1}'$ is a good sequence. So we clearly have $\Lambda(v_1,\dots,v_{j})\leq n$.

Suppose that the labeling $v_1,\dots,v_{k-1}$ of $A'$ and $0\leq j\leq k-2$ are such that $B\su \lbrace v_1,\dots,v_j\rbrace$. Then $e(\lbrace v_1,\dots,v_j\rbrace)=e(B)=e(A')=\l-d$. For any choice of $v_{j+1}'\in V(G)$ such that $v_1,\dots,v_j,v_{j+1}'$ is a good sequence, we must have $e(\lbrace v_1,\dots,v_j, v_{j+1}'\rbrace)<\l$ and therefore
\[\vert \mathcal{N}(v_{j+1}',\lbrace v_1,\dots,v_j\rbrace)\vert=\deg_{\lbrace v_1,\dots,v_j,v_{j+1}'\rbrace}(v_{j+1}')=e(\lbrace v_1,\dots,v_j, v_{j+1}'\rbrace)-e(\lbrace v_1,\dots,v_j\rbrace)<d.\]
Recall that for any subfamily $\mathcal{X}\su \PR(\lbrace v_1,\dots,v_j\rbrace)$ with $\vert \mathcal{X}\vert=d$, there are precisely $c(\mathcal{X})\cdot n$ vertices $v\in V(G)\sm A'$ with $\mathcal{N}(v,A')=\mathcal{X}\su \PR(\lbrace v_1,\dots,v_j\rbrace)$. For all these vertices $v$ we have $\mathcal{N}(v, \lbrace v_1,\dots,v_j\rbrace)=\mathcal{N}(v,A')=\mathcal{X}$ and consequently $\vert \mathcal{N}(v, \lbrace v_1,\dots,v_j\rbrace)\vert=\vert \mathcal{X}\vert=d$. Thus, none of these $c(\mathcal{X})\cdot n$ vertices are eligible choices for $v_{j+1}'\in V(G)$ such that $v_1,\dots,v_j,v_{j+1}'$ is a good sequence. So for any subfamily $\mathcal{X}\su \PR(\lbrace v_1,\dots,v_j\rbrace)$ with $\vert \mathcal{X}\vert=d$, we find $c(\mathcal{X})\cdot n$ vertices in $v\in V(G)\sm A'$ that are ineligible to be chosen as $v_{j+1}'\in V(G)$. These $c(\mathcal{X})\cdot n$ vertices are disjoint for different $\mathcal{X}\su \PR(\lbrace v_1,\dots,v_j\rbrace)$, because they were given by the condition $\mathcal{N}(v,A')=\mathcal{X}$. Hence all in all we found at least
\[\sum_{\substack{\mathcal{X}\su \PR(\lbrace v_1,\dots,v_j\rbrace)\\\vert \mathcal{X}\vert=d}}c(\mathcal{X})\cdot n\]
vertices that cannot be chosen as $v_{j+1}'\in V(G)$ such that $v_1,\dots,v_j,v_{j+1}'$ is a good sequence. Hence we can conclude that the number $\Lambda(v_1,\dots,v_{j})$ of choices for $v_{j+1}'\in V(G)$ satisfies
\begin{multline*}
\Lambda(v_1,\dots,v_{j})\leq n-\sum_{\substack{\mathcal{X}\su \PR(\lbrace v_1,\dots,v_j\rbrace)\\\vert \mathcal{X}\vert=d}}c(\mathcal{X})\cdot n=n\cdot \left(1-\sum_{\substack{\mathcal{X}\su \PR(\lbrace v_1,\dots,v_j\rbrace)\\\vert \mathcal{X}\vert=d}}c(\mathcal{X})\right)\\
\leq n\cdot \exp\left(-\sum_{\substack{\mathcal{X}\su \PR(\lbrace v_1,\dots,v_j\rbrace)\\\vert \mathcal{X}\vert=d}}c(\mathcal{X})\right)
\end{multline*}
if $0\leq j\leq k-2$ and $B\su \lbrace v_1,\dots,v_j\rbrace$.

Thus, for every labeling $v_1,\dots,v_{k-1}$ of $A'$, we obtain
\begin{multline*}
\Lambda(\emptyset)\cdot \Lambda(v_1)\cdot \Lambda(v_1, v_2)\dotsm \Lambda(v_1,\dots,v_{k-2})\\
\leq \prod_{\substack{0\leq j\leq k-2\text{ s.t.}\\B\not\su \lbrace v_1,\dots,v_j\rbrace}}n\cdot \prod_{\substack{0\leq j\leq k-2\text{ s.t.}\\B\su \lbrace v_1,\dots,v_j\rbrace}}n\cdot \exp\left(-\sum_{\substack{\mathcal{X}\su \PR(\lbrace v_1,\dots,v_j\rbrace)\\\vert \mathcal{X}\vert=d}}c(\mathcal{X})\right)\\
=n^{k-1}\cdot \exp\left(-\sum_{\substack{0\leq j\leq k-2\text{ s.t.}\\B\su \lbrace v_1,\dots,v_j\rbrace}}\ \sum_{\substack{\mathcal{X}\su \PR(\lbrace v_1,\dots,v_j\rbrace)\\\vert \mathcal{X}\vert=d}}c(\mathcal{X})\right)\\
=n^{k-1}\cdot \exp\left(-\sum_{\substack{\mathcal{X}\su \PR(A')\\\vert \mathcal{X}\vert=d}}\ \sum_{\substack{0\leq j\leq k-2\text{ s.t.}\\B\cup \U(\mathcal{X})\su \lbrace v_1,\dots,v_j\rbrace}}c(\mathcal{X})\right).
\end{multline*}
Note that for any given labeling $v_1,\dots,v_{k-1}$ of $A'$ and any subfamily $\mathcal{X}\su \PR(A')$, each index $0\leq j\leq k-2$ satisfies $B\cup \U(\mathcal{X})\su \lbrace v_1,\dots,v_j\rbrace$ if and only if $v_{j+1}\in A'\sm (B\cup \U(\mathcal{X}))$ comes after all vertices in $B\cup \U(\mathcal{X})$ in the labeling $v_1,\dots,v_{k-1}$. Hence the number of indices $0\leq j\leq k-2$ with $B\cup \U(\mathcal{X})\su \lbrace v_1,\dots,v_j\rbrace$ is precisely the number of vertices $w\in A'\sm (B\cup \U(\mathcal{X}))$ that come after all vertices in $B\cup \U(\mathcal{X})$ in the labeling $v_1,\dots,v_{k-1}$. Thus, we can rewrite the previous inequality as
\[\Lambda(\emptyset)\cdot \Lambda(v_1)\cdot \Lambda(v_1, v_2)\dotsm \Lambda(v_1,\dots,v_{k-2})\leq n^{k-1}\cdot \exp\left(-\sum_{\substack{\mathcal{X}\su \PR(A')\\\vert \mathcal{X}\vert=d}}\ \sum_{\substack{w\in A'\sm(B\cup \U(\mathcal{X}))\text{ s.t. }w\\\text{after }B\cup \U(\mathcal{X})\text{ in }v_1,\dots,v_{k-1}}}c(\mathcal{X})\right)\]
for every labeling $v_1,\dots,v_{k-1}$ of $A'$.

Plugging this, together with (\ref{eq-Q-C}), into (\ref{eq-w}) yields
\begin{multline*}
\rho(v_1,\dots,v_k)=\frac{1}{m(A)}\cdot \frac{1}{\Lambda(\emptyset)\cdot \Lambda(v_1)\cdot \Lambda(v_1, v_2)\dotsm \Lambda(v_1,\dots,v_{k-2})}\cdot \frac{1}{\Lambda(v_1,\dots,v_{k-2})}\\
\geq \frac{1}{m(A)}\cdot \frac{1}{n^{k-1}}\cdot \exp\left(\sum_{\substack{\mathcal{X}\su \PR(A')\\\vert \mathcal{X}\vert=d}}\ \sum_{\substack{w\in A'\sm(B\cup \U(\mathcal{X}))\text{ s.t. }w\\\text{after }B\cup \U(\mathcal{X})\text{ in }v_1,\dots,v_{k-1}}}c(\mathcal{X})\right)\cdot \frac{1}{C\cdot n}\\
=\frac{1}{m(A)\cdot n^k}\cdot \frac{1}{C}\cdot \exp\left(\sum_{\substack{\mathcal{X}\su \PR(A')\\\vert \mathcal{X}\vert=d}}\ \sum_{\substack{w\in A'\sm(B\cup \U(\mathcal{X}))\text{ s.t. }w\\\text{after }B\cup \U(\mathcal{X})\text{ in }v_1,\dots,v_{k-1}}}c(\mathcal{X})\right)
\end{multline*}
for every labeling $v_1,\dots,v_{k-1}$ of $A'$.

Summing this up for all labelings $v_1,\dots,v_{k-1}$ of $A'$, we obtain
\begin{equation}\label{ineq1}
\sum_{\text{labelings }v_1,\dots, v_{k-1}\text{ of } A'}\rho(v_1,\dots,v_k)
\geq \frac{1}{m(A)\cdot n^k}\cdot \frac{1}{C}\cdot T,
\end{equation}
where we define
\[T:=\sum_{\substack{\text{labelings }\\v_1,\dots, v_{k-1}\text{ of } A'}}\exp\left(\sum_{\substack{\mathcal{X}\su \PR(A')\\\vert \mathcal{X}\vert=d}}\ \sum_{\substack{w\in A'\sm(B\cup \U(\mathcal{X}))\text{ s.t. }w\\\text{after }B\cup \U(\mathcal{X})\text{ in }v_1,\dots,v_{k-1}}}c(\mathcal{X})\right).\]

On the other hand, the inequality between arithmetic and geometric mean yields
\begin{multline*}
T\geq (k-1)!\cdot \exp\left(\frac{1}{(k-1)!}\sum_{\substack{\text{labelings }\\v_1,\dots, v_{k-1}\text{ of } A'}}\ \sum_{\substack{\mathcal{X}\su \PR(A')\\\vert \mathcal{X}\vert=d}}\ \sum_{\substack{w\in A'\sm(B\cup \U(\mathcal{X}))\text{ s.t. }w\\\text{after }B\cup \U(\mathcal{X})\text{ in }v_1,\dots,v_{k-1}}}c(\mathcal{X})\right)\\
=(k-1)!\cdot \exp\left(\frac{1}{(k-1)!}\sum_{\substack{\mathcal{X}\su \PR(A')\\\vert \mathcal{X}\vert=d}}\ \sum_{w\in A'\sm(B\cup \U(\mathcal{X}))}\ \sum_{\substack{\text{labelings }v_1,\dots, v_{k-1}\text{ of } A'\\\text{s.t. }w\text{ after }B\cup \U(\mathcal{X})}}c(\mathcal{X})\right).
\end{multline*}
We claim that for every subfamily $\mathcal{X}\su \PR(A')$ and every $w\in A'\sm(B\cup  \U(\mathcal{X}))$ the number of labelings $v_1,\dots,v_{k-1}$ of $A'$ such that $w$ comes after $B\cup  \U(\mathcal{X})$ is precisely $(k-1)!/(\vert B\cup  \U(\mathcal{X})\vert+1)$. Indeed, when picking a labeling $v_1,\dots,v_{k-1}$ of $A'$ uniformly at random, the probability that $w$ is the last element of the set $B\cup  \U(\mathcal{X})\cup \lbrace w\rbrace$ in that labeling is precisely
\[\frac{1}{\vert B\cup  \U(\mathcal{X})\cup \lbrace w\rbrace\vert}=\frac{1}{\vert B\cup  \U(\mathcal{X})\vert+1}.\]
Consequently, the number of labelings $v_1,\dots,v_{k-1}$ of $A'$ such that $w$ comes after $B\cup  \U(\mathcal{X})$ is indeed $(k-1)!/(\vert B\cup  \U(\mathcal{X})\vert+1)$. Thus, we obtain
\begin{multline*}
T \geq (k-1)!\cdot \exp\left(\frac{1}{(k-1)!}\sum_{\substack{\mathcal{X}\su \PR(A')\\\vert \mathcal{X}\vert=d}}\ \sum_{w\in A'\sm(B\cup \U(\mathcal{X}))}\frac{(k-1)!}{\vert B\cup  \U(\mathcal{X})\vert+1}\cdot c(\mathcal{X})\right)\\
=(k-1)!\cdot \exp\left(\sum_{\substack{\mathcal{X}\su \PR(A')\\\vert \mathcal{X}\vert=d}}\ \sum_{w\in A'\sm(B\cup \U(\mathcal{X}))}\frac{c(\mathcal{X})}{\vert B\cup \U(\mathcal{X})\vert+1}\right)\\
=(k-1)!\cdot \exp\left(\sum_{\substack{\mathcal{X}\su \PR(A')\\\vert \mathcal{X}\vert=d}}(k-1-\vert B\cup \U(\mathcal{X})\vert)\cdot \frac{c(\mathcal{X})}{\vert B\cup \U(\mathcal{X})\vert+1}\right).
\end{multline*}
Recall that $\vert B\vert\leq r\cdot (\l-d)$. Furthermore observe that for every $\mathcal{X}\su \PR(A')$ with $\vert \mathcal{X}\vert=d$, we have $\vert \U(\mathcal{X})\vert\leq (r-1)\cdot \vert \mathcal{X}\vert\leq (r-1)\cdot d$. Thus, for every $\mathcal{X}\su \PR(A')$ with $\vert \mathcal{X}\vert=d$, we have $\vert B\cup \U(\mathcal{X})\vert\leq r\cdot \l-d\leq r\cdot \l-1$ and consequently $k-1-\vert B\cup X\vert\geq k-r\l$. Hence
\begin{multline*}
T\geq (k-1)!\cdot \exp\left(\sum_{\substack{\mathcal{X}\su \PR(A')\\\vert \mathcal{X}\vert=d}}(k-r\l)\cdot \frac{c(\mathcal{X})}{\vert B\cup \U(\mathcal{X})\vert+1}\right)\\
=(k-1)!\cdot \exp\left((k-r\l)\cdot \sum_{\substack{\mathcal{X}\su \PR(A')\\\vert \mathcal{X}\vert=d}}\frac{c(\mathcal{X})}{\vert B\cup \U(\mathcal{X})\vert+1}\right)=(k-1)!\cdot \exp\left((k-r\l)\cdot C\right),
\end{multline*}
where the last step just used the definition of $C$. Plugging this into (\ref{ineq1}), we obtain
\begin{multline*}
\sum_{\text{labelings }v_1,\dots, v_{k-1}\text{ of } A'}\rho(v_1,\dots,v_k)
\geq \frac{1}{m(A)\cdot n^k}\cdot \frac{1}{C}\cdot (k-1)!\cdot \exp\left((k-r\l)\cdot C\right)\\
= \frac{(k-1)!}{m(A)\cdot n^k}\cdot \frac{\exp\left((k-r\l)\cdot C\right)}{C}
\end{multline*}

Now we use the following lemma, which is a straightforward calculation.

\begin{lemma}\label{lemma2}For all $x\in [0,\infty)$ we have
\[x\cdot \exp(-(k-r\l)\cdot x)\leq \frac{1}{k-r\l}\cdot \frac{1}{e}.\]
\end{lemma}
\begin{proof}
The function $\varphi(x)=x\cdot \exp(-(k-r\l)\cdot x)$ for $x\in [0,\infty)$ is non-negative, satisfies $\varphi(0)=0$ and tends to zero for $x\to \infty$ (as $k-r\l>0$). Note that $\varphi$ has derivative
\[\varphi'(x)=\exp(-(k-r\l)\cdot x)-(k-r\l)\cdot x\cdot \exp(-(k-r\l)\cdot x),\]
hence it has its unique maximum at $x=1/(k-r\l)$. So we indeed have \[x\cdot \exp(-(k-r\l)\cdot x)=\varphi(x)\leq \varphi\left(\frac{1}{k-r\l}\right)= \frac{1}{k-r\l}\cdot \frac{1}{e}\]
for all $x\in [0,\infty)$.
\end{proof}

Applying Lemma \ref{lemma2} to $x=C$, we obtain
\[\frac{\exp((k-r\l)\cdot C)}{C}=\frac{1}{C\cdot \exp(-(k-r\l)\cdot C)}\geq (k-r\l)\cdot e.\]
Thus,
\begin{multline*}
\sum_{\text{labelings }v_1,\dots, v_{k-1}\text{ of } A'}\rho(v_1,\dots,v_k)\geq \frac{(k-1)!}{m(A)\cdot n^k}\cdot \frac{\exp\left((k-r\l)\cdot C\right)}{C}\\
\geq \frac{(k-1)!}{m(A)\cdot n^k}\cdot (k-r\l)\cdot e=\frac{1}{m(A)} \cdot \frac{k-r\l}{k}\cdot e\cdot \frac{k!}{n^k},
\end{multline*}
which proves (\ref{inequality}). This finishes the proof of Theorem \ref{thm-hyper-1-e}.

\section{Proof of Theorems \ref{thm-graph-o1}, \ref{thm-hyper-o1} and \ref{thm-graph-forest}}\label{sect-proof-o1}

The proofs of Theorem \ref{thm-graph-o1}, \ref{thm-hyper-o1} and \ref{thm-graph-forest} are based on the following three lemmas.

\begin{lemma}\label{propo-1}Let $r$, $k$ and $\l$ be positive integers and let $G$ be a hypergraph on $n \geq k$ vertices all of whose edges have size at most $r$. Then for each real number $c$ with $0<c<\sqrt{k}/2$, the following holds: The number of $k$-vertex subsets $A\su V(G)$ with $e(A)=\l$ and $c\leq m(A)\leq \sqrt{k}/2$ is at most
\[32\cdot \sqrt{r}\cdot \frac{1}{\sqrt{c}}\cdot \frac{n^k}{k!}.\]
\end{lemma}

\begin{lemma}\label{propo-2}Let $r$, $k$ and $\l$ be positive integers  and let $G$ be a hypergraph on $n \geq k$ vertices all of whose edges have size at most $r$. Then for each real number $c$ satisfying $\sqrt{k}/2\leq c\leq k/(32r)$, the following holds: The number of $k$-vertex subsets $A\su V(G)$ with $e(A)=\l$ and $c\leq m(A)\leq 2c$ is at most
\[44\cdot \sqrt{r}\cdot k^{-1/4}\cdot \frac{n^k}{k!}.\]
\end{lemma}

\begin{lemma}\label{propo-3}Let $r$, $k$ and $\l$ be positive integers and let $G$ be a hypergraph on $n \geq k$ vertices all of whose edges have size at most $r$. Then for each real number $\eps$ with $0<\eps<1/2$, the following holds: The number of $k$-vertex subsets $A\su V(G)$ with $e(A)=\l$ and $\eps k \leq m(A)\leq (1-\eps) k$ is at most
\[8\cdot r^{1/4}\cdot \eps^{-1/2}\cdot k^{-1/4}\cdot \frac{n^k}{k!}.\]
\end{lemma}

We will prove these three lemmas in the following two sections.

Lemmas \ref{propo-1}, \ref{propo-2} and \ref{propo-3} imply the following corollary, from which we will deduce Theorems \ref{thm-graph-o1} and \ref{thm-hyper-o1}. Corollary \ref{coro-propo-o1} will also be used again in the proof of Theorem \ref{thm-graph-1-e} in Section \ref{sect-proof-graph-1-e}.

\begin{corollary}\label{coro-propo-o1}Let us fix a positive integer $r$ and a real number $\eps>0$. Assume that $k$ is an integer which is sufficiently large with respect to $\eps$. Let $\l$ be a positive integer and $G$ a hypergraph on $n \geq k$ vertices all of whose edges have size at most $r$. Then for each real number $c'>0$, there are at most
\[\left(32\cdot \sqrt{r}\cdot \frac{1}{\sqrt{c'}}+23 \cdot \sqrt{r}\cdot k^{-1/4}\cdot \log k\right)\cdot \frac{n^k}{k!} \]
$k$-vertex subsets $A\su V(G)$ with $e(A)=\l$ and $c' \leq m(A)\leq (1-\eps) k$.
\end{corollary}
\begin{proof}We may assume without loss of generality that $\eps<1/(32r)<1/2$, because otherwise we can just make $\eps$ smaller.

Let us divide the $k$-vertex subsets $A\su V(G)$ with $e(A)=\l$ and $c' \leq m(A)\leq (1-\eps) k$ into three groups. The first group consists of those subsets $A$ with $c' \leq m(A)< \sqrt{k}/2$, the second group of those with $\sqrt{k}/2\leq m(A)< \eps k$ and finally the third group of those with $\eps k\leq m(A)\leq (1-\eps) k$. We will bound the number of $k$-vertex subsets in each of these three groups separately, in each case using one of the three lemmas above.

For the first group, we claim that the number of $k$-vertex subsets $A\su V(G)$ with $e(A)=\l$ and $c'\leq m(A)<\sqrt{k}/2$ is at most
\[32\cdot \sqrt{r}\cdot \frac{1}{\sqrt{c'}}\cdot \frac{n^k}{k!}.\]
If $c'\geq \sqrt{k}/2$, then there cannot be any such $k$-vertex subset $A\su V(G)$, so this claim is trivially true. If $c'< \sqrt{k}/2$, then the claim follows directly from Lemma \ref{propo-1}.

For the second group, we claim that the number of $k$-vertex subsets $A\su V(G)$ with $e(A)=\l$ and $\sqrt{k}/2\leq m(A)< \eps k$ is at most
\[22 \cdot \sqrt{r}\cdot k^{-1/4}\cdot \log k\cdot \frac{n^k}{k!}.\]
As $k$ is sufficiently large with respect to $\eps$, we may assume that $\eps\sqrt{k}\geq 2$. Let us apply Lemma \ref{propo-2} several times with different values of $c$. More precisely, let us take
\[c_j=\frac{\sqrt{k}}{2}\cdot 2^j,\]
for $j=0,1,\dots,\lceil \log (\eps \sqrt{k})\rceil$. For each of these $j$, we have $c_j\geq \sqrt{k}/2$ and also
\[c_j\leq \frac{\sqrt{k}}{2}\cdot 2^{\lceil \log (\eps \sqrt{k})\rceil}\leq \frac{\sqrt{k}}{2}\cdot 2^{\log (\eps \sqrt{k})+1}=\sqrt{k}\cdot (\eps \sqrt{k})=\eps k<\frac{k}{32r},\]
where we used the assumption $\eps<1/(32r)$. Thus, all these values $c_j$  satisfy the assumptions of Lemma \ref{propo-2}.  Thus, by Lemma \ref{propo-2}, for each $j=0,1,\dots,\lceil \log (\eps \sqrt{k})\rceil$ the number of $k$-vertex subsets $A\su V(G)$ with $e(A)=\l$ and $(\sqrt{k}/2)\cdot 2^j=c_j\leq m(A)\leq 2c_j=(\sqrt{k}/2)\cdot 2^{j+1}$ is at most
\[44\cdot \sqrt{r}\cdot k^{-1/4}\cdot \frac{n^k}{k!}.\]
Adding this up for all $j=0,1,\dots,\lceil \log (\eps \sqrt{k})\rceil$, we see that the total number of $k$-vertex subsets $A\su V(G)$ with $e(A)=\l$ and $\sqrt{k}/2\leq m(A)< (\sqrt{k}/2)\cdot 2^{\log (\eps \sqrt{k})+1}=\eps k$ is at most
\[\left(\lceil\log (\eps \sqrt{k})\rceil+1\right)\cdot 44\cdot \sqrt{r}\cdot k^{-1/4}\cdot \frac{n^k}{k!}< \frac{1}{2}\log k\cdot 44\cdot \sqrt{r}\cdot k^{-1/4}\cdot \frac{n^k}{k!}=22 \cdot \sqrt{r}\cdot k^{-1/4}\cdot \log k\cdot \frac{n^k}{k!},\]
where we used that $\lceil\log (\eps \sqrt{k})\rceil+1< \log (\eps \sqrt{k})+2\leq \log (\sqrt{k}/32)+2<\log \sqrt{k}=\frac{1}{2}\log k$. This indeed proves our claim about the second group.

Finally, for the third group, Lemma \ref{propo-3} implies that the number of $k$-vertex subsets $A\su V(G)$ with $e(A)=\l$ and $\eps k\leq m(A)\leq (1-\eps)k$ is at most
\[8\cdot r^{1/4}\cdot \eps^{-1/2}\cdot k^{-1/4}\cdot \frac{n^k}{k!}\leq \sqrt{r}\cdot k^{-1/4}\cdot \log k\cdot \frac{n^k}{k!}.\]
Here we used that $k$ is sufficiently large with respect to $\eps$, so that $\log k\geq 8\cdot \eps^{-1/2}$.

All in all, adding up our bounds for each of the three groups, we obtain that the total number of $k$-vertex subsets $A\su V(G)$ with $e(A)=\l$ and $c' \leq m(A)\leq (1-\eps) k$ is at most
\[\left(32\cdot \sqrt{r}\cdot \frac{1}{\sqrt{c'}}+23 \cdot \sqrt{r}\cdot k^{-1/4}\cdot \log k\right)\cdot \frac{n^k}{k!},\]
as desired.
\end{proof}

In order to deduce Theorem \ref{thm-graph-o1} from Corollary \ref{coro-propo-o1}, we will use the following claim.

\begin{claim}\label{claim-mA-bounds-graph} For every graph $G$ and every subset $A\su V(G)$ of its vertex set, we have
\[\sqrt{2e(A)}\leq m(A)\leq 2e(A).\]
\end{claim}
\begin{proof}The inequality $m(A)\leq 2e(A)$ follows directly from the definition of $m(A)$, because every vertex that is non-isolated in $A$ needs to be incident with at least one of the $e(A)$ edges within $A$.

For the other inequality, note that the $e(A)$ edges inside $A$ are all between the $m(A)$ non-isolated vertices. Therefore we must have
\[e(A)\leq {m(A) \choose 2}\leq \frac{m(A)^2}{2},\]
which implies $m(A)\geq \sqrt{2e(A)}$.
\end{proof}

Now we are ready to prove Theorem \ref{thm-graph-o1}

\begin{proof}[Proof of Theorem  \ref{thm-graph-o1}]Let $\eps>0$ be a fixed real number and let $k$ be sufficiently large (with respect to $\eps$). Furthermore, let $\l$ be an integer with $1 \leq \l\leq (1-\eps)k/2$.

Now, consider any graph $G$ on $n\geq k$ vertices. By Claim \ref{claim-mA-bounds-graph}, for any subset $A\su V(G)$ with $e(A)=\l$ we have $\sqrt{2\l}\leq m(A)\leq 2\l\leq  (1-\eps)k$. Thus, by Corollary \ref{coro-propo-o1} applied to $r=2$ and $c'=\sqrt{2l}$ the number of subsets $A\su V(G)$ with $e(A)=\l$ is at most
\[\left(32\cdot \sqrt{2}\cdot (2\l)^{-1/4}+23 \cdot \sqrt{2}\cdot k^{-1/4}\cdot \log k\right)\cdot \frac{n^k}{k!}\leq \left(48\cdot \l^{-1/4}+36 \cdot k^{-1/4}\cdot \log k\right)\cdot (1+o_n(1))\cdot {n\choose k}.\]
Here we used $\sqrt{2}\leq 3/2$. Thus, we can conclude that
\[\ind(k,\l)\leq 48\cdot \l^{-1/4}+36 \cdot k^{-1/4}\cdot \log k.\]

If $1\leq \l\leq k/\log^4 k$, then $\l^{-1/4}\geq k^{-1/4}\cdot \log k$ and so we can conclude
\[\ind(k,\l)\leq 48\cdot \l^{-1/4}+36 \cdot k^{-1/4}\cdot \log k\leq 90\cdot \l^{-1/4}.\]
On the other hand, if $k/\log^4 k \leq \l\leq (1-\eps)k/2$, then $\l^{-1/4}\leq k^{-1/4}\cdot \log k$ and we obtain
\[\ind(k,\l)\leq 48\cdot \l^{-1/4}+36 \cdot k^{-1/4}\cdot \log k\leq 90\cdot k^{-1/4}\cdot \log k.\] 
This finishes the proof of Theorem \ref{thm-graph-o1}.
\end{proof}

Now, let us turn to proving Theorem \ref{thm-hyper-o1}. Here, we will use the following claim.

\begin{claim}\label{claim-mA-bounds-hyper} Let $G$ be a hypergraph all of whose edges have size at most $r$. Then for every subset $A\su V(G)$ of its vertex set with $e(A)\geq 2^{2r}$, we have
\[\frac{1}{6}\cdot r\cdot e(A)^{1/r}\leq m(A)\leq r\cdot e(A).\]
\end{claim}
\begin{proof}Again, the inequality $m(A)\leq r\cdot e(A)$ follows directly from the definition of $m(A)$, because every vertex that is non-isolated in $A$ needs to be part of at least one of the $e(A)$ edges within $A$.

For the other inequality, note that the $e(A)$ edges inside $A$ are subsets of the set consisting of the $m(A)$ non-isolated vertices in $A$. Since all edges have size at least 1 and at most $r$, we obtain
\[e(A)\leq {m(A)\choose r}+{m(A)\choose r-1}+\dots+{m(A)\choose 1}.\]
If $m(A)<2r$, this yields $e(A)\leq 2^{m(A)}<2^{2r}$, a contradiction. Hence we must have $m(A)\geq 2r$. Then ${m(A)\choose r}$ is the largest of the binomial coefficients in the inequality above and we can conclude
\[e(A)\leq r\cdot {m(A)\choose r}\leq 2^r\cdot {m(A)\choose r}\leq 2^r\cdot\left(\frac{e\cdot m(A)}{r}\right)^{r}=\left(\frac{2e\cdot m(A)}{r}\right)^{r}\leq \left(\frac{6\cdot m(A)}{r}\right)^{r}.\]
Thus, $m(A)\geq  r\cdot e(A)^{1/r}/6$, as desired.
\end{proof}

Now it is easy to deduce Theorem \ref{thm-hyper-o1} from Corollary \ref{coro-propo-o1}.

\begin{proof}[Proof of Theorem  \ref{thm-hyper-o1}]Let $r\geq 3$ and $\eps>0$ be fixed and let $k$ be sufficiently large (with respect to $r$ and $\eps$). Furthermore, let $\l$ be an integer with $1 \leq \l\leq (1-\eps)k/r$. 

Note that we certainly have $\ind_{\leq r}(k,\l)\leq 1$, so the bound on $\ind_{\leq r}(k,\l)\leq 1$ claimed in Theorem  \ref{thm-hyper-o1} is trivial if $\l\leq 100^{2r}$. So let us from now on assume that $\l> 100^{2r}$

Consider any hypergraph $G$ on $n\geq k$ vertices all of whose edges have size at most $r$. By Claim \ref{claim-mA-bounds-hyper}, for any subset $A\su V(G)$ with $e(A)=\l$ we have $\frac{1}{6}\cdot r\cdot \l^{1/r}\leq m(A)\leq r\cdot \l\leq  (1-\eps)k$. Thus, by Corollary \ref{coro-propo-o1} applied to $c'=\frac{1}{6}\cdot r\cdot \l^{1/r}$ the number of subsets $A\su V(G)$ with $e(A)=\l$ is at most
\begin{multline*}\left(32\cdot \sqrt{r}\cdot \left(\frac{1}{6}\cdot r\cdot \l^{1/r}\right)^{-1/2}+23 \cdot \sqrt{r}\cdot k^{-1/4}\cdot \log k\right)\cdot \frac{n^k}{k!}\\
\leq \left(96\cdot \l^{-1/(2r)}+23\cdot \sqrt{r} \cdot k^{-1/4}\cdot \log k\right)\cdot \frac{n^k}{k!}
\leq 97\cdot \l^{-1/(2r)}\cdot (1+o_n(1))\cdot {n\choose k}.
\end{multline*}
Here we used $\sqrt{6}\leq 3$ and $23\cdot \sqrt{r} \cdot k^{-1/4}\cdot \log k\leq k^{-1/(2r)}\leq \l^{-1/(2r)}$ for $k$ sufficiently large (recall that $r\geq 3$). Thus, we can conclude that
\[\ind_{r}(k,\l)\leq \ind_{\leq r}(k,\l)\leq 97\cdot \l^{-1/(2r)}\leq 100\cdot \l^{-1/(2r)},\]
as desired.
\end{proof}

Finally, Theorem \ref{thm-graph-forest} is an easy consequence of Lemma \ref{propo-1}.

\begin{proof}[Proof of Theorem \ref{thm-graph-forest}]As in the theorem statement, let $k$ and $\l$ be positive integers with $\l\leq \sqrt{k}/4$ and let $G$ be a graph on $n \geq k$ vertices. Note that every $k$-vertex subset $A\su V(G)$ that induces a forest with exactly $\l$ edges satisfies $\l\leq \l+1\leq m(A)\leq 2\l\leq \sqrt{k}/2$. Hence, by Lemma \ref{propo-1} applied to $r=2$ and $c=\l$, the number of $k$-vertex subsets $A\su V(G)$ that induce a forest with exactly $\l$ edges is at most
\[32\cdot \sqrt{2}\cdot \l^{-1/2}\cdot \frac{n^k}{k!}\leq 50\cdot \l^{-1/2}\cdot \frac{n^k}{k!},\]
where we used that $\sqrt{2}\leq 3/2$.\end{proof}

The next two sections will be devoted to proving Lemmas \ref{propo-1}, \ref{propo-2} and \ref{propo-3}. All three of these proofs follow the same general strategy. We will first prove Lemma \ref{propo-3} in Section \ref{sect-proof-propo-3}, because the proof is slightly easier than the proofs of Lemmas \ref{propo-1} and \ref{propo-2}. Afterwards, we will proof Lemmas \ref{propo-1} and \ref{propo-2} in Section \ref{sect-proof-propo-1-2}. Let us now introduce some notation that will be relevant for the proofs of all three of the lemmas: 

For a hypergraph $G$ and a subset $B\su V(G)$, call a vertex $v\in V(G)\sm B$ \emph{connected to $B$} if there is an edge $e\in E(G)$ with $v\in e$ and $e\sm \lbrace v\rbrace\su B$. Note that in the case where the singleton set $\lbrace v\rbrace$ is an edge of $G$, the vertex $v$ is connected to every subset $B\su V(G)\sm \lbrace v\rbrace$. Also note that in the case where $G$ is a graph, $v$ being connected to $B$ simply means that $v$ has an edge to a vertex in $B$.

For a pair of subsets $B\su A\su V(G)$, let $h(A,B)$ denote the number of vertices $v\in A\sm B$ that are connected to $B$. Furthermore, let $m(A,B)$ denote the number of vertices $v\in A\sm B$ that are non-isolated in $A$. Clearly, $m(A,B)\leq m(A)$. Also note that every vertex $v\in A\sm B$ that is connected to $B$ is automatically non-isolated in $A$. Hence $h(A,B)\leq m(A,B)$. Finally, let $f(A,B)=m(A,B)-h(A,B)$ be the number of vertices $v\in A\sm B$ that are non-isolated in $A$ but not connected to $B$.

\section{Proof of Lemma \ref{propo-3}}\label{sect-proof-propo-3}

In this section, we will prove Lemma \ref{propo-3}. So let us fix positive integers $r$, $k$ and $\l$ and let $G$ be a hypergraph on $n \geq k$ vertices all of whose edges have size at most $r$.

The following definition introduces the key concept for our proof. Recall that we defined the quantities $h(A,B)$, $m(A,B)$ and $f(A,B)$ at the end of Section \ref{sect-proof-o1}.

\begin{definition}\label{defi-pleasant}For a real number $0<\eps<1/2$, let a pair $(A,B)$ of subsets $B\su A\su V(G)$ be called \emph{$\eps$-pleasant} if all of the following conditions are satisfied:
\begin{itemize}
\item[(i)] $\vert A\vert=k$ and $e(A)=\l$.
\item[(ii)] $f(A,B)=0$, meaning that every vertex in $A\sm B$ is either connected to $B$ or isolated in $A$.
\item[(iii)] $h(A,B)\geq \frac{1}{4}r^{-1/2}\cdot \eps\cdot \sqrt{k}$.
\item[(iv)] $\vert A\sm B\vert-h(A,B)\geq \frac{1}{4}r^{-1/2}\cdot \eps\cdot \sqrt{k}$.
\end{itemize}
\end{definition}

The crucial insight for the proof of Lemma \ref{propo-3} is that for each fixed subset $B\su V(G)$ we can bound the number of sets $A\su V(G)$ such that $(A,B)$ is $\eps$-pleasant.

\begin{lemma}\label{lemma-B-pleasant-bound}Let $0<\eps<1/2$ be a real number and let $B\su V(G)$ be a subset of size $\vert B\vert\leq k$. Let $a=k-\vert B\vert$. Then there are at most 
\[2\cdot r^{1/4}\cdot \eps^{-1/2}\cdot k^{-1/4}\cdot \frac{n^a}{a!}\]
different sets $A\su V(G)$ such that $(A,B)$ is an $\eps$-pleasant pair.
\end{lemma}

We will prove Lemma \ref{lemma-B-pleasant-bound} in the second subsection. First, we will prove Lemma \ref{propo-3} assuming Lemma \ref{lemma-B-pleasant-bound} in the first subsection.

The idea of the proof of Lemma \ref{propo-3} is to show that for every $k$-vertex subset $A\su V(G)$ with the properties specified in the lemma, we have a sufficiently high probability of generating an $\eps$-pleasant pair $(A,B)$ when choosing a random subset $B\su A$ in an appropriate way. Combining this with Lemma \ref{lemma-B-pleasant-bound} will give an upper bond for the number of such sets $A$.

\subsection{Main proof}

In this subsection we will prove Lemma \ref{propo-3} assuming Lemma \ref{lemma-B-pleasant-bound}. As in the statement of Lemma \ref{propo-3}, let $\eps$ be a real number with $0<\eps<1/2$.

Note that if $\eps^2\cdot k\leq 2^{12} \cdot r$, then
\[8\cdot r^{1/4}\cdot \eps^{-1/2}\cdot k^{-1/4}\cdot \frac{n^k}{k!}\geq \frac{n^k}{k!}\geq {n \choose k}.\]
So in this case Lemma \ref{propo-3} is trivially true as the total number of $k$-vertex subsets of $V(G)$ is only ${n\choose k}$. Thus, we can assume that $\eps^2\cdot k\geq 2^{12} \cdot r$.

As indicated at the end of the previous subsection, we start by proving that for each fixed $k$-vertex subset $A\su V(G)$ with the properties in Lemma \ref{propo-3}, we have a probability of at least $\frac{1}{2}$ to obtain an $\eps$-pleasant pair $(A,B)$ when choosing a random subset $B\su A$ by picking each element of $A$ independently with some appropriately chosen probability $p$.

\begin{lemma}\label{lemma-A-3}Let $p=1-\frac{1}{2}r^{-1/2}k^{-1/2}$. Then for every $k$-vertex subset $A\su V(G)$ with $e(A)=\l$ and $\eps k \leq m(A)\leq (1-\eps) k$, the following holds: If we choose a subset $B\su A$ randomly by taking each element of $A$ into $B$ with probability $p$ (independently for all elements of $A$), then with probability at least $\frac{1}{2}$ the resulting pair $(A,B)$ is $\eps$-pleasant.
\end{lemma}

\begin{proof}Fix a subset $A\su V(G)$ with $\vert A\vert=k$ as well as $e(A)=\l$ and $\eps k \leq m(A)\leq (1-\eps) k$. As in the lemma, let the subset $B\su A$ be chosen by randomly picking elements of $A$ independently with probability $p$. Let $q=1-p=\frac{1}{2}r^{-1/2}k^{-1/2}$, then each element of $A$ will be in $A\sm B$ with probability $q$. We need to prove that the pair $(A,B)$ satisfies the conditions in Definition \ref{defi-pleasant} with probability at least $\frac{1}{2}$.

First, note that condition (i) in Definition \ref{defi-pleasant} is automatically satisfied by the assumptions on $A$.

Next, we claim that condition (ii) holds with probability at least $\frac{3}{4}$. For condition (ii) to fail, there needs to be a vertex $v\in A\sm B$ such that $v$ is non-isolated in $A$ but not connected to $B$. Each vertex $v\in A$ that is non-isolated in $A$ is, by definition, contained in at least one edge $e\su A$. For each fixed such $v$, the probability for $v\in A\sm B$ is precisely $q$. Let us pick one edge $e\su A$ with $v\in e$. For $v$ to cause condition (ii) to fail, furthermore at least one vertex in $e\sm \lbrace v\rbrace$ needs to be in $A\sm B$ as well (otherwise $v$ is connected to $B$). As $\vert e\sm \lbrace v\rbrace\vert\leq r-1$, the probability of this happening is at most $(r-1)\cdot q$ (and this is independent of whether $v\in A\sm B$). Thus, the total probability of $v$ causing condition (ii) to fail is at most $q\cdot (r-1)\cdot q=(r-1)\cdot q^2$. Summing this over all the at most $k$ choices of a non-isolated vertex $v\in A$, we see that the total probability of condition (ii) failing is at most
\[k\cdot (r-1)\cdot q^2\leq k\cdot r\cdot q^2=k\cdot r\cdot \frac{1}{4}\cdot \frac{1}{r\cdot k}=\frac{1}{4}.\]
Thus, condition (ii) indeed holds with probability at least $\frac{3}{4}$.

Note that if condition (ii) holds, then we have $h(A,B)=m(A,B)-f(A,B)=m(A,B)$. Therefore, in order to check conditions (iii) and (iv) we can consider $m(A,B)$ instead of $h(A,B)$.

We claim that with probability at least $\frac{7}{8}$ we have
\begin{equation}\label{ineq-mAB-low-3}
m(A,B)\geq \frac{1}{2}\cdot q \cdot m(A).
\end{equation}
Recall that $m(A,B)$ counts the number of vertices in $A\sm B$ that are non-isolated in $A$. There are $m(A)$ vertices of $A$ that are non-isolated in $A$ and each of them ends up in $A\sm B$ with probability $q$ (and all independently). Thus, $m(A,B)$ is binomially distributed with parameters $m(A)$ and $q$. So using the standard Chernoff bound for the lower tail of binomially distributed variables (see, for example, \cite[Theorem A.1.13]{alon-spencer}), we have 
\begin{multline*}
\mathbb{P}\left[m(A,B)< \frac{1}{2}\cdot q \cdot m(A)\right]<\exp\left(-\frac{\left(\frac{1}{2}\cdot q \cdot m(A)\right)^2}{2\cdot q\cdot m(A)}\right)=\exp\left(-\frac{q}{8}\cdot m(A)\right)\leq \exp\left(-\frac{q}{8}\cdot \eps\cdot k\right)\\
=\exp\left(-\frac{1}{16}\cdot \frac{1}{\sqrt{r}\cdot \sqrt{k}}\cdot \eps\cdot k\right)=\exp\left(-\frac{1}{16}\cdot \frac{\eps\cdot \sqrt{k}}{\sqrt{r}}\right)\leq \exp\left(-\frac{1}{16}\cdot 2^6\right)= e^{-4}\leq \frac{1}{8}
\end{multline*}
where we used $m(A)\geq \eps k$ and $\eps^2\cdot k\geq 2^{12} \cdot r$. So (\ref{ineq-mAB-low-3}) indeed holds with probability at least $\frac{7}{8}$.

Similarly, we claim that with probability at least $\frac{7}{8}$ we have
\begin{equation}\label{ineq-a-mAB-low-3}
\vert A\sm B\vert-m(A,B)\geq \frac{1}{2}\cdot q \cdot (k-m(A)).
\end{equation}
Note that $\vert A\sm B\vert-m(A,B)$ counts the number of vertices in $A\sm B$ that are isolated in $A$. There are $k-m(A)$ vertices of $A$ that are isolated in $A$ and each of them ends up in $A\sm B$ with probability $q$ (and all independently). Thus, $\vert A\sm B\vert-m(A,B)$ is binomially distributed with parameters $k-m(A)$ and $q$. So using the same computation as for inequality (\ref{ineq-mAB-low-3}), we see that
\begin{multline*}
\mathbb{P}\left[\vert A\sm B\vert-m(A,B)< \frac{1}{2}\cdot q \cdot (k-m(A))\right]<\exp\left(-\frac{\left(\frac{1}{2}\cdot q \cdot (k-m(A))\right)^2}{2\cdot q\cdot (k-m(A))}\right)\\
=\exp\left(-\frac{q}{8}\cdot (k-m(A))\right)\leq \exp\left(-\frac{q}{8}\cdot \eps\cdot k\right)\leq \frac{1}{8}
\end{multline*}
where this time we used $m(A)\leq (1-\eps) k$. So (\ref{ineq-a-mAB-low-3}) indeed holds with probability at least $\frac{7}{8}$.

All in all, condition (ii) fails with probability at most $\frac{1}{4}$, and each of (\ref{ineq-mAB-low-3}) and (\ref{ineq-a-mAB-low-3}) fails with probability at most $\frac{1}{8}$. Thus, with probability at least $\frac{1}{2}$ condition (ii) and both of the inequalities (\ref{ineq-mAB-low-3}) and (\ref{ineq-a-mAB-low-3}) hold. We claim that in this case the pair $(A,B)$ is $\eps$-pleasant. Indeed, we already saw that condition (i) is always satisfied by the assumptions on $A$. Furthermore, condition (ii) implies $h(A,B)=m(A,B)-f(A,B)=m(A,B)$. Thus, (\ref{ineq-mAB-low-3}) gives (using $m(A)\geq \eps k$)
\[h(A,B)=m(A,B)\geq \frac{1}{2}\cdot q \cdot m(A)\geq \frac{1}{2}\cdot q\cdot \eps\cdot k=\frac{1}{4}\cdot \frac{1}{\sqrt{r}\cdot \sqrt{k}}\cdot \eps\cdot k=\frac{1}{4}r^{-1/2}\cdot \eps\cdot \sqrt{k}.\]
So condition (iii) holds. Finally, condition (iv) is satisfied, because (\ref{ineq-a-mAB-low-3}) yields (using $m(A)\leq (1-\eps) k$ this time)
\[\vert A\sm B\vert- h(A,B)=\vert A\sm B\vert - m(A,B)\geq \frac{1}{2}\cdot q \cdot (k-m(A))\geq \frac{1}{2}\cdot q\cdot \eps\cdot k=\frac{1}{4}r^{-1/2}\cdot \eps\cdot \sqrt{k}.\]

Thus, the pair $(A,B)$ is indeed $\eps$-pleasant if condition (ii) as well as the inequalities (\ref{ineq-mAB-low-3}) and (\ref{ineq-a-mAB-low-3}) hold. Hence we have shown that $(A,B)$ is an $\eps$-pleasant pair with probability at least $\frac{1}{2}$.
\end{proof}

Let us now prove Lemma \ref{propo-3}. So suppose for contradiction that there are more than
\[8\cdot r^{1/4}\cdot \eps^{-1/2}\cdot k^{-1/4}\cdot \frac{n^k}{k!}.\]
$k$-vertex subsets $A\su V(G)$ with $e(A)=\l$ and $\eps k \leq m(A)\leq (1-\eps) k$. 

Define $p=1-\frac{1}{2}r^{-1/2}k^{-1/2}$ as in Lemma \ref{lemma-A-3}. Now, let us choose a pair $(A,B)$ of subsets $B\su A\su V(G)$ with $\vert A\vert=k$ randomly as follows: First, choose $A\su V(G)$ uniformly at random among all subsets of size $\vert A\vert=k$. Then, choose a subset $B\su A$ by taking each element of $A$ into $B$ with probability $p$ (independently for all elements of $A$).

By our assumption above, the probability that $A$ satisfies $e(A)=\l$ and $\eps k \leq m(A)\leq (1-\eps) k$ is at least
\[8\cdot r^{1/4}\cdot \eps^{-1/2}\cdot k^{-1/4}\cdot \frac{n^k}{k!}\cdot {n\choose k}^{-1}=8\cdot r^{1/4}\cdot \eps^{-1/2}\cdot k^{-1/4}\cdot \frac{n^k}{(n)_k}.\]
If this is the case, then by Lemma \ref{lemma-A-3} the pair $(A,B)$ is $\eps$-pleasant with probability at least $\frac{1}{2}$. Thus, for the total probability that $(A,B)$ is an $\eps$-pleasant pair, we obtain
\begin{equation}\label{ineq-ABpleasantup3}
\mathbb{P}[(A,B)\text{ is }\eps\text{-pleasant}]\geq 
4\cdot r^{1/4}\cdot \eps^{-1/2}\cdot k^{-1/4}\cdot\frac{n^k}{(n)_k}.
\end{equation}

On the other hand, we can find an upper bound for the probability that $(A,B)$ is $\eps$-pleasant by conditioning on $B$. Note that $B$ will always be a subset $B\su V(G)$ with $\vert B\vert \leq k$. So let us fix any subset $B\su V(G)$ with $\vert B\vert \leq k$ and let $a=k-\vert B\vert$. If we condition on having this particular set $B$ as the result of our random draw of the pair $(A,B)$, then there are
\[{n-\vert B\vert\choose k-\vert B\vert}={n-\vert B\vert\choose a}=\frac{(n-\vert B\vert)_a}{a!}\]
possibilities for the set $A$ and they are all equally likely (recall that we choose $A$ uniformly at random among all subsets $A\su V(G)$ with $\vert A\vert=k$). But by Lemma \ref{lemma-B-pleasant-bound} at most
\[2\cdot r^{1/4}\cdot \eps^{-1/2}\cdot k^{-1/4}\cdot \frac{n^a}{a!}\]
of these possible sets $A$ have the property that $(A,B)$ is $\eps$-pleasant. Thus, when we condition on having this set $B$, the probability for $(A,B)$ to be $\eps$-pleasant  is at most
\[2\cdot r^{1/4}\cdot \eps^{-1/2}\cdot k^{-1/4}\cdot \frac{n^a}{(n-\vert B\vert)_a}=
2\cdot r^{1/4}\cdot \eps^{-1/2}\cdot k^{-1/4}\cdot \frac{n^{k-\vert B\vert}}{(n-\vert B\vert)_{k-\vert B\vert}}
\leq 2\cdot r^{1/4}\cdot \eps^{-1/2}\cdot k^{-1/4}\cdot \frac{n^k}{(n)_k}.\]
Since the set $B$ was arbitrary, this shows that the total probability for $(A,B)$ to be an $\eps$-pleasant pair satisfies
\[\mathbb{P}[(A,B)\text{ is }\eps\text{-pleasant}]\leq 2\cdot r^{1/4}\cdot \eps^{-1/2}\cdot k^{-1/4}\cdot \frac{n^k}{(n)_k}.\]
But now we have a contradiction to (\ref{ineq-ABpleasantup3}). This finishes the proof of Lemma \ref{propo-3}.

\subsection{Proof of Lemma \ref{lemma-B-pleasant-bound}}

Now, let us finally prove Lemma \ref{lemma-B-pleasant-bound}. As in the statement of the lemma, fix $0<\eps<1/2$ as well as a subset $B\su V(G)$ of size $\vert B\vert\leq k$, and let $a=k-\vert B\vert$.

Note that for every $\eps$-pleasant pair $(A,B)$ we have $\vert A\vert=k$ by condition (i) in Definition \ref{defi-pleasant} and therefore $\vert A\sm B\vert=k-\vert B\vert=a$.

The following definition introduces a central notion for our proof of Lemma \ref{lemma-B-pleasant-bound}.

\begin{definition}\label{defi-tidy-sequence} Let us call a sequence $v_1,\dots,v_a$ of distinct vertices in $V(G)\sm B$ \emph{tidy} if $e(B\cup \lbrace v_1,\dots,v_a\rbrace)=\l$ and if there exists an index $h\in \lbrace 1,\dots,a-1\rbrace$ with
\[\frac{\eps\cdot \sqrt{k}}{4\sqrt{r}}\leq h\leq a-\frac{\eps\cdot \sqrt{k}}{4\sqrt{r}}\]
and such that
\begin{itemize}
\item each of the vertices $v_1,\dots,v_{h}$ is connected to $B$, and
\item each of the vertices $v_{h+1},\dots,v_{a}$ is isolated in $B\cup \lbrace v_1,\dots,v_a\rbrace$.
\end{itemize}
\end{definition}

Note that for a sequence $v_1,\dots,v_a$ as in Definition \ref{defi-tidy-sequence}, the vertices $v_{h+1},\dots,v_a$ cannot be connected to $B$ (since otherwise they would be non-isolated in $B\cup \lbrace v_1,\dots,v_a\rbrace$). So the number $h(B\cup \lbrace v_1,\dots,v_a\rbrace, B)$ of vertices in $\lbrace v_1,\dots,v_a\rbrace$ that are connected to $B$ is precisely $h$. Thus, for every tidy sequence $v_1,\dots,v_a$, the index $h$ in Definition \ref{defi-tidy-sequence} must be equal to $h(B\cup \lbrace v_1,\dots,v_a\rbrace, B)$.

An important observation is that for every set $A\su V(G)$ such that $(A,B)$ is $\eps$-pleasant, we have $\vert A\sm B\vert = a$ and we can label the vertices of $A\sm B$ as $v_1,\dots,v_a$ in such a way that $v_1,\dots,v_a$ is a tidy sequence. In fact, there are many ways to label the vertices of $A\sm B$ as $v_1,\dots,v_a$ so that $v_1,\dots,v_a$ is a tidy sequence. The following claim makes this precise.

\begin{claim}\label{claim-count-labelings-tidy} Let $A\su V(G)$ be such that $(A,B)$ is an $\eps$-pleasant pair. Then there are at least
\[h(A,B)!\cdot (a-h(A,B))!\]
different ways to  label the vertices of $A\sm B$ as $v_1,\dots,v_a$ such that $v_1,\dots,v_a$ is a tidy sequence.
\end{claim}
\begin{proof}First, note that by condition (i) in Definition \ref{defi-pleasant} we have $\vert A\sm B\vert=k-\vert B\vert=a$.

Let $h=h(A,B)$ and note that by conditions (iii) and (iv) in Definition \ref{defi-pleasant} we have 
\[\frac{\eps\cdot \sqrt{k}}{4\sqrt{r}}\leq h\leq \vert A\sm B\vert-\frac{\eps\cdot \sqrt{k}}{4\sqrt{r}}=a-\frac{\eps\cdot \sqrt{k}}{4\sqrt{r}}.\]
In particular $0<h<a$. As $h=h(A,B)$ is an integer, this implies $1\leq h\leq a-1$.

There are exactly $h=h(A,B)$ vertices in $A\sm B$ that are connected to $B$. There are $h!$ ways to label these vertices as $v_1,\dots,v_h$. By condition (ii) in Definition \ref{defi-pleasant}, the remaining $a-h$ vertices in $A\sm B$ are all isolated in $A$. There are $(a-h)!$ ways to label these vertices as $v_{h+1},\dots,v_a$. It is easy to see that for each labeling of $A\sm B$ as $v_1,\dots,v_a$ obtained in this way, the sequence $v_1,\dots,v_a$ together with the index $h$ satisfies the conditions in Definition \ref{defi-pleasant}.

All in all, this gives $h!\cdot (a-h)!=h(A,B)!\cdot (a-h(A,B))!$ labelings of $A\sm B$ as $v_1,\dots,v_a$ such that $v_1,\dots,v_a$ is a tidy sequence.
\end{proof}

In particular, for each subset $A\su V(G)$ such that $(A,B)$ is an $\eps$-pleasant pair, there is at least one way to  label the vertices of $A\sm B$ as $v_1,\dots,v_a$ such that $v_1,\dots,v_a$ is a tidy sequence. Thus, we may assume that there exists at least one tidy sequence $v_1,\dots,v_a$, as otherwise there are no sets $A\su V(G)$ such that $(A,B)$ is an $\eps$-pleasant pair (and then Lemma \ref{lemma-B-pleasant-bound} is trivially true).

Note that for every tidy sequence $v_1,\dots,v_a$ (with the index $h$ as in Definition \ref{defi-tidy-sequence}), we must have $e(B\cup \lbrace v_1,\dots,v_{h}\rbrace)=e(B\cup \lbrace v_1,\dots,v_a\rbrace)=\l$, because the vertices $v_{h+1},\dots,v_a$ are all isolated in $B\cup \lbrace v_1,\dots,v_a\rbrace$ (so they are not part of any edges $e\su B\cup \lbrace v_1,\dots,v_a\rbrace$). On the other hand, we have $e(B\cup \lbrace v_1,\dots,v_{h-1}\rbrace)<e(B\cup \lbrace v_1,\dots,v_{h}\rbrace)=\l$, because $v_{h}$ is connected to $B$ (so  there exists at least one edge $e\su B\cup \lbrace v_h\rbrace$ containing $v_{h}$).

Hence we see that $e(B\cup \lbrace v_1,\dots,v_{i-1}\rbrace)<\l$ for all $i$ with $1\leq i\leq h$ and $e(B\cup \lbrace v_1,\dots,v_{i-1}\rbrace)=\l$ for all $i$ with $h+1\leq i\leq a$. Therefore, for all indices $i=1,\dots,a$, the following holds: If $e(B\cup \lbrace v_1,\dots,v_{i-1}\rbrace)<\l$, then $v_i$ is connected $B$, but if $e(B\cup \lbrace v_1,\dots,v_{i-1}\rbrace)=\l$, then $v_i$ is not connected to $B$ (as $v_i$ must be isolated in $B\cup \lbrace v_1,\dots,v_a\rbrace$).

Let us now choose a random tidy sequence $v_1,\dots,v_a$ according to the following procedure: If for some $1\leq i\leq a$ the vertices $v_1,\dots, v_{i-1}$ have already been chosen, choose $v_{i}$ uniformly at random among all vertices in $V(G)\sm (B\cup \lbrace v_1,\dots, v_{i-1}\rbrace)$ such that $v_1,\dots, v_{i}$ can be extended to some tidy sequence.

For example, $v_1$ will be chosen uniformly at random among all vertices that occur as the first vertex in some tidy sequence. Then, $v_2$ will be chosen uniformly at random among all vertices that appear as the second vertex in a tidy sequence that starts with vertex $v_1$. Continuing this random process, in the end one obtains a tidy sequence $v_1,\dots,v_a$ (and in each step there is indeed at least one choice for the next vertex).

For every tidy sequence $v_1,\dots,v_a$, the following claim gives a lower bound on the probability that this particular sequence got chosen in the random procedure.

\begin{claim}\label{claim-prob-tidy-seq} Let $v_1,\dots,v_a$ be a tidy sequence and let the index $h$ be as in Definition \ref{defi-tidy-sequence}. Then the probability that the sequence $v_1,\dots,v_a$ was chosen during the random procedure is at least
\[\frac{a^a}{h^h\cdot (a-h)^{a-h}}\cdot \frac{1}{n^a}.\]
\end{claim}
\begin{proof}
Let $X$ be the number of vertices in $V(G)\sm B$ that are connected to $B$. Note that $1\leq h\leq a-1$ implies that $v_1$ is connected to $B$, but $v_{a}$ is not connected to $B$ (since it is isolated in $B\cup \lbrace v_1,\dots,v_a\rbrace$). In particular, we see that $0<X<n$.

For every $1\leq i\leq a$, let us analyze the number of choices we have when choosing $v_i$ in the random procedure described above (after having chosen $v_1,\dots,v_{i-1}$).

If $1\leq i\leq h$, then by the above observations we have $e(B\cup \lbrace v_1,\dots,v_{i-1}\rbrace)<\l$ and, when choosing $v_i$, we must choose a vertex that is connected to $B$. Thus, there are at most $X$ choices, and the probability that the desired vertex $v_i$ is chosen is at least $1/X$.

On the other hand, if $h+1\leq i\leq a$, then by the above observations we have $e(B\cup \lbrace v_1,\dots,v_{i-1}\rbrace)=\l$ and, when choosing $v_i$, we must choose a vertex that is not connected to $B$. Thus, there are at most $n-X$ choices, and the probability that the desired vertex $v_i$ is chosen is at least $1/(n-X)$.

So all in all, the probability that we make the desired choice in each step and obtain precisely the sequence $v_1,\dots,v_a$ is at least
\[\left(\frac{1}{X}\right)^h\cdot \left(\frac{1}{n-X}\right)^{a-h}.\]

Now, let $0<x<1$ be such that $X=xn$, then 
\[\left(\frac{1}{X}\right)^h\cdot \left(\frac{1}{n-X}\right)^{a-h}=\left(\frac{1}{xn}\right)^h\cdot \left(\frac{1}{(1-x)n}\right)^{a-h}=\frac{1}{x^h(1-x)^{a-h}}\cdot \frac{1}{n^a}.\]
Note that by the inequality of arithmetic and geometric mean we have
\begin{multline*}
x^h(1-x)^{a-h}=h^h\cdot (a-h)^{a-h}\cdot\left(\frac{x}{h}\right)^h\cdot\left(\frac{1-x}{a-h}\right)^{a-h}\\
\leq h^h\cdot (a-h)^{a-h}\cdot\left(\frac{h\cdot \frac{x}{h}+(a-h)\cdot \frac{1-x}{a-h}}{h+a-h}\right)^{h+a-h}=h^h\cdot (a-h)^{a-h}\cdot\left(\frac{1}{a}\right)^{a}.
\end{multline*}
Hence the probability to obtain precisely the sequence $v_1,\dots,v_a$ during the random procedure is at least
\[\left(\frac{1}{X}\right)^h\cdot \left(\frac{1}{n-X}\right)^{a-h}=\frac{1}{x^h(1-x)^{a-h}}\cdot \frac{1}{n^a}\geq \frac{a^a}{h^h\cdot (a-h)^{a-h}}\cdot \frac{1}{n^a},\]
which finishes the proof of the claim.
\end{proof}

\begin{corollary}\label{coro-A-pleasant} Let $A\su V(G)$ be such that $(A,B)$ is an $\eps$-pleasant pair. Then, when choosing $v_1,\dots,v_a$ according to the random procedure described above, we have $B\cup \lbrace v_1,\dots,v_a\rbrace=A$ with probability at least
\[\frac{\eps^{1/2}\cdot k^{1/4}}{2\cdot r^{1/4}}\cdot \frac{a!}{n^a}.\]
\end{corollary}
\begin{proof}
Recall that by Claim \ref{claim-count-labelings-tidy} there are at least $h(A,B)!\cdot (a-h(A,B))!$ labelings of $A\sm B$ as $v_1,\dots,v_a$ such that $v_1,\dots,v_a$ is a tidy sequence.  Clearly, for each of these labelings we have $B\cup \lbrace v_1,\dots,v_a\rbrace=A$. Hence for each of these labelings the index $h$ in Definition \ref{defi-tidy-sequence} equals $h(B\cup \lbrace v_1,\dots,v_a\rbrace,B)=h(A,B)$.

Thus, abbreviating $h(A,B)$ just by $h$, by Claim \ref{claim-prob-tidy-seq} each of these $h!\cdot (a-h)!$ tidy sequences $v_1,\dots,v_a$ is chosen during the random procedure with probability at least
\[\frac{a^a}{h^h\cdot (a-h)^{a-h}}\cdot \frac{1}{n^a}.\]
So all in all the probability of having $B\cup \lbrace v_1,\dots,v_a\rbrace=A$ is at least
\[h!\cdot (a-h)!\cdot \frac{a^a}{h^h\cdot (a-h)^{a-h}}\cdot \frac{1}{n^a}=\frac{h!}{h^h}\cdot \frac{(a-h)!}{(a-h)^{a-h}}\cdot \frac{a^a}{a!}\cdot \frac{a!}{n^a}.\]
Note that by conditions (iii) and (iv) in Definition \ref{defi-pleasant} we have both $h\geq \frac{1}{4}r^{-1/2}\cdot \eps\cdot \sqrt{k}$ and $a-h\geq \frac{1}{4}r^{-1/2}\cdot \eps\cdot \sqrt{k}$. Furthermore, at least one of the two numbers $h$ and $a-h$ is at least $\frac{1}{2}a$. Thus,
\[h\cdot (a-h)\geq \frac{\eps\cdot \sqrt{k}}{4\sqrt{r}} \cdot \frac{1}{2}a=\frac{\eps\cdot \sqrt{k}}{8\sqrt{r}} \cdot a.\]
Now, by Stirling's formula (see for example \cite{robbins}) we have
\begin{multline*}
\frac{h!}{h^h}\cdot \frac{(a-h)!}{(a-h)^{a-h}}\cdot \frac{a^a}{a!}\geq 
\left(\sqrt{2\pi}\cdot \sqrt{h}\cdot e^{-h}\right)\cdot \left(\sqrt{2\pi}\cdot \sqrt{a-h}\cdot e^{-a+h}\right)\cdot \left(e\cdot \sqrt{a}\cdot e^{-a}\right)^{-1}\\
=\frac{2\pi}{e}\cdot \sqrt{\frac{h\cdot (a-h)}{a}}\geq \frac{2\pi}{e}\cdot \sqrt{\frac{\eps\cdot \sqrt{k}}{8\sqrt{r}}}=\frac{2\pi}{e\cdot \sqrt{2}}\cdot \frac{1}{2}\cdot \frac{\eps^{1/2}\cdot k^{1/4}}{r^{1/4}}\geq \frac{6}{3\cdot 2}\cdot  \frac{\eps^{1/2}\cdot k^{1/4}}{2\cdot r^{1/4}}=\frac{\eps^{1/2}\cdot k^{1/4}}{2\cdot r^{1/4}},
\end{multline*}
Thus, the probability of having $B\cup \lbrace v_1,\dots,v_a\rbrace=A$ is at least
\[\frac{h!}{h^h}\cdot \frac{(a-h)!}{(a-h)^{a-h}}\cdot \frac{a^a}{a!}\cdot \frac{a!}{n^a}\geq\frac{\eps^{1/2}\cdot k^{1/4}}{2\cdot r^{1/4}}\cdot \frac{a!}{n^a},\]
as desired.
\end{proof}
Recall that we fixed the set $B\su V(G)$ and we chose a tidy sequence $v_1,\dots,v_a$ according to the random procedure described above. For different sets $A\su V(G)$ such that $(A,B)$ is an $\eps$-pleasant pair, the events $B\cup \lbrace v_1,\dots,v_a\rbrace=A$ are clearly disjoint. Hence, by Corollary \ref{coro-A-pleasant} the number of different sets $A\su V(G)$ such that $(A,B)$ is an $\eps$-pleasant pair can be at most
\[\left(\frac{\eps^{1/2}\cdot k^{1/4}}{2\cdot r^{1/4}}\cdot \frac{a!}{n^a}\right)^{-1}=2\cdot r^{1/4}\cdot \eps^{-1/2}\cdot k^{-1/4}\cdot \frac{n^a}{a!},\]
which finishes the proof of Lemma \ref{lemma-B-pleasant-bound}.

\section{Proof of Lemmas \ref{propo-1} and \ref{propo-2}}\label{sect-proof-propo-1-2}

This section is devoted to proving Lemmas \ref{propo-1} and \ref{propo-2}. The general proof strategy is very similar to the proof of Lemma \ref{propo-3} in the previous section. However, we need slightly different versions of the notions introduced in Definitions \ref{defi-pleasant} and \ref{defi-tidy-sequence}, see Definitions \ref{defi-nice} and \ref{defi-tame-sequence} below.

As in Lemmas \ref{propo-1} and \ref{propo-2}, let us fix positive integers $r$, $k$ and $\l$ and let $G$ be a hypergraph on $n \geq k$ vertices all of whose edges have size at most $r$. 

Recall that we defined the quantities $h(A,B)$, $m(A, B)$ and $f(A,B)$ at the end of Section \ref{sect-proof-o1}.

\begin{definition}\label{defi-nice}For a real number $z>0$, let a pair $(A,B)$ of subsets $B\su A\su V(G)$ be called \emph{$z$-nice} if all of the following conditions are satisfied:
\begin{itemize}
\item[(i)] $\vert A\vert=k$ and $e(A)=\l$.
\item[(ii)] $z\leq h(A,B)\leq \sqrt{k}/2$.
\item[(iii)] $\vert A\sm B\vert\geq 2\sqrt{k}\cdot f(A,B)$.
\item[(iv)] $\vert A\sm B\vert\geq \sqrt{k}$.
\end{itemize}
\end{definition}

Similarly as in the proof of Lemma \ref{propo-3}, for a given value of $z$ and a fixed subset $B\su V(G)$ we can bound the number of sets $A\su V(G)$ such that $(A,B)$ is $z$-nice:

\begin{lemma}\label{lemma-B-nice-bound}Let $z>0$ be a real number and let $B\su V(G)$ be a subset of size $\vert B\vert\leq k$. Let $a=k-\vert B\vert$. Then there are at most 
\[\frac{4}{3}\cdot\frac{1}{\sqrt{z}}\cdot \frac{n^a}{a!}\]
different sets $A\su V(G)$ such that $(A,B)$ is a $z$-nice pair.
\end{lemma}

We will prove Lemma \ref{lemma-B-nice-bound} in Subsection \ref{subsect-lemma-proof} at the end of this section. First, we will prove Lemmas \ref{propo-1} and \ref{propo-2} in the next two subsections assuming Lemma \ref{lemma-B-nice-bound}. Most of the arguments will be quite similar to those we used in the previous section to prove Lemma \ref{propo-3}. However, some of the details are different, which makes it necessary to prove Lemmas \ref{propo-1} and \ref{propo-2} separately from Lemma \ref{propo-3}.

\subsection{Proof of Lemma \ref{propo-1}}

In this subsection we will prove Lemma \ref{propo-1} assuming Lemma \ref{lemma-B-nice-bound}. As in the statement of Lemma \ref{propo-1}, let $c$ be a real number with $0<c<\sqrt{k}/2$. Note that if $c\leq 32^2 \cdot r$, then
\[32\cdot \sqrt{r}\cdot \frac{1}{\sqrt{c}}\cdot \frac{n^k}{k!}\geq \frac{n^k}{k!}\geq {n \choose k}.\]
So in this case Lemma \ref{propo-1} is trivially true as the total number of $k$-vertex subsets of $V(G)$ is only ${n\choose k}$. Thus, we can assume that $c\geq 32^2 \cdot r$.

We start by proving that for each fixed $k$-vertex subset $A\su V(G)$ with the properties in Lemma \ref{propo-1}, we have a probability of at least $\frac{1}{4}$ to obtain a $z$-nice pair $(A,B)$ when choosing a random subset $B\su A$ by picking each element of $A$ independently with probability $p$ (for appropriate choices of the parameters $z$ and $p$).

\begin{lemma}\label{lemma-A-1}Let $z=\frac{1}{32r}c$ and $p=1-\frac{1}{8r}$. Then for every $k$-vertex subset $A\su V(G)$ with $e(A)=\l$ and $c\leq m(A)\leq \sqrt{k}/2$, the following holds: If we choose a subset $B\su A$ randomly by taking each element of $A$ into $B$ with probability $p$ (independently for all elements of $A$), then with probability at least $\frac{1}{4}$ the resulting pair $(A,B)$ is $z$-nice.
\end{lemma}

\begin{proof}Fix a subset $A\su V(G)$ with $\vert A\vert=k$ as well as $e(A)=\l$ and $c\leq m(A)\leq \sqrt{k}/2$. As in the lemma, let the subset $B\su A$ be chosen by randomly picking elements of $A$ independently with probability $p$. Let $q=1-p=\frac{1}{8r}$, then each element of $A$ will be in $A\sm B$ with probability $q$.

We claim that with probability at least $\frac{1}{2}$ we have
\begin{equation}\label{ineq-fAB}
f(A,B)\leq 2r\cdot q^2 \cdot m(A)=\frac{1}{32r}\cdot m(A).
\end{equation}
Recall that $f(A,B)$ counts the number of vertices in $A\sm B$ that are non-isolated in $A$ but not connected to $B$. In total there are $m(A)$ non-isolated vertices $v$ in $A$. For each such $v$ in order to count towards $f(A,B)$, it must satisfy $v\in A\sm B$ and for each edge $e\su A$ with $v\in e$ there must be at least one vertex $v'\in e\sm \lbrace v\rbrace$ with $v'\in A\sm B$. The probability for $v\in A\sm B$ is exactly $q$. Let us just pick one edge $e\su A$ with $v\in e$ (such an edge exists as $v$ is non-isolated in $A$). The probability for there to exist $v'\in e\sm \lbrace v\rbrace$ with $v'\in A\sm B$ is at most $(r-1)\cdot q$ (because there are at most $\vert e\sm \lbrace v\rbrace\vert\leq r-1$ possibilities for $v'$ and each of them is in $A\sm B$ with probability $q$), and this is independent of whether $v\in A\sm B$. Thus, the total probability for $v$ to count towards $f(A,B)$ is at most $(r-1)\cdot q^2$. Summing this over all $m(A)$ choices for $v$, we obtain
\[\mathbb{E}[f(A,B)]\leq (r-1)\cdot q^2\cdot m(A)\leq r\cdot q^2\cdot m(A).\]
Clearly, $f(A,B)$ is a non-negative random variable, so by Markov's inequality we have
\[\mathbb{P}[f(A,B)\geq 2r\cdot q^2\cdot m(A)]\leq \frac{\mathbb{E}[f(A,B)]}{2r\cdot q^2\cdot m(A)}\leq \frac{r\cdot q^2\cdot m(A)}{2r\cdot q^2\cdot m(A)}=\frac{1}{2}.\]
This shows that (\ref{ineq-fAB}) indeed holds with probability at least $\frac{1}{2}$.

Now, we claim that with probability at least $\frac{7}{8}$ we have
\begin{equation}\label{ineq-mAB-low}
m(A,B)\geq \frac{1}{2}\cdot q \cdot m(A)=\frac{1}{16r}\cdot m(A).
\end{equation}
Recall that $m(A,B)$ counts the number of vertices in $A\sm B$ that are non-isolated in $A$. As there are $m(A)$ vertices of $A$ that are non-isolated in $A$, we can see that $m(A,B)$ is binomially distributed with parameters $m(A)$ and $q$. By \cite[Theorem A.1.13]{alon-spencer}, we therefore have 
\begin{multline*}
\mathbb{P}\left[m(A,B)< \frac{1}{2}\cdot q \cdot m(A)\right]<\exp\left(-\frac{\left(\frac{1}{2}\cdot q \cdot m(A)\right)^2}{2\cdot q\cdot m(A)}\right)=\exp\left(-\frac{q}{8}\cdot m(A)\right)\\
=\exp\left(-\frac{1}{64r}\cdot m(A)\right)\leq e^{-3}\leq \frac{1}{8}
\end{multline*}
where we used $m(A)\geq c\geq 32^2\cdot r\geq3\cdot  64r$. So (\ref{ineq-mAB-low}) indeed holds with probability at least $\frac{7}{8}$.

Furthermore, by the assumptions on $A$ we always have
\begin{equation}\label{ineq-mAB-high}
h(A,B)\leq m(A,B)\leq m(A)\leq \sqrt{k}/2.
\end{equation}

Finally, we claim that with probability at least $\frac{7}{8}$ we have
\begin{equation}\label{ineq-AsmB}
\vert A\sm B\vert\geq \frac{1}{2}\cdot q \cdot k=\frac{1}{16r}\cdot k.
\end{equation}
Indeed, $\vert A\sm B\vert$ is binomially distributed with parameters $\vert A\vert=k$ and $q$. So again using \cite[Theorem A.1.13]{alon-spencer}, we have 
\[\mathbb{P}\left[m(A,B)< \frac{1}{2}\cdot q \cdot k\right]<\exp\left(-\frac{\left(\frac{1}{2}\cdot q \cdot k\right)^2}{2\cdot q\cdot k}\right)=\exp\left(-\frac{q}{8}\cdot k\right)=\exp\left(-\frac{1}{64r}\cdot k\right)\leq e^{-3}\leq \frac{1}{8}\]
where we used $k\geq \sqrt{k}/2\geq c\geq 32^2\cdot r\geq 3\cdot 64r$. So (\ref{ineq-AsmB}) indeed holds with probability at least $\frac{7}{8}$.

All in all, (\ref{ineq-fAB}) fails with probability at most $\frac{1}{2}$,  each of (\ref{ineq-mAB-low}) and  (\ref{ineq-AsmB}) fails with probability at most $\frac{1}{8}$, and (\ref{ineq-mAB-high}) always holds. Thus, with probability at least $\frac{1}{4}$ all the inequalities (\ref{ineq-fAB}) to (\ref{ineq-AsmB}) hold. We claim that in this case the pair $(A,B)$ is $z$-nice.

First, note that condition (i) in Definition \ref{defi-nice} is automatically satisfied by the assumptions on $A$. The upper bound $h(A,B)\leq \sqrt{k}/2$ in condition (ii) follows from (\ref{ineq-mAB-high}). Furthermore, note that (\ref{ineq-fAB}) and (\ref{ineq-mAB-low}) imply
\[h(A,B)=m(A,B)-f(A,B)\geq \frac{1}{16r}\cdot m(A)-\frac{1}{32r}\cdot m(A)=\frac{1}{32r}\cdot m(A)\geq \frac{1}{32r}\cdot c=z,\]
which gives the lower bound in condition (ii). For condition (iii), note that the inequalities (\ref{ineq-AsmB}) and (\ref{ineq-fAB}) imply (using $m(A)\leq \sqrt{k}/2\leq \sqrt{k}$)
\[\vert A\sm B\vert\geq \frac{1}{16r}\cdot k=2\sqrt{k}\cdot \frac{1}{32r}\cdot \sqrt{k}\geq 
2\sqrt{k}\cdot \frac{1}{32r}\cdot m(A)\geq 2\sqrt{k}\cdot f(A,B).\]
For condition (iv), note that $16r\leq 32^2\cdot r\leq c\leq \sqrt{k}/2\leq \sqrt{k}$, so (\ref{ineq-AsmB}) gives
\[\vert A\sm B\vert\geq \frac{1}{16r}\cdot k\geq \frac{1}{\sqrt{k}}\cdot k=\sqrt{k}.\]
Thus, the pair $(A,B)$ is indeed $z$-nice if all the inequalities (\ref{ineq-fAB}) to (\ref{ineq-AsmB}) hold. Hence we have shown that $(A,B)$ is a $z$-nice pair with probability at least $\frac{1}{4}$.
\end{proof}

We will now prove Lemma \ref{propo-1}. Let us suppose for contradiction that there are more than
\[32\cdot\sqrt{r}\cdot \frac{1}{\sqrt{c}}\cdot \frac{n^k}{k!}.\]
$k$-vertex subsets $A\su V(G)$ with $e(A)=\l$ and $c\leq m(A)\leq \sqrt{k}/2$. 

Define $z=\frac{1}{32r}c$ and $p=1-\frac{1}{8r}$ as in Lemma \ref{lemma-A-1}. Now, let us choose a pair $(A,B)$ of subsets $B\su A\su V(G)$ with $\vert A\vert=k$ randomly as follows: First, choose $A\su V(G)$ uniformly at random among all subsets of size $\vert A\vert=k$. Then, choose a subset $B\su A$ by taking each element of $A$ into $B$ with probability $p$ (independently for all elements of $A$).

The probability that $A$ satisfies $e(A)=\l$ and $c\leq m(A)\leq \sqrt{k}/2$ is at least
\[32\cdot\sqrt{r}\cdot \frac{1}{\sqrt{c}}\cdot \frac{n^k}{k!}\cdot {n\choose k}^{-1}=32\cdot\sqrt{r}\cdot \frac{1}{\sqrt{c}}\cdot\frac{n^k}{(n)_k}.\]
If this is the case, then by Lemma \ref{lemma-A-1} the pair $(A,B)$ is $z$-nice with probability at least $\frac{1}{4}$. So we obtain
\begin{equation}\label{ineq-ABniceup1}
\mathbb{P}[(A,B)\text{ is }z\text{-nice}]\geq 8\cdot\sqrt{r}\cdot \frac{1}{\sqrt{c}}\cdot\frac{n^k}{(n)_k}.
\end{equation}

However, we can find an upper bound on the probability that $(A,B)$ is $z$-nice by conditioning on $B$. Note that $B$ will always be a subset $B\su V(G)$ with $\vert B\vert \leq k$. So let us fix any subset $B\su V(G)$ with $\vert B\vert \leq k$ and let $a=k-\vert B\vert$. Conditioning on having this particular set $B$ as the result of our random draw of the pair $(A,B)$, there are
\[{n-\vert B\vert\choose k-\vert B\vert}={n-\vert B\vert\choose a}=\frac{(n-\vert B\vert)_a}{a!}\]
possibilities for the set $A$ and they are all equally likely. But by Lemma \ref{lemma-B-nice-bound} at most
\[\frac{4}{3}\cdot\frac{1}{\sqrt{z}}\cdot \frac{n^a}{a!}=\frac{4}{3}\cdot\frac{1}{\sqrt{\frac{1}{32r}c}}\cdot \frac{n^a}{a!}< \frac{4}{3}\cdot 6\cdot \sqrt{r}\cdot\frac{1}{\sqrt{c}}\cdot \frac{n^a}{a!}=8\cdot \sqrt{r}\cdot\frac{1}{\sqrt{c}}\cdot \frac{n^a}{a!}\]
of these possible sets $A$ have the property that $(A,B)$ is $z$-nice. Thus, when we condition on having this set $B$, the probability for $(A,B)$ to be $z$-nice is less than
\[8\cdot \sqrt{r}\cdot\frac{1}{\sqrt{c}}\cdot \frac{n^a}{(n-\vert B\vert)_a}=
8\cdot \sqrt{r}\cdot\frac{1}{\sqrt{c}}\cdot \frac{n^{k-\vert B\vert}}{(n-\vert B\vert)_{k-\vert B\vert}}\leq 8\cdot \sqrt{r}\cdot\frac{1}{\sqrt{c}}\cdot \frac{n^k}{(n)_k}.\]
Since the set $B$ was arbitrary, this shows
\[\mathbb{P}[(A,B)\text{ is }z\text{-nice}]<8\cdot \sqrt{r}\cdot\frac{1}{\sqrt{c}}\cdot \frac{n^k}{(n)_k},\]
a contradiction to (\ref{ineq-ABniceup1}). This finishes the proof of Lemma \ref{propo-1}.

\subsection{Proof of Lemma \ref{propo-2}}

In this subsection we will prove Lemma \ref{propo-2} assuming Lemma \ref{lemma-B-nice-bound}. The proof is very similar to the proof of Lemma \ref{propo-1} in the previous subsection. This time, let $c$ be a real number with $\sqrt{k}/2\leq c\leq k/(32r)$. Note that if $k^{1/4}\leq 44 \cdot \sqrt{r}$, then
\[44\cdot \sqrt{r}\cdot k^{-1/4}\cdot \frac{n^k}{k!}\geq \frac{n^k}{k!}\geq {n \choose k}.\]
and Lemma \ref{propo-2} is trivially true. Thus, we can assume that $k^{1/4}\geq 44 \cdot \sqrt{r}$, which means $k\geq 44^4 \cdot r^2$.

\begin{lemma}\label{lemma-A-2}Let $z=\frac{1}{64r}\sqrt{k}$ and $p=1-\frac{1}{16r}\sqrt{k}\cdot c^{-1}$. Then for every $k$-vertex subset $A\su V(G)$ with $e(A)=\l$ and $c\leq m(A)\leq 2c$, the following holds: If we choose a subset $B\su A$ randomly by taking each element of $A$ into $B$ with probability $p$ (independently for all elements of $A$), then with probability at least $\frac{1}{4}$ the resulting pair $(A,B)$ is $z$-nice.
\end{lemma}

\begin{proof}Fix a subset $A\su V(G)$ with $\vert A\vert=k$ as well as $e(A)=\l$ and $c\leq m(A)\leq 2c$. As in the lemma, let the subset $B\su A$ be chosen by randomly picking elements of $A$ independently with probability $p$. Let $q=1-p=\frac{1}{16r}\sqrt{k}\cdot c^{-1}$, so each element of $A$ will be in $A\sm B$ with probability $q$.

We claim that with probability at least $\frac{1}{2}$ we have
\begin{equation}\label{ineq-fAB-2}
f(A,B)\leq 2r\cdot q^2 \cdot m(A).
\end{equation}
With the same argument as in the proof of Lemma \ref{lemma-A-1}, we can see that $\mathbb{E}[f(A,B)]\leq r\cdot q^2\cdot m(A)$ and again by Markov's inequality we have
\[\mathbb{P}[f(A,B)\geq 2r\cdot q^2\cdot m(A)]\leq \frac{\mathbb{E}[f(A,B)]}{2r\cdot q^2\cdot m(A)}\leq \frac{r\cdot q^2\cdot m(A)}{2r\cdot q^2\cdot m(A)}=\frac{1}{2}.\]
This shows that (\ref{ineq-fAB-2}) indeed holds with probability at least $\frac{1}{2}$.

Next, we claim that with probability at least $\frac{15}{16}$ we have
\begin{equation}\label{ineq-mAB-low-2}
m(A,B)\geq \frac{1}{2}\cdot q \cdot m(A).
\end{equation}
As in the proof of Lemma \ref{lemma-A-1} we see that $m(A,B)$ is binomially distributed with parameters $m(A)$ and $q$. So again using \cite[Theorem A.1.13]{alon-spencer}, we have 
\begin{multline*}
\mathbb{P}\left[m(A,B)< \frac{1}{2}\cdot q \cdot m(A)\right]<\exp\left(-\frac{\left(\frac{1}{2}\cdot q \cdot m(A)\right)^2}{2\cdot q\cdot m(A)}\right)=\exp\left(-\frac{1}{8}\cdot q\cdot m(A)\right)\\
\leq\exp\left(-\frac{1}{8}\cdot q\cdot c\right)=\exp\left(-\frac{1}{8}\cdot \frac{1}{16r}\cdot \sqrt{k}\right)=\exp\left(-\frac{1}{128r}\cdot \sqrt{k}\right)\leq e^{-4}\leq \frac{1}{16}
\end{multline*}
where this time we used $m(A)\geq c$ and $\sqrt{k}\geq 44^2 \cdot r\geq 4\cdot 128r$. So (\ref{ineq-mAB-low-2}) indeed holds with probability at least $\frac{15}{16}$.

Furthermore, we claim that with probability at least $\frac{15}{16}$ we have
\begin{equation}\label{ineq-mAB-high-2}
m(A,B)\leq \sqrt{k}/2.
\end{equation}
Note that
\[q\cdot m(A)\leq q\cdot 2c=2\cdot \frac{1}{16r}\sqrt{k}=\frac{1}{8r}\sqrt{k}\leq \frac{\sqrt{k}}{8}.\]
Therefore, again using that $m(A,B)$ is binomially distributed with parameters $m(A)$ and $q$, the standard Chernoff bound for the upper tail of binomially distributed variables (see, for example, \cite[Theorem A.1.4]{alon-spencer}) gives
\begin{multline*}
\mathbb{P}\left[m(A,B)> \frac{\sqrt{k}}{2}\right]<\exp\left(-2\cdot \left(\frac{\sqrt{k}}{2}-q\cdot m(A)\right)^2\cdot \frac{1}{m(A)}\right)\leq \exp\left(-2\cdot \left(\frac{3}{8}\sqrt{k}\right)^2\cdot \frac{1}{m(A)}\right)\\
=\exp\left(-\frac{9}{32}\cdot k\cdot \frac{1}{m(A)}\right)\leq \exp\left(-\frac{9}{32}\cdot k\cdot \frac{1}{k/16}\right)=\exp\left(-\frac{9}{2}\right)\leq e^{-4}\leq \frac{1}{16}.
\end{multline*}
Here we used that $m(A)\leq 2c\leq k/(16r)\leq k/16$. So (\ref{ineq-mAB-high-2}) holds with probability at least $\frac{15}{16}$.

Finally, we claim that with probability at least $\frac{15}{16}$ we have
\begin{equation}\label{ineq-AsmB-2}
\vert A\sm B\vert\geq \frac{1}{2}\cdot q \cdot k.
\end{equation}
Using that $\vert A\sm B\vert$ is binomially distributed with parameters $\vert A\vert=k$ and $q$ and applying \cite[Theorem A.1.13]{alon-spencer}, we obtain
\begin{multline*}
\mathbb{P}\left[m(A,B)< \frac{1}{2}\cdot q \cdot k\right]<\exp\left(-\frac{\left(\frac{1}{2}\cdot q \cdot k\right)^2}{2\cdot q\cdot k}\right)=\exp\left(-\frac{q}{8}\cdot k\right)=\exp\left(-\frac{1}{128r}\cdot \frac{\sqrt{k}}{c}\cdot k\right)\\
\leq \exp\left(-\frac{1}{128r}\cdot \frac{\sqrt{k}}{k/(32r)}\cdot k\right)=\exp\left(-\frac{\sqrt{k}}{4}\right)\leq e^{-4}\leq \frac{1}{16},
\end{multline*}
where we used $c\leq k/(32r)$ and $k\geq 44^4 \cdot r^2\geq 16^2$. So (\ref{ineq-AsmB-2}) holds with probability at least $\frac{15}{16}$.

All in all, (\ref{ineq-fAB-2}) fails with probability at most $\frac{1}{2}$,  and each of (\ref{ineq-mAB-low-2}), (\ref{ineq-mAB-high-2}) and  (\ref{ineq-AsmB-2}) fails with probability at most $\frac{1}{16}$. Thus, with probability at least $\frac{1}{4}$ all the inequalities (\ref{ineq-fAB-2}) to (\ref{ineq-AsmB-2}) hold. We claim that in this case the pair $(A,B)$ is $z$-nice.

First, note that condition (i) in Definition \ref{defi-nice} is automatically satisfied by the assumptions on $A$. For the upper bound in condition (ii), note that  (\ref{ineq-mAB-high-2}) implies $h(A,B)\leq m(A,B)\leq \sqrt{k}/2$. For the lower bound in condition (ii), first note that by the assumption $c\geq \sqrt{k}/2$ we have
\[q=\frac{1}{16r}\cdot \frac{\sqrt{k}}{c}\leq \frac{1}{16r}\cdot \frac{\sqrt{k}}{\sqrt{k}/2}=\frac{1}{8r}.\]
Hence (\ref{ineq-fAB-2}) and (\ref{ineq-mAB-low-2}) imply
\begin{multline*}
h(A,B)=m(A,B)-f(A,B)\geq \frac{1}{2}\cdot q \cdot m(A)-2r\cdot q^2 \cdot m(A)=q\cdot m(A)\cdot \left(\frac{1}{2}-2r\cdot q\right)\\
\geq q\cdot m(A)\cdot \left(\frac{1}{2}-2r\cdot \frac{1}{8r}\right)=\frac{1}{4}\cdot q\cdot m(A)\geq \frac{1}{4}\cdot q\cdot c=\frac{1}{4}\cdot\frac{1}{16r}\cdot \sqrt{k}=\frac{1}{64r}\cdot \sqrt{k}=z,
\end{multline*}
which gives the lower bound in condition (ii). For condition (iii), note that the inequalities (\ref{ineq-fAB-2}) and (\ref{ineq-AsmB-2}) imply (using $m(A)\leq 2c$)
\[2\sqrt{k}\cdot f(A,B)\leq 2\sqrt{k}\cdot 2r\cdot q^2 \cdot m(A)\leq 4\sqrt{k}\cdot r\cdot q^2 \cdot 2c
=8\sqrt{k}\cdot r\cdot q \cdot \frac{1}{16r}\sqrt{k}=\frac{1}{2}\cdot q \cdot k\leq \vert A\sm B\vert.\]
For condition (iv), note that using $c\leq k/(32r)$, inequality (\ref{ineq-AsmB-2}) gives
\[\vert A\sm B\vert\geq \frac{1}{2}\cdot q\cdot k=\frac{1}{2}\cdot \frac{1}{16r}\cdot \frac{\sqrt{k}}{c}\cdot k\geq\frac{1}{32r}\cdot \frac{\sqrt{k}}{k/(32r)}\cdot k =\sqrt{k}.\]
Thus, the pair $(A,B)$ is indeed $z$-nice if the inequalities (\ref{ineq-fAB-2}) to (\ref{ineq-AsmB-2}) hold. Hence we have shown that $(A,B)$ is a $z$-nice pair with probability at least $\frac{1}{4}$.
\end{proof}

Let us now prove Lemma \ref{propo-2}. So suppose for contradiction that there are more than
\[44\cdot \sqrt{r}\cdot k^{-1/4}\cdot \frac{n^k}{k!}\]
$k$-vertex subsets $A\su V(G)$ with $e(A)=\l$ and $c\leq m(A)\leq 2c$. 

Define $z=\frac{1}{64r}\sqrt{k}$ and $p=1-\frac{1}{16r}\sqrt{k}\cdot c^{-1}$ as in Lemma \ref{lemma-A-2}. Again, let us choose a pair $(A,B)$ of subsets $B\su A\su V(G)$ with $\vert A\vert=k$ randomly as follows: First, choose $A\su V(G)$ uniformly at random among all subsets of size $\vert A\vert=k$. Then, choose a subset $B\su A$ by taking each element of $A$ into $B$ with probability $p$ (independently for all elements of $A$).

The probability that $A$ satisfies $e(A)=\l$ and $c\leq m(A)\leq 2c$ is at least
\[44\cdot \sqrt{r}\cdot k^{-1/4}\cdot \frac{n^k}{k!}\cdot {n\choose k}^{-1}=44\cdot \sqrt{r}\cdot k^{-1/4}\cdot\frac{n^k}{(n)_k},\]
and in this case, by Lemma \ref{lemma-A-1} the pair $(A,B)$ is $z$-nice with probability at least $\frac{1}{4}$. Thus, 
\begin{equation}\label{ineq-ABniceup2}
\mathbb{P}[(A,B)\text{ is }z\text{-nice}]\geq 11\cdot \sqrt{r}\cdot k^{-1/4}\cdot\frac{n^k}{(n)_k}.
\end{equation}

On the other hand, we can again find an upper bound on the probability that $(A,B)$ is $z$-nice by conditioning on $B$. Let us fix any subset $B\su V(G)$ with $\vert B\vert \leq k$ and let $a=k-\vert B\vert$. Conditioning on having this particular set $B$ as the result of our random draw of the pair $(A,B)$, there are
\[{n-\vert B\vert\choose k-\vert B\vert}={n-\vert B\vert\choose a}=\frac{(n-\vert B\vert)_a}{a!}\]
possibilities for the set $A$ and they are all equally likely. By Lemma \ref{lemma-B-nice-bound} at most
\[\frac{4}{3}\cdot\frac{1}{\sqrt{z}}\cdot \frac{n^a}{a!}=\frac{4}{3}\cdot\frac{1}{\sqrt{\frac{1}{64r}\sqrt{k}}}\cdot \frac{n^a}{a!}= \frac{4}{3}\cdot 8\cdot \sqrt{r}\cdot k^{-1/4}\cdot \frac{n^a}{a!}<11\cdot \sqrt{r}\cdot k^{-1/4}\cdot \frac{n^a}{a!}\]
of these possible sets $A$ have the property that $(A,B)$ is $z$-nice. Thus, conditioning on having this set $B$, the probability for $(A,B)$ to be $z$-nice is less than
\[11\cdot \sqrt{r}\cdot k^{-1/4}\cdot \frac{n^a}{(n-\vert B\vert)_a}
=11\cdot \sqrt{r}\cdot k^{-1/4}\cdot \frac{n^{k-\vert B\vert}}{(n-\vert B\vert)_{k-\vert B\vert}}\leq 11\cdot \sqrt{r}\cdot k^{-1/4}\cdot \frac{n^k}{(n)_k}.\]
Since the set $B$ was arbitrary, we obtain
\[\mathbb{P}[(A,B)\text{ is }z\text{-nice}]<11\cdot \sqrt{r}\cdot k^{-1/4}\cdot \frac{n^k}{(n)_k},\]
which contradicts (\ref{ineq-ABniceup2}). This finishes the proof of Lemma \ref{propo-2}.

\subsection{Proof of Lemma \ref{lemma-B-nice-bound}}\label{subsect-lemma-proof}

In this subsection, we will prove Lemma \ref{lemma-B-nice-bound}. The proof is overall quite similar to the proof of Lemma \ref{lemma-B-pleasant-bound}, but some of the details differ significantly.

As in the statement of the lemma, fix $z>0$ as well as a subset $B\su V(G)$ of size $\vert B\vert\leq k$, and let $a=k-\vert B\vert$. Furthermore, let $s=\lfloor \sqrt{k}/2\rfloor$. 

Note that for every $z$-nice pair $(A,B)$ we have $\vert A\vert=k$ by condition (i) in Definition \ref{defi-nice} and therefore $\vert A\sm B\vert=k-\vert B\vert=a$. Thus, by condition (iv) in Definition \ref{defi-nice} we must have $a\geq \sqrt{k}$ if there exists at least one $z$-nice pair $(A,B)$. This means that we may assume that $a\geq \sqrt{k}$, because otherwise Lemma \ref{lemma-B-nice-bound} is trivially true. This in particular implies $a\geq 2\cdot\lfloor \sqrt{k}/2\rfloor=2s$.

Furthermore, if there exists some $z$-nice pair $(A,B)$, we must have $0<z\leq h(A,B)\leq \sqrt{k}/2$ by condition (ii) in Definition \ref{defi-nice}. As $h(A,B)$ is an integer, this implies $h(A,B)\geq 1$ and hence $\sqrt{k}/2\geq 1$. Thus, we must have $s=\lfloor \sqrt{k}/2\rfloor\geq 1$ if there exists at least one $z$-nice pair $(A,B)$. So we may assume $s\geq 1$ from now on.

\begin{definition}\label{defi-tame-sequence} We call a sequence $v_1,\dots,v_a$ of distinct vertices in $V(G)\sm B$ \emph{tame} if $e(B\cup \lbrace v_1,\dots,v_a\rbrace)=\l$ and if there exists a positive integer $h$ with $z\leq h\leq s$ such that
\begin{itemize}
\item none of the vertices $v_1,\dots,v_{a-s}$ is connected to $B$, 
\item each of the vertices $v_{a-s+1},\dots,v_{a-s+h}$ is connected to $B$, and
\item each of the vertices $v_{a-s+h+1},\dots,v_a$ is isolated in $B\cup \lbrace v_1,\dots,v_a\rbrace$.
\end{itemize}
\end{definition}

Note that for a sequence $v_1,\dots,v_a$ as in Definition \ref{defi-tame-sequence}, the vertices $v_{a-s+h+1},\dots,v_a$ cannot be connected to $B$ (since otherwise they would be non-isolated in $B\cup \lbrace v_1,\dots,v_a\rbrace$). Thus, the number $h(B\cup \lbrace v_1,\dots,v_a\rbrace, B)$ of vertices in $\lbrace v_1,\dots,v_a\rbrace$ that are connected to $B$ is precisely $h$. Hence for every tame sequence $v_1,\dots,v_a$, the integer $h$ in Definition \ref{defi-tame-sequence} must equal $h(B\cup \lbrace v_1,\dots,v_a\rbrace, B)$.

Similarly as in the proof of Lemma \ref{lemma-B-pleasant-bound}, we can show that for every set $A\su V(G)$ such that $(A,B)$ is $z$-nice, there are many ways to label the vertices of $A\sm B$ as $v_1,\dots,v_a$ so that $v_1,\dots,v_a$ is a tame sequence:

\begin{claim}\label{claim-count-labelings} Let $A\su V(G)$ be such that $(A,B)$ is a $z$-nice pair. Then there are at least
\[\frac{3}{4}\cdot  h(A,B)!\cdot (a-h(A,B))!\]
different ways to  label the vertices of $A\sm B$ as $v_1,\dots,v_a$ such that $v_1,\dots,v_a$ is a tame sequence.
\end{claim}
\begin{proof}First, note that by condition (i) in Definition \ref{defi-nice} we have $\vert A\sm B\vert=k-\vert B\vert=a$.

Let $h=h(A,B)$ and note that by condition (ii) in Definition \ref{defi-nice} we have $z\leq h\leq \sqrt{k}/2$, and therefore $h\leq \lfloor \sqrt{k}/2\rfloor=s$ as $h=h(A,B)$ is an integer. Thus, $z\leq h\leq s$. From $z>0$, we also obtain $h\geq 1$. 

There are exactly $h=h(A,B)$ vertices in $A\sm B$ that are connected to $B$. Label these $h$ vertices as $v_{a-s+1},\dots,v_{a-s+h}$; there are $h!$ ways how to do that. Now, pick a labeling of the remaining $a-h$ vertices in $A\sm B$ as $v_1,\dots,v_{a-s}$ and $v_{a-s+h+1},\dots,v_a$ uniformly at random from the $(a-h)!$ possible labelings.

Then $v_1,\dots,v_a$ are always distinct vertices in $V(G)\sm B$ and furthermore $e(B\cup \lbrace v_1,\dots,v_a\rbrace)=e(A)=\l$ (by condition (i) in Definition \ref{defi-nice}). We already saw that $z\leq h\leq s$. Note that by the choice of the labeling, the first and second condition in Definition \ref{defi-tame-sequence} are also automatically satisfied. So we just need to make sure that with a sufficiently high probability the third condition is satisfied as well.

In order for the third condition in Definition \ref{defi-tame-sequence} to fail, there needs to be a vertex $v\in A\sm B$ that is non-isolated in $B\cup \lbrace v_1,\dots,v_a\rbrace=A$ and obtains a label $v_i$ with $a-s+h+1\leq i\leq a$. In particular this vertex is not one of the $h=h(A,B)$ vertices in $A\sm B$ that are connected to $B$, because otherwise $v$ would have already obtained one of the labels $v_{a-s+1},\dots,v_{a-s+h}$ in the first step. This means that $v$ is one of the $f(A,B)=m(A,B)-h(A,B)$ vertices in $A\sm B$ that are non-isolated in $A$ but not connected to $B$. So there are at most $f(A,B)$ possibilities for the vertex $v$ and for each of them the probability that $v$ is labeled $v_i$ for some $a-s+h+1\leq i\leq a$ is at most
\[\frac{s-h}{a-h}\leq \frac{s}{a}.\]
So by a simple union bound, the third condition in Definition \ref{defi-tame-sequence} holds with probability at least
\[1-f(A,B)\cdot \frac{s}{a}= 1-\frac{s\cdot f(A,B)}{a}\geq 1-\frac{1}{4}=\frac{3}{4},\]
where we used that
\[a=\vert A\sm B\vert \geq  2\sqrt{k}\cdot f(A,B)\geq 4\cdot \lfloor\sqrt{k}/2\rfloor\cdot f(A,B)=4\cdot s\cdot f(A,B)\]
by condition (iii) in Definition \ref{defi-nice}.

Hence for at least $\frac{3}{4}(a-h)!$ labelings of the $a-h$ vertices in the second step as $v_1,\dots,v_{a-s}$ and $v_{a-s+h+1},\dots,v_a$, the resulting sequence $v_1,\dots,v_a$ is tame. Recall that we also had $h!$ choices in the first step (to distribute the labels $v_{a-s+1},\dots,v_{a-s+h}$). Thus, the total number of ways to  label the vertices of $A\sm B$ as $v_1,\dots,v_a$ such that $v_1,\dots,v_a$ is a tame sequence is at least
\[h!\cdot\frac{3}{4}(a-h)!=\frac{3}{4}\cdot  h!\cdot (a-h)!=\frac{3}{4}\cdot  h(A,B)!\cdot (a-h(A,B))!,\]
as desired.
\end{proof}

In particular, for each subset $A\su V(G)$ such that $(A,B)$ is a $z$-nice pair, there is a labeling of $A\sm B$ as $v_1,\dots,v_a$ such that $v_1,\dots,v_a$ is a tame sequence. Thus, we may assume that there exists at least one tame sequence $v_1,\dots,v_a$, as otherwise Lemma \ref{lemma-B-nice-bound} is trivially true.

Note that for every tame sequence $v_1,\dots,v_a$ (with the integer $h$ as in Definition \ref{defi-tame-sequence}), we must have $e(B\cup \lbrace v_1,\dots,v_{a-s+h}\rbrace)=e(B\cup \lbrace v_1,\dots,v_a\rbrace)=\l$, because the vertices $v_{a-s+h+1},\dots,v_a$ are all isolated in $B\cup \lbrace v_1,\dots,v_a\rbrace$. On the other hand, we have $e(B\cup \lbrace v_1,\dots,v_{a-s+h-1}\rbrace)<e(B\cup \lbrace v_1,\dots,v_{a-s+h}\rbrace)=\l$, because $v_{a-s+h}$ is connected to $B$.

Therefore we obtain that $e(B\cup \lbrace v_1,\dots,v_{i-1}\rbrace)<\l$ for all $i$ with $a-s+1\leq i\leq a-s+h$ and that $e(B\cup \lbrace v_1,\dots,v_{i-1}\rbrace)=\l$ for all $i$ with $a-s+h+1\leq i\leq a$. Therefore, for all $i$ with $a-s+1\leq i\leq a$, the following holds: If $e(B\cup \lbrace v_1,\dots,v_{i-1}\rbrace)<\l$, then $v_i$ is connected $B$, but if $e(B\cup \lbrace v_1,\dots,v_{i-1}\rbrace)=\l$, then $v_i$ is not connected to $B$ (as $v_i$ must be isolated in $B\cup \lbrace v_1,\dots,v_a\rbrace$). Furthermore, for all $i$ with $1\leq i\leq a-s$, the vertex $v_i$ is also not connected to $B$.

As in the proof of Lemma \ref{lemma-B-pleasant-bound}, let us choose a random tame sequence $v_1,\dots,v_a$ according to the following procedure: If for some $1\leq i\leq a$ the vertices $v_1,\dots, v_{i-1}$ have already been chosen, choose $v_{i}$ uniformly at random among all vertices in $V(G)\sm (B\cup \lbrace v_1,\dots, v_{i-1}\rbrace)$ such that $v_1,\dots, v_{i}$ can be extended to some tame sequence.

\begin{claim}\label{claim-prob-nice-seq} Let $v_1,\dots,v_a$ be a tame sequence and let the integer $h$ be as in Definition \ref{defi-tame-sequence}. Then the probability that the sequence $v_1,\dots,v_a$ was chosen during the random procedure is at least
\[\frac{a^a}{h^h\cdot (a-h)^{a-h}}\cdot \frac{1}{n^a}.\]
\end{claim}
\begin{proof}
Let $X$ be the number of vertices in $V(G)\sm B$ that are connected to $B$. Note that $a-s\geq 2s-s\geq 1$ implies that $v_1$ is not connected to $B$, whereas  $v_{a-s+1}$ is connected to $B$. In particular, we see that $0<X<n$.

For every $1\leq i\leq a$, let us analyze the number of choices we have when choosing $v_i$ in the random procedure described above (after having chosen $v_1,\dots,v_{i-1}$).

When choosing $v_i$ for $1\leq i\leq a-s$, we must choose a vertex that is not connected to $B$. Thus, there are at most $n-X$ choices, and the probability that the desired vertex $v_i$ is chosen is at least $1/(n-X)$.

If $a-s+1\leq i\leq a-s+h$, then by the above observations we have $e(B\cup \lbrace v_1,\dots,v_{i-1}\rbrace)<\l$ and, when choosing $v_i$, we must choose a vertex that is connected to $B$. Thus, there are at most $X$ choices, and the probability that the desired vertex $v_i$ is chosen is at least $1/X$.

If $a-s+ h+1\leq i\leq a$, then by the above observations we have $e(B\cup \lbrace v_1,\dots,v_{i-1}\rbrace)=\l$ and, when choosing $v_i$, we must choose a vertex that is not connected to $B$. Thus, there are at most $n-X$ choices, and the probability that the desired vertex $v_i$ is chosen is at least $1/(n-X)$.

So all in all, the probability that we make the desired choice in each step and obtain precisely the sequence $v_1,\dots,v_a$ is at least
\[\left(\frac{1}{n-X}\right)^{a-s}\cdot \left(\frac{1}{X}\right)^h\cdot \left(\frac{1}{n-X}\right)^{s-h}=\left(\frac{1}{X}\right)^h\cdot \left(\frac{1}{n-X}\right)^{a-h}.\]

In exactly the same way as in the proof of Claim \ref{claim-prob-tidy-seq}, we can show that
\[\left(\frac{1}{X}\right)^h\cdot \left(\frac{1}{n-X}\right)^{a-h}\geq \frac{a^a}{h^h\cdot (a-h)^{a-h}}\cdot \frac{1}{n^a}.\]
This finishes the proof of Claim \ref{claim-prob-nice-seq}.
\end{proof}

\begin{corollary}\label{coro-A-nice} Let $A\su V(G)$ be such that $(A,B)$ is a $z$-nice pair. Then, when choosing $v_1,\dots,v_a$ according to the random procedure described above, we have $B\cup \lbrace v_1,\dots,v_a\rbrace=A$ with probability at least
\[\frac{3}{4}\cdot \sqrt{z}\cdot \frac{a!}{n^a}.\]
\end{corollary}
\begin{proof}
Recall that by Claim \ref{claim-count-labelings} there are at least $\frac{3}{4}\cdot h(A,B)!\cdot (a-h(A,B))!$ labelings of $A\sm B$ as $v_1,\dots,v_a$ such that $v_1,\dots,v_a$ is a tame sequence.  For each of these labelings we clearly have $B\cup \lbrace v_1,\dots,v_a\rbrace=A$ and the integer $h$ in Definition \ref{defi-tame-sequence} equals $h(B\cup \lbrace v_1,\dots,v_a\rbrace,B)=h(A,B)$.

Abbreviating $h(A,B)$ just by $h$, by Claim \ref{claim-prob-nice-seq} each of these $\frac{3}{4}\cdot h!\cdot (a-h)!$ tame sequences $v_1,\dots,v_a$ is chosen during the random procedure with probability at least
\[\frac{a^a}{h^h\cdot (a-h)^{a-h}}\cdot \frac{1}{n^a}.\]
Therefore the probability of having $B\cup \lbrace v_1,\dots,v_a\rbrace=A$ is at least
\[\frac{3}{4}\cdot h!\cdot (a-h)!\cdot \frac{a^a}{h^h\cdot (a-h)^{a-h}}\cdot \frac{1}{n^a}=\frac{3}{4}\cdot \frac{h!}{h^h}\cdot \frac{(a-h)!}{(a-h)^{a-h}}\cdot \frac{a^a}{a!}\cdot \frac{a!}{n^a}.\]
By Stirling's formula (see for example \cite{robbins}) we have
\begin{multline*}
\frac{h!}{h^h}\cdot \frac{(a-h)!}{(a-h)^{a-h}}\cdot \frac{a^a}{a!}\geq 
\left(\sqrt{2\pi}\cdot \sqrt{h}\cdot e^{-h}\right)\cdot \left(\sqrt{2\pi}\cdot \sqrt{a-h}\cdot e^{-a+h}\right)\cdot \left(e\cdot \sqrt{a}\cdot e^{-a}\right)^{-1}\\
=\frac{2\pi}{e}\cdot \sqrt{h\cdot\frac{a-h}{a}}\geq \frac{2\pi}{e}\cdot \sqrt{z\cdot\frac{1}{2}}=\frac{2\pi}{e\cdot \sqrt{2}}\cdot \sqrt{z}\geq \frac{6}{3\cdot 2}\cdot \sqrt{z}=\sqrt{z},
\end{multline*}
where in the second inequality we used that $z\leq h\leq s\leq a/2$. Consequently, the probability of having $B\cup \lbrace v_1,\dots,v_a\rbrace=A$ is at least
\[\frac{3}{4}\cdot\frac{h!}{h^h}\cdot \frac{(a-h)!}{(a-h)^{a-h}}\cdot \frac{a^a}{a!}\cdot \frac{a!}{n^a}\geq\frac{3}{4}\cdot \sqrt{z}\cdot \frac{a!}{n^a},\]
as desired.
\end{proof}
Recall that we fixed the set $B\su V(G)$ and we chose a tame sequence $v_1,\dots,v_a$ according to the random procedure described above. As the events $B\cup \lbrace v_1,\dots,v_a\rbrace=A$ are clearly disjoint for different sets $A$, by Corollary \ref{coro-A-nice} the number of sets $A\su V(G)$ such that $(A,B)$ is a $z$-nice pair can be at most
\[\left(\frac{3}{4}\cdot \sqrt{z}\cdot \frac{a!}{n^a}\right)^{-1}=\frac{4}{3}\cdot \frac{1}{\sqrt{z}}\cdot \frac{n^a}{a!}.\]
This finishes the proof of Lemma \ref{lemma-B-nice-bound}.

\section{Proof of Theorem \ref{thm-graph-1-e}}\label{sect-proof-graph-1-e}

Now, we will finally prove Theorem \ref{thm-graph-1-e}. This section contains the main part of the proof of the theorem, but the proofs of several lemmas will be postponed to the following three sections.

We may assume without loss of generality that the constant $C$ in the statement of Theorem \ref{thm-graph-1-e} satisfies $C\geq 3$ (otherwise we can just make $C$ larger). Furthermore, we may assume that $k$ is sufficiently large and in particular that $k\geq 10^3 C$ and $k\geq 4\cdot \log^{10} k$.

If $1\leq \l\leq \sqrt{k}$, then Theorem \ref{thm-hyper-1-e} applied to $r=2$ implies
\[\ind(k,\l)=\ind_2(k,\l)\leq \ind_{\leq 2}(k,\l)\leq \frac{k}{k-2\l}\cdot \frac{1}{e}\leq \frac{k}{k-2\sqrt{k}}\cdot \frac{1}{e}=\frac{1}{e}+o_k(1).\]
This proves Theorem \ref{thm-graph-1-e} in the case $1\leq \l\leq \sqrt{k}$. So from now on let us assume that $\sqrt{k}\leq \l\leq C\cdot k$.

Let $n$ be large with respect to $k$ and consider a graph $G$ on $n$ vertices. We will show that for sufficiently large $n$, the graph $G$ has at most
\[\left(\frac{1}{e}+o_k(1)\right)\cdot \frac{n^k}{k!}\]
different $k$-vertex subsets $A\su V(G)$ with $e(A)=\l$. Proving this for all sufficiently large $n$ and all $n$-vertex graphs $G$ gives
\[\ind(k,\l)\leq \frac{1}{e}+o_k(1),\]
and finishes the proof of Theorem \ref{thm-graph-1-e}.

We can assume that the graph $G$ has at least $e^{-1}\cdot n^k/k!$ different $k$-vertex subsets $A\su V(G)$ with $e(A)=\l$, because otherwise we are already done.

Our proof is based on partitioning the vertex set $V(G)$ into vertices of high, medium and low degree. These notions are defined in the following definition.

\begin{definition}Let us say that a vertex $v\in V(G)$ has \emph{low degree} if $\deg_G(v)\leq 10C\cdot n/k$. Furthermore, let us say that $v$ has \emph{high degree} if $\deg_G(v)\geq 10C\cdot n/(\log^2 k)$. Finally, if $10C\cdot n/k<\deg_G(v)<10C\cdot n/(\log^2 k)$, let us say that $v$ has \emph{medium degree}.
\end{definition}

Let $\Vh$ be the set of high degree vertices, $\Vm$ be the set of medium degree vertices, and $\Vl$ the set of low degree vertices. Then $\Vh$, $\Vm$ and $\Vl$ form a partition of $V(G)$ (since $k\geq 4\cdot \log^{10} k>\log^2 k$).

\begin{lemma}\label{lemma-Vl-large-enough}$\vert \Vl\vert\geq n/12$.
\end{lemma}

We postpone the proof of Lemma \ref{lemma-Vl-large-enough} to Section \ref{sect-proofs-two-lemmas}.

Let $\gamma=e^{-120C}/48$. Note that $0<\gamma<1/64<1/2$ and $\gamma$ only depends on the constant $C$.

In order to prove the desired bound on the number of $k$-vertex subsets $A\su V(G)$ with $e(A)=\l$, we will distinguish between those subsets $A$ with $A\cap \Vh=\emptyset$ and those with $A\cap \Vh\neq \emptyset$. Let us first investigate the $k$-vertex subsets $A\su V(G)$ with $e(A)=\l$ and $A\cap \Vh=\emptyset$.

\begin{lemma}\label{lemma-suff-isolated}If we choose a $k$-vertex subset $A\su \Vl\cup \Vm$ uniformly at random, then we have $m(A)>(1-\gamma)k$ with probability at most $o_k(1)$.
\end{lemma}

We will prove Lemma \ref{lemma-suff-isolated} in Section \ref{sect-suff-isolated}. The proof will be based on a second moment computation. Using this lemma as well as Corollary \ref{coro-propo-o1}, we can now bound the number of $k$-vertex subsets $A\su V(G)$ with $e(A)=\l$ and $A\cap \Vh=\emptyset$.

\begin{corollary}\label{coro-no-Vh}The number of $k$-vertex subsets $A\su V(G)$ with $e(A)=\l$ and $A\cap \Vh=\emptyset$ is at most
\[o_k(1)\cdot \frac{n^k}{k!}.\]
\end{corollary}
\begin{proof}Lemma \ref{lemma-suff-isolated} implies that the number of $k$-vertex subsets $A\su V(G)$ with $A\cap \Vh=\emptyset$ and $m(A)>(1-\gamma)k$ is at most
\[o_k(1)\cdot {\vert \Vl\cup \Vm\vert\choose k}\leq o_k(1)\cdot {n\choose k}\leq o_k(1)\cdot \frac{n^k}{k!}.\]
In particular, the number of $k$-vertex subsets $A\su V(G)$ with $e(A)=\l$, $A\cap \Vh=\emptyset$ and $m(A)>(1-\gamma)k$ is at most $o_k(1)\cdot n^k/k!$.

On the other hand, recall that by Claim \ref{claim-mA-bounds-graph} each $k$-vertex subset $A\su V(G)$ with $e(A)=\l$ satisfies $m(A)\geq \sqrt{2\l}$. So by Corollary \ref{coro-propo-o1} applied to $r=2$, $\eps=\gamma$ and $c'=\sqrt{2\l}$, the number of $k$-vertex subsets $A\su V(G)$ with $e(A)=\l$ and $m(A)\leq (1-\gamma)k$ is at most
\[\left(32\cdot \sqrt{2}\cdot \frac{1}{(2\l)^{1/4}}+23 \cdot \sqrt{2}\cdot k^{-1/4}\cdot \log k\right)\cdot \frac{n^k}{k!}\leq  \left(48\cdot k^{-1/8}+36\cdot k^{-1/4}\cdot \log k\right)\cdot \frac{n^k}{k!}=o_k(1)\cdot \frac{n^k}{k!}.\]
Here we used our assumption that $\l\geq \sqrt{k}$. Hence, the number of $k$-vertex subsets $A\su V(G)$ with $e(A)=\l$, $A\cap \Vh=\emptyset$ and $m(A)\leq (1-\gamma)k$ is also at most $o_k(1)\cdot n^k/k!$.

Combining the bounds for these two cases, we see that total number of $k$-vertex subsets $A\su V(G)$ with $e(A)=\l$ and $A\cap \Vh=\emptyset$ is at most $o_k(1)\cdot n^k/k!$, as desired.
\end{proof}

At this point, it remains to consider the $k$-vertex subsets $A\su V(G)$ with $e(A)=\l$ and $A\cap \Vh\neq \emptyset$. The following lemma further restricts the subsets to consider.

\begin{lemma}\label{lemma-vertex-degree-off}There are at most $o_k(1)\cdot n^k/k!$ different $k$-vertex subsets $A\su V(G)$ with the property that there exists a vertex $v\in A$ with
\[\left\vert\deg_A(v)-\frac{k-1}{n}\cdot \deg_G(v)\right\vert>\sqrt{k\cdot \log k}.\]
\end{lemma}

We will prove Lemma \ref{lemma-vertex-degree-off} in Section \ref{sect-proofs-two-lemmas}. The proof is a relatively standard Chernoff bound computation.

The following definition captures the properties of the $k$-vertex subsets $A\su V(G)$ that still need to be considered to finish the proof of Theorem \ref{thm-graph-1-e}.

\begin{definition}Let us call a $k$-vertex subset $A\su V(G)$ \emph{interesting} if $e(A)=\l$ and $A\cap \Vh\neq \emptyset$ and if every vertex $v\in A$ satisfies
\[\left\vert\deg_A(v)-\frac{k-1}{n}\cdot \deg_G(v)\right\vert\leq \sqrt{k\cdot \log k}.\]
\end{definition}

The next lemma implies an upper bound on the number of interesting $k$-vertex subsets $A\su V(G)$.

\begin{lemma}\label{lemma-A-interesting}If we choose a $k$-vertex subset $A\su V(G)$ uniformly at random, then $A$ is interesting with probability at most
\[\frac{1}{e}+o_k(1).\]
\end{lemma}

We will prove Lemma \ref{lemma-A-interesting} in Section \ref{sect-A-interesting}. Now, let us finish the proof of Theorem \ref{thm-graph-1-e}.

\begin{corollary}\label{coro-Vh}The number of $k$-vertex subsets $A\su V(G)$ with $e(A)=\l$ and $A\cap \Vh\neq \emptyset$ is at most
\[\left(\frac{1}{e}+o_k(1)\right)\cdot \frac{n^k}{k!}.\]
\end{corollary}
\begin{proof}Each $k$-vertex subset $A\su V(G)$ with $e(A)=\l$ and $A\cap \Vh\neq \emptyset$ is either interesting or it contains a vertex $v\in A$ with $\left\vert\deg_A(v)-\deg_G(v)\cdot (k-1)/n\right\vert>\sqrt{k\cdot \log k}$. By Lemma \ref{lemma-vertex-degree-off}, the number of  $k$-vertex subsets $A\su V(G)$ with this second property is at most $o_k(1)\cdot n^k/k!$.

On the other hand, by Lemma \ref{lemma-A-interesting}, the number of interesting $k$-vertex subsets $A\su V(G)$ is at most
\[\left(\frac{1}{e}+o_k(1)\right)\cdot {n\choose k}\leq \left(\frac{1}{e}+o_k(1)\right)\cdot \frac{n^k}{k!}.\]
Thus, the total number of $k$-vertex subsets $A\su V(G)$ with $e(A)=\l$ and $A\cap \Vh\neq \emptyset$ is also at most $(e^{-1}+o_k(1))\cdot n^k/k!$, as desired.
\end{proof}

Now, combining Corollaries \ref{coro-no-Vh} and \ref{coro-Vh} yields that the total number of $k$-vertex subsets $A\su V(G)$ with $e(A)=\l$ is at most
\[\left(\frac{1}{e}+o_k(1)\right)\cdot \frac{n^k}{k!}.\]
This finishes the proof of Theorem \ref{thm-graph-1-e}.

\section{Proof of Lemmas \ref{lemma-Vl-large-enough} and \ref{lemma-vertex-degree-off}}\label{sect-proofs-two-lemmas}

In this section, we will prove Lemma \ref{lemma-Vl-large-enough} and \ref{lemma-vertex-degree-off}. As a preparation, we start with the following lemma, which is an easy consequence of the Chernoff bound for binomial random variables.

\begin{lemma}\label{lemma-Chernoff-sequence-G} Suppose we choose a random sequence $v_1,\dots,v_k$ of vertices in $V(G)$ by independently picking each $v_i\in V(G)$ uniformly at random. Let $Z$ be the random variable counting the number of indices $j\in \lbrace 2,\dots,k\rbrace$ such that $v_1$ and $v_j$ are adjacent (and in particular $v_1\neq v_j$). Then we have
\[\mathbb{P}\left[\left\vert Z-\deg_G(v_1)\cdot (k-1)/n\right\vert> \sqrt{k\cdot \log k}\right]<\frac{2}{k^2}.\]
Furthermore, for the conditional probability of $Z\leq \frac{1}{2}\cdot \deg_G(v_1)\cdot (k-1)/n$ conditioned on $v_1\not\in \Vl$, we have
\[\mathbb{P}\left[ Z\leq \frac{1}{2}\cdot \deg_G(v_1)\cdot \frac{k-1}{n}\,\bigg\vert\, v_1\not\in\Vl\right]<e^{-3}.\]
\end{lemma}
\begin{proof}Let us condition on the choice of $v_1$, so fix any $v_1\in V(G)$. Then for each $j\in \lbrace 2,\dots,k\rbrace$ the vertex $v_j$ will adjacent to $v_1$ with probability $\deg_G(v_1)/n$, and this is independent for all $j\in \lbrace 2,\dots,k\rbrace$. Hence, when fixing $v_1$, we see that $Z$ is a binomially distributed random variable with parameters $k-1$ and $\deg_G(v_1)/n$. In particular, the expectation of $Z$ when $v_1$ is fixed equals $\deg_G(v_1)\cdot (k-1)/n$. So by a Chernoff bound, for example in the form of \cite[Corollary A.1.7]{alon-spencer}, we obtain
\begin{multline*}
\mathbb{P}\left[\left\vert Z-\deg_G(v_1)\cdot (k-1)/n\right\vert> \sqrt{k\cdot \log k}\right]\leq \mathbb{P}\left[\left\vert Z-\deg_G(v_1)\cdot (k-1)/n\right\vert> \sqrt{k\cdot \ln k}\right]\\
<2\exp\left(-2\left(\sqrt{k\cdot \ln k}\right)^2\cdot \frac{1}{k-1}\right)=2\exp\left(-2\cdot \ln k\cdot \frac{k}{k-1}\right)<2\exp\left(-2\cdot \ln k\right)=\frac{2}{k^2},
\end{multline*}
where in the first step we used that $\log k>\ln k$. Since this holds for any choice of $v_1\in V(G)$, we also obtain the desired bound for the unconditional probability of $\left\vert Z-\deg_G(v_1)\cdot (k-1)/n\right\vert> \sqrt{k\cdot \log k}$. This proves the first inequality.

For the second inequality, let us again condition on the choice of $v_1$, but this time we assume $v_1\not\in \Vl$. As before, when fixing $v_1\not\in \Vl$, we see that $Z$ is a binomially distributed random variable with parameters $k-1$ and $\deg_G(v_1)/n$. By the Chernoff bound for lower tails of binomial random variables (see for example \cite[Corollary A.1.13]{alon-spencer}), we have
\begin{multline*}
\mathbb{P}\left[Z\leq \frac{1}{2}\cdot \deg_G(v_1)\cdot \frac{k-1}{n}\right] =
\mathbb{P}\left[Z\leq \frac{1}{2}\cdot (k-1) \cdot \frac{\deg_G(v_1)}{n}\right]\\
<\exp\left(-\frac{\left(\frac{1}{2}\cdot (k-1)\cdot \deg_G(v_1)/n\right)^2}{2\cdot (k-1)\cdot \deg_G(v_1)/n}\right)=\exp\left(-\frac{1}{8}\cdot (k-1)\cdot \frac{\deg_G(v_1)}{n}\right)\\
\leq \exp\left(-\frac{1}{8}\cdot (k-1)\cdot \frac{10C}{k}\right)\leq \exp\left(-C\right)\leq e^{-3},
\end{multline*}
where we used that $\deg_G(v_1)\geq 10C\cdot n/k$ and also $k\geq 5$ and $C\geq 3$. Since this holds for every $v_1\not\in \Vl$, we obtain the second inequality in Lemma \ref{lemma-Chernoff-sequence-G}.
\end{proof}

Let us now prove Lemma \ref{lemma-Vl-large-enough}, which states that $\vert \Vl\vert\geq n/12$.

\begin{proof}[Proof of Lemma \ref{lemma-Vl-large-enough}]
For a sequence $v_1,\dots, v_k$ of (not necessarily distinct) vertices in $V(G)$, define $Z(v_1,\dots,v_k)$ to be the number of indices $j\in \lbrace 2,\dots,k\rbrace$ such that $v_1$ and $v_j$ are adjacent. We first claim that there are at least $n^k/(2e)$ sequences $v_1,\dots, v_k$ with $Z(v_1,\dots,v_k)\leq 4C$.

Recall that we assumed at the beginning of Section \ref{sect-proof-graph-1-e} that there are at least $e^{-1}\cdot n^k/k!$ different $k$-vertex subsets $A\su V(G)$ with $e(A)=\l$. Each of these subsets $A$ satisfies
\[\sum_{v\in A}\deg_A(v)=2e(A)=2\l\leq 2C\cdot k.\]
Hence each of these subsets $A$ can contain at most $k/2$ vertices $v$ with $\deg_A(v)\geq 4C$. So for each $k$-vertex subset $A\su V(G)$ with $e(A)=\l$, there exist at least $k/2$ vertices $v\in A$ with $\deg_A(v)\leq 4C$.

Furthermore note that for each $k$-vertex subset $A\su V(G)$, for any labeling of the vertices of $A$ as $v_1,\dots,v_k$, we have $Z(v_1,\dots,v_k)=\deg_A(v_1)$, since both of these quantities just count how many of the vertices $v_2,\dots,v_k$ are adjacent to $v_1$. So from each $k$-vertex subset $A\su V(G)$ with $e(A)=\l$, we can generate at least $(k/2)\cdot (k-1)!$ labelings of the vertices of $A$ as $v_1,\dots,v_k$ such that $Z(v_1,\dots,v_k)\leq 4C$. Indeed, we can just choose $v_1\in A$ such that $\deg_A(v_1)\leq 4C$ (and we have at least $k/2$ choices to do so), and then choose any labeling of the $k-1$ remaining vertices as $v_2,\dots,v_k$. Since there are at least $e^{-1}\cdot n^k/k!$ different $k$-vertex subsets $A\su V(G)$ with $e(A)=\l$, in total that gives at least
\[\frac{1}{e}\cdot \frac{n^k}{k!}\cdot \frac{k}{2}\cdot (k-1)!=\frac{1}{2e}\cdot n^k\]
sequences $v_1,\dots, v_k$ with $Z(v_1,\dots,v_k)\leq 4C$, as claimed above.

Now suppose that we choose a random sequence $v_1,\dots,v_k$ of vertices in $V(G)$ by picking each element uniformly at random in $V(G)$, as in Lemma \ref{lemma-Chernoff-sequence-G}. Let $Z=Z(v_1,\dots,v_k)$ be the random variable counting the number of indices $j\in \lbrace 2,\dots,k\rbrace$ such that $v_1$ and $v_j$ are adjacent. Since there are at least $n^k/(2e)$ sequences $v_1,\dots, v_k$ with $Z(v_1,\dots,v_k)\leq 4C$, we have
\begin{equation}\label{ineq-prob-Z-4C-1}
\mathbb{P}[Z\leq 4C]\geq \frac{1}{2e}.
\end{equation}

On the other hand, by the second part of Lemma \ref{lemma-Chernoff-sequence-G}, we have
\[\mathbb{P}\left[ Z\leq \frac{1}{2}\cdot \deg_G(v_1)\cdot \frac{k-1}{n}\,\bigg\vert\,  v_1\not\in\Vl\right]<e^{-3}.\]
Recall that for every $v_1\not\in\Vl$, we have $\deg_G(v_1)\geq 10C\cdot n/k$ and therefore (as $k\geq 10^3C\geq 5$)
\[\frac{1}{2}\cdot \deg_G(v_1)\cdot \frac{k-1}{n}\geq \frac{1}{2}\cdot 10C\cdot \frac{n}{k}\cdot \frac{k-1}{n}=5C\cdot \frac{k-1}{k}\geq 4C.\]
Thus,
\begin{equation}\label{ineq-prob-Z-4C-2}
\mathbb{P}\left[ Z\leq 4C\mid v_1\not\in\Vl\right]\leq \mathbb{P}\left[ Z\leq \frac{1}{2}\cdot \deg_G(v_1)\cdot \frac{k-1}{n} \,\bigg\vert\,  v_1\not\in\Vl\right]<e^{-3}<\frac{1}{4e}.
\end{equation}
We clearly have
\[\mathbb{P}[Z\leq 4C]=\mathbb{P}[v_1\in\Vl]\cdot \mathbb{P}\left[ Z\leq 4C\mid v_1\in\Vl\right]+\mathbb{P}[v_1\not\in\Vl]\cdot \mathbb{P}\left[ Z\leq 4C\mid v_1\not\in\Vl\right].\]
Together with (\ref{ineq-prob-Z-4C-1}) and  (\ref{ineq-prob-Z-4C-2}), this implies
\[\frac{1}{2e}\leq \mathbb{P}[Z\leq 4C]\leq \mathbb{P}[v_1\in\Vl]\cdot 1+1\cdot \mathbb{P}\left[ Z\leq 4C\mid v_1\not\in\Vl\right]\leq \mathbb{P}[v_1\in\Vl]+\frac{1}{4e}.\]
So we obtain
\[\mathbb{P}[v_1\in\Vl]\geq \frac{1}{2e}-\frac{1}{4e}=\frac{1}{4e}.\]
Since $v_1$ was chosen uniformly at random inside $V(G)$, this means that
\[\vert \Vl\vert\geq \frac{1}{4e}\cdot n\geq \frac{n}{12}.\]
This finishes the proof of Lemma \ref{lemma-Vl-large-enough}.\end{proof}

Now, let us prove Lemma \ref{lemma-vertex-degree-off}.

\begin{proof}[Proof of Lemma \ref{lemma-vertex-degree-off}]As before, for a sequence $v_1,\dots, v_k$ of (not necessarily distinct) vertices in $V(G)$, let $Z(v_1,\dots,v_k)$ be the number of indices $j\in \lbrace 2,\dots,k\rbrace$ such that $v_1$ and $v_j$ are adjacent. We will prove Lemma \ref{lemma-vertex-degree-off} by considering the number of sequences $v_1,\dots, v_k$ such that
\begin{equation}\label{eq-property-degree-off}
\vert Z(v_1,\dots,v_k)-\deg_G(v_1)\cdot (k-1)/n\vert> \sqrt{k\cdot \log k}.
\end{equation}
The first part of Lemma \ref{lemma-Chernoff-sequence-G} states that when choosing a sequence $v_1,\dots,v_k$ randomly by independently picking each element uniformly from $V(G)$, the probability for (\ref{eq-property-degree-off}) to hold is less than $2/k^2$. Hence there are at most
\[\frac{2}{k^2}\cdot n^k\]
sequences $v_1,\dots, v_k$ satisfying (\ref{eq-property-degree-off}).

On the other hand, for each $k$-vertex subset $A\su V(G)$ with the property that there exists a vertex $v\in A$ with $\left\vert\deg_A(v)-\frac{k-1}{n}\cdot \deg_G(v)\right\vert>\sqrt{k\cdot \log k}$, we can generate at least $(k-1)!$ sequences $v_1,\dots, v_k$ with $\lbrace v_1,\dots,v_k\rbrace=A$ that satisfy (\ref{eq-property-degree-off}). Indeed, first choose $v_1\in A$ to be a vertex with $\left\vert\deg_A(v_1)-\frac{k-1}{n}\cdot \deg_G(v_1)\right\vert>\sqrt{k\cdot \log k}$ (there is at least one such choice) and then choose any labeling of the remaining $k-1$ vertices in $A$ as $v_2, \dots,v_k$ (there are $(k-1)!$ choices). Thus, we obtain at least $(k-1)!$ labelings of the vertices of $A$ as $v_1,\dots, v_k$ such that $\left\vert\deg_A(v_1)-\frac{k-1}{n}\cdot \deg_G(v_1)\right\vert>\sqrt{k\cdot \log k}$. Now, noting that for each such labeling we have $Z(v_1,\dots,v_k)=\deg_A(v_1)$, we see that we indeed generated $(k-1)!$ sequences $v_1,\dots, v_k$ with $\lbrace v_1,\dots,v_k\rbrace=A$ that satisfy (\ref{eq-property-degree-off}). Because of the condition $\lbrace v_1,\dots,v_k\rbrace=A$, these $(k-1)! $ sequences are disjoint for different choices of the set $A$.

Thus, the total number of possible $k$-vertex subsets $A\su V(G)$ with the property that there exists a vertex $v\in A$ with $\left\vert\deg_A(v)-\frac{k-1}{n}\cdot \deg_G(v)\right\vert>\sqrt{k\cdot \log k}$ can be at most
\[\frac{(2/k^2)\cdot n^k}{(k-1)!}=\frac{2}{k}\cdot \frac{n^k}{k!}=o_k(1)\cdot \cdot \frac{n^k}{k!}.\]
This proves Lemma \ref{lemma-vertex-degree-off}.
\end{proof}

\section{Proof of Lemma \ref{lemma-suff-isolated}}\label{sect-suff-isolated}

In this section, we will prove Lemma \ref{lemma-suff-isolated}. So let us assume that we choose a $k$-vertex subset $A\su \Vl\cup \Vm$ uniformly at random. We need to prove that we have $m(A)>(1-\gamma)k$ with probability at most $o_k(1)$. 

First, let $n'=\vert \Vl\cup \Vm\vert$. Then each vertex $v\in \Vl\cup \Vm$ is contained in $A$ with probability $k/n'$. We clearly have $n'\leq n$. On the other hand, Lemma \ref{lemma-Vl-large-enough} implies $n'\geq \vert \Vl\vert\geq n/12$.

Let $Y$ be the random variable counting the number of vertices $v\in \Vl$ with $v\in A$ and such that $v$ is isolated in $A$. Since $m(A)$ counts the number of non-isolated vertices in $A$, we always have $m(A)+Y\leq \vert A\vert = k$. In particular, this gives
\[\mathbb{P}[m(A)>(1-\gamma)k]\leq \mathbb{P}[Y<\gamma k].\]
Hence it suffices to prove that the probability for $Y<\gamma k$ is at most $o_k(1)$.

For each vertex $v\in \Vl$, let $Y_v$ be the indicator random variable for the event that $v\in A$ and $v$ is isolated in $A$. By definition we have
\[Y=\sum_{v\in \Vl}Y_v.\]

For each $v\in \Vl$, let $0\leq \delta(v)\leq 1$ be such that $\delta(v)n'=\vert N(v)\cap (\Vl\cup\Vm)\vert$. Then the number of vertices in $\Vl\cup\Vm\sm \lbrace v\rbrace$ that are not adjacent to $v$ is precisely
\[\vert \Vl\cup\Vm\vert-1-\vert N(v)\cap (\Vl\cup\Vm)\vert=n'-1-\delta(v) n'=(1-\delta(v))n'-1.\]
Since $v\in \Vl$, we have $\delta(v)n'\leq \vert N(v)\vert=\deg_G(v)\leq 10C\cdot n/k$. As $n'\geq n/12$, this gives $\delta(v)\leq 120C/k\leq 1/4$ (recall that we assumed $k\geq 10^3C$ at the beginning of Section \ref{sect-proof-graph-1-e}).

First, for each vertex $v\in \Vl$, let us analyze the expectation $\E[Y_v]$. This is the probability that $v\in A$ and $v$ is isolated in $A$. Note that we have $v\in A$ with probability $k/n'$. Conditioning on this, in order for $v$ to be isolated in $A$, the remaining $k-1$ vertices have to be chosen in such a way that none of them is adjacent to $v$. Since there are precisely $(1-\delta(v))n'-1$ vertices in $\Vl\cup\Vm\sm \lbrace v\rbrace$ that are not adjacent to $v$, the probability for this to happen is
\[\frac{(1-\delta(v))n'-1}{n'-1}\cdot \frac{(1-\delta(v))n'-2}{n'-2}\dotsm \frac{(1-\delta(v))n'-k+1}{n'-k+1}=\frac{((1-\delta(v))n'-1)_{k-1}}{(n'-1)_{k-1}}.\]
Hence we obtain
\[\E[Y_v]=\frac{k}{n'}\cdot \frac{((1-\delta(v))n'-1)_{k-1}}{(n'-1)_{k-1}}.\]
Recall that we are operating under the assumption that $n$ is sufficiently large with respect to $k$. Hence, $n'\geq n/12$ and $(1-\delta(v))n'\geq n'/2\geq n/24$ are also large with respect to $k$. Hence, the falling factorials in the numerator and denominator of the previous expression can be approximated by powers and we obtain
\begin{equation}\label{eq-E-Yv}
\E[Y_v]=(1+o_k(1))\cdot \frac{k}{n'}\cdot \left(\frac{(1-\delta(v))n'}{n'}\right)^{k-1}=(1+o_k(1))\cdot \frac{k}{n'}\cdot (1-\delta(v))^{k-1}
\end{equation}
for each $v\in \Vl$.

Recall that we have $n'\leq n$ and $\delta(v)\leq 120C/k$ for each $v\in \Vl$. Hence we obtain
\[\E[Y_v]=(1+o_k(1))\cdot \frac{k}{n'}\cdot (1-\delta(v))^{k-1}\geq (1+o_k(1))\cdot \frac{k}{n}\cdot \left(1-\frac{120C}{k}\right)^{k-1}=(1+o_k(1))\cdot \frac{k}{n}\cdot e^{-120C},\]
where we used that $\left(1-\frac{120C}{k}\right)^{k-1}$ converges to $e^{-120C}$ as $k\to \infty$.

Letting $\mu=\E[Y]$, we can use this to obtain a lower bound on $\mu$. Indeed, as $\vert \Vl\vert\geq n/12$ by Lemma \ref{lemma-Vl-large-enough}, we have
\[\mu=\E[Y]=\sum_{v\in \Vl}\E[Y_v]\geq \vert \Vl\vert\cdot (1+o_k(1))\cdot \frac{k}{n}\cdot e^{-120C}\geq (1+o_k(1))\cdot \frac{e^{-120C}}{12}\cdot k.\]
So as long as $k$ is sufficiently large, we have
\begin{equation}\label{ineq-mu-lower-bd}
\mu=\E[Y]\geq \frac{e^{-120C}}{24}\cdot k=2\gamma \cdot k.
\end{equation}

We would like to find an upper bound on the variance $\E[(Y-\mu)^2]=\E[Y^2]-\mu^2$. Note that
\[\E[Y^2]=\sum_{(v,w)\in \Vl^2}\E[Y_vY_w].\]

For every pair $(v,w)\in \Vl^2$, let $0\leq \delta(v,w)\leq 1$ be such that
\[\delta(v,w)n'=\vert N(v)\cap N(w)\cap (\Vl\cup\Vm)\vert,\]
which is the number of vertices in $\Vl\cup\Vm$ that are adjacent to both $v$ and $w$. Note that we clearly have $\delta(v,w)\leq \delta(v)$ and $\delta(v,w)\leq \delta(w)$.

We can now analyze the expectation $\E[Y_vY_w]$ for each pair $(v,w)\in \Vl^2$. Note that $Y_vY_w$ is an indicator random variable for the event that $v,w\in A$ and both of $v$ and $w$ are isolated in $A$. This is impossible if $v$ and $w$ are connected by an edge, so in this case we have $\E[Y_vY_w]=0$. Otherwise, the probability for $v,w\in A$ is
\[\frac{k}{n'}\cdot \frac{k-1}{n'-1}\]
and in order for $v$ and $w$ to be isolated in $A$, the remaining $k-2$ vertices have to be chosen in such a way that none of them is adjacent to $v$ or to $w$. The number of vertices in $\Vl\cup\Vm\sm \lbrace v,w\rbrace$ that are not adjacent to $v$ or to $w$ is
\begin{multline*}
\vert \Vl\cup\Vm\vert-2-\vert N(v)\cap (\Vl\cup\Vm)\vert-\vert N(w)\cap (\Vl\cup\Vm)\vert+\vert N(v)\cap N(w)\cap (\Vl\cup\Vm)\vert\\
=n'-2-\delta(v)n'-\delta(w)n'+\delta(v,w)n'=(1-\delta(v)-\delta(w)+\delta(v,w))n'-2.
\end{multline*}
Thus, if $v$ and $w$ are not adjacent, we obtain
\begin{multline*}
\E[Y_vY_w]=\frac{k}{n'}\cdot \frac{k-1}{n'-1}\cdot \frac{((1-\delta(v)-\delta(w)+\delta(v,w))n'-2)_{k-2}}{(n'-2)_{k-2}}\\
=(1+o_k(1)) \frac{k^2}{n'^2}\frac{((1-\delta(v)-\delta(w)+\delta(v,w))n')^{k-2}}{n'^{k-2}}=(1+o_k(1)) \frac{k^2}{n'^2}(1-\delta(v)-\delta(w)+\delta(v,w))^{k-2},
\end{multline*}
again using that $n'\geq n/12$ is large compared to $k$ as well as $\delta(v)\leq 1/4$ and $\delta(w)\leq 1/4$. So for every pair $(v,w)\in \Vl^2$ (independently of whether or not $v$ and $w$ are adjacent) we have
\[\E[Y_vY_w]\leq (1+o_k(1))\frac{k^2}{n'^2}(1-\delta(v)-\delta(w)+\delta(v,w))^{k-2}=(1+o_k(1))\frac{k^2}{n'^2}(1-\delta(v)-\delta(w)+\delta(v,w))^{k-1},\]
where in the last step we used that $1\geq 1-\delta(v)-\delta(w)+\delta(v,w)\geq 1-(240C/k)$ as $\delta(v)\leq 120C/k$, $\delta(w)\leq 120C/k$ and $\delta(v,w)\leq \delta(w)$.

Note that using $\delta(v,w)\leq \delta(w)$ again, this in particular gives
\begin{equation}\label{ineq-special}
\E[Y_vY_w]\leq (1+o_k(1))\cdot\frac{k^2}{n'^2}\cdot(1-\delta(v))^{k-1}=(1+o_k(1))\cdot\frac{k}{n'}\cdot \E[Y_v]
\end{equation}
for all pairs $(v,w)\in \Vl^2$, where in the last step we plugged in (\ref{eq-E-Yv}).

On the other hand, for all pairs $(v,w)\in \Vl^2$ we also obtain
\begin{multline}\label{ineq-non-special}
\E[Y_vY_w]\leq (1+o_k(1))\cdot \frac{k^2}{n'^2}\cdot(1-\delta(v)-\delta(w)+\delta(v,w))^{k-1}\\
\leq (1+o_k(1))\cdot \frac{k^2}{n'^2}\cdot(1-\delta(v)-\delta(w)+\delta(v)\delta(w)+\delta(v,w))^{k-1}\\
=(1+o_k(1))\cdot \frac{k^2}{n'^2}\cdot(1-\delta(v))^{k-1}\cdot(1-\delta(w))^{k-1}\cdot \left(1+\frac{\delta(v,w)}{(1-\delta(v))\cdot (1-\delta(w))}\right)^{k-1}\\
\leq (1+o_k(1))\cdot\E[Y_v]\cdot \E[Y_w]\cdot (1+2\delta(v,w))^{k-1}\leq (1+o_k(1))\cdot\E[Y_v]\cdot \E[Y_w]\cdot e^{2k\cdot \delta(v,w)},
\end{multline}
where in the second-last step we used (\ref{eq-E-Yv}) as well as $(1-\delta(v))\cdot (1-\delta(w))\geq 9/16\geq 1/2$ (since $\delta(v)\leq 1/4$ and $\delta(w)\leq 1/4$).

Let us call a pair $(v,w)\in \Vl^2$ \emph{special} if $\delta(v,w)\geq 1/(k\cdot \log k)$ and  \emph{non-special} otherwise.

For each non-special pair $(v,w)\in \Vl^2$ we have $2k\cdot \delta(v,w)< 2/\log k$ and so (\ref{ineq-non-special}) gives
\[\E[Y_vY_w]\leq (1+o_k(1))\cdot\E[Y_v]\cdot \E[Y_w]\cdot e^{2/\log k}=(1+o_k(1))\cdot\E[Y_v]\cdot \E[Y_w].\]
Hence
\begin{multline}\label{ineq-non-special-2}
\sum_{\substack{(v,w)\in \Vl^2\\\text{non-special}}}\E[Y_vY_w]\leq (1+o_k(1))\cdot \sum_{\substack{(v,w)\in \Vl^2\\\text{non-special}}}\E[Y_v]\cdot \E[Y_w]\leq (1+o_k(1))\cdot \sum_{(v,w)\in \Vl^2}\E[Y_v]\cdot \E[Y_w]\\
= (1+o_k(1))\cdot\E[Y]^2=(1+o_k(1))\cdot\mu^2
\end{multline}

Recall that each vertex $v\in \Vl$ satisfies $\deg_G(v)\leq 10C\cdot n/k$. In particular, $v$ has at most $10C\cdot n/k$ neighbors $u\in \Vl\cup \Vm$. Each of these neighbors $u$ has degree $\deg_G(u)\leq 10C\cdot n/(\log^2 k)$, so $u$ has at most $10C\cdot n/(\log^2 k)$ neighbors $w\in \Vl$. Thus, for each $v\in \Vl$, the number of triples $(v,w,u)\in \Vl\times \Vl\times (\Vl\cup \Vm)$ such that $u$ is a common neighbor of $v$ and $w$ is at most
\[10C\cdot \frac{n}{k}\cdot 10C\cdot \frac{n}{\log^2 k}=100C^2\cdot \frac{n^2}{k\cdot \log^2 k}.\]
On the other hand the number of such triples $(v,w,u)\in \Vl\times \Vl\times (\Vl\cup \Vm)$ is precisely
\[\sum_{w\in \Vl}\vert N(v)\cap N(w)\cap (\Vl\cup\Vm)\vert=\sum_{w\in \Vl}\delta(v,w)n'.\]
So we obtain (using $n'\geq n/12$)
\[\sum_{w\in \Vl}\delta(v,w)\leq \frac{1}{n'}\cdot 100C^2\cdot \frac{n^2}{k\cdot \log^2 k}\leq \frac{1}{n'}\cdot 120^2C^2\cdot \frac{n'^2}{k\cdot \log^2 k} =120^2C^2\cdot \frac{n'}{k\cdot \log^2 k}.\]
For each $w\in \Vl$ such that $(v,w)$ is special we have $\delta(v,w)\geq 1/(k\cdot \log k)$. Thus, for each $v\in \Vl$, the number of $w\in \Vl$ such that $(v,w)$ is special can be at most
\[\frac{120^2C^2\cdot n'/(k\cdot \log^2 k)}{1/(k\cdot \log k)}=120^2C^2\cdot\frac{n'}{\log k}.\]
Using (\ref{ineq-special}), we now obtain
\begin{multline}\label{ineq-special-2}
\sum_{\substack{(v,w)\in \Vl^2\\\text{special}}}\E[Y_vY_w]\leq (1+o_k(1))\cdot \frac{k}{n'}\cdot \sum_{\substack{(v,w)\in \Vl^2\\\text{special}}}\E[Y_v]\\
\leq (1+o_k(1))\cdot \frac{k}{n'}\cdot \sum_{v\in \Vl}\left(120^2C^2\cdot\frac{n'}{\log k}\cdot\E[Y_v]\right)\\
= (1+o_k(1))\cdot 120^2C^2\cdot \frac{k}{\log k}\cdot \sum_{v\in \Vl}\E[Y_v] =(1+o_k(1))\cdot 120^2C^2\cdot \frac{k}{\log k}\cdot\mu\leq o_k(1)\cdot \mu^2,
\end{multline}
where in the last step we used that $k/\log k\leq o_k(1)\cdot \mu$ as $\mu\geq (e^{-120C}/24)\cdot k$ by (\ref{ineq-mu-lower-bd}).

Combining (\ref{ineq-non-special-2}) and (\ref{ineq-special-2}), we now obtain
\[\E[Y^2]=\sum_{(v,w)\in \Vl^2}\E[Y_vY_w]=\sum_{\substack{(v,w)\in \Vl^2\\\text{non-special}}}\E[Y_vY_w]+\sum_{\substack{(v,w)\in \Vl^2\\\text{special}}}\E[Y_vY_w]\leq (1+o_k(1))\cdot\mu^2.\]
Thus,
\[\E[(Y-\mu)^2]=\E[Y^2]-\mu^2=o_k(1)\cdot \mu^2.\]
Now, Markov's inequality gives
\[\mathbb{P}[Y\leq \mu/2]\leq \mathbb{P}[(Y-\mu)^2\geq \mu^2/4]\leq \frac{\E[(Y-\mu)^2]}{\mu^2/4}\leq \frac{o_k(1)\cdot \mu^2}{\mu^2/4}=o_k(1).\]
Using that $\mu/2\geq \gamma k$ by (\ref{ineq-mu-lower-bd}), this implies
\[\mathbb{P}[Y\leq \gamma k]\leq \mathbb{P}[Y\leq \mu/2]\leq o_k(1).\]
This finishes the proof of Lemma \ref{lemma-suff-isolated}.

\section{Proof of Lemma \ref{lemma-A-interesting}}\label{sect-A-interesting}

In this section we will prove Lemma \ref{lemma-A-interesting}, which states that a $k$-vertex subset $A\su V(G)$ chosen uniformly at random is interesting with probability at most $e^{-1}+o_k(1)$.

First we start with an easy observation: In order for $A$ to be interesting, we need $A\cap \Vh\neq \emptyset$. Each of the $\vert \Vh\vert$ vertices of $\Vh$ is contained in $A$ with probability $k/n$. Hence the probability for $A\cap \Vh\neq \emptyset$ is at most $\vert \Vh\vert\cdot k/n$, and in particular the probability that $A$ is interesting is at most $\vert \Vh\vert\cdot k/n$. Thus, Lemma \ref{lemma-A-interesting} is definitely true if $\vert \Vh\vert\leq e^{-1}\cdot n/k$. So from now, we may assume that $\vert \Vh\vert\geq e^{-1}\cdot n/k$. Since $n$ is large with respect to $k$, this in particular implies $\vert \Vh\vert \geq k$.

Also recall that by Lemma \ref{lemma-Vl-large-enough} we have $\vert \Vm\cup \Vl\vert\geq \vert \Vl\vert\geq n/12$, so we also have $\vert \Vm\cup \Vl\vert\geq k$ as long as $n$ is large enough with respect to $k$.

For a uniformly random $k$-vertex subset $A\su V(G)$, let us consider the quantity $\vert A\cap \Vh\vert$. Clearly $0\leq \vert A\cap \Vh\vert\leq k$. For each $j=0,\dots,k$, let $p_j$ be the probability that $\vert A\cap \Vh\vert=j$ when $A$ is chosen uniformly at random among all $k$-vertex subsets of $V(G)$. Then $p_0+\dots+p_k=1$.

The following claim can be checked via a straightforward calculation. We provide a proof in the appendix.

\begin{claim}\label{claim-p-j-1-e}
We have $p_j\leq e^{-1}+o_k(1)$ for all $1\leq j\leq k-1$.
\end{claim}

We could now model the uniform random choice of the $k$-vertex subset $A\su V(G)$ in the following way: First, choose a random variable $a\in \lbrace 0,\dots,k\rbrace$ by taking $a=j$ with probability $p_j$ for each $j=0,\dots,k$. Then choose an $a$-vertex subset of $\Vh$ uniformly at random and independently of that choose a $(k-a)$-vertex subset of $V(G)\sm \Vh=\Vl\cup \Vm$ uniformly at random. Then taking $A$ to be the union of those two sets results in a uniformly random choice of $A$ among all $k$-vertex subsets of $V(G)$.

However, we will need a slightly more complicated variant of this idea. So instead, let us model the random choice of $A$ as follows, where each of the three steps is performed independently.
\begin{itemize}
\item First choose a sequence $v_1,\dots,v_k$ of distinct vertices in $\Vh$ uniformly at random among all such sequences of distinct vertices.
\item Then choose a $k$-vertex subset $W\su \Vm\cup \Vl$ uniformly at random. Now, choose a vertex $w_1\in W$ uniformly at random. Afterwards, choose a vertex $w_2\in W\sm \lbrace w_1\rbrace$ uniformly at random. Continue like this, for each $i=1,\dots,k$ choosing a vertex $w_i\in W\sm \lbrace w_1,\dots,w_{i-1}\rbrace$ uniformly at random. This results in a sequence $w_1,\dots,w_k$ of distinct vertices of $\Vm\cup \Vl$ with $W=
\lbrace w_1,\dots,w_k\rbrace$.
\item Finally, choose a random variable $a\in \lbrace 0,\dots,k\rbrace$ by taking $a=j$ with probability $p_j$ for each $j=0,\dots,k$.
\end{itemize}
After having made all these random choices, set $A=\lbrace v_1,\dots,v_a\rbrace\cup \lbrace w_{a+1},\dots,w_k\rbrace$.

Note that for every fixed $a$, the set $\lbrace v_1,\dots,v_a\rbrace$ will be uniformly random among all $a$-vertex subsets of $\Vh$ and the set $\lbrace w_{a+1},\dots,w_k\rbrace$ will be uniformly random among all $(k-a)$-vertex subsets of $\Vm\cup \Vl$. Therefore, this random procedure yields the same probability distribution for the set $A$ as the slightly simpler procedure described above. In particular, we see that the resulting set $A$ will have a uniform distribution among all $k$-vertex subsets of $V(G)$.

So let us from now on assume that the random set $A$ is chosen in the described way. Our goal is to show that $A$ is interesting with probability at most $e^{-1}+o_k(1)$.

If $A$ is interesting, we must have $A\cap \Vh\neq\emptyset$ and therefore $a=\vert A\cap \Vh\vert\geq 1$.

Furthermore, if $A$ is interesting, for each vertex $v_j$ with $1\leq j\leq a$ we must have
\begin{equation}\label{ineq-v-j-A-int}
\left\vert\deg_A(v_j)-\frac{k-1}{n}\cdot \deg_G(v_j)\right\vert\leq  \sqrt{k\cdot \log k}.
\end{equation}
Hence, for each vertex $v_j$ with $1\leq j\leq a$,
\begin{multline*}
\deg_A(v_j)\geq \frac{k-1}{n}\cdot \deg_G(v_j)- \sqrt{k\cdot \log k}\geq 
\frac{k-1}{n}\cdot 10C\cdot \frac{n}{\log^2 k}- \sqrt{k\cdot \log k}\\
\geq 10C\cdot \frac{k/2}{\log^2 k}- \sqrt{k\cdot \log k}\geq 5C\cdot \frac{k}{\log^2 k}- \frac{k}{\log^2 k}\geq 4C\cdot \frac{k}{\log^2 k},
\end{multline*}
where we used that $v_j\in \Vh$ as well as the assumption $k\geq 4\cdot \log^{10} k$ made at the beginning of Section \ref{sect-proof-graph-1-e}.

On the other hand, if $A$ is interesting, we must also have $e(A)=\l$ and therefore
\[\sum_{j=1}^{a}\deg_A(v_j)\leq \sum_{v\in A}\deg_A(v)=2e(A)=2\l\leq 2C\cdot k.\]
Thus, if $A$ is interesting, we obtain
\[a\cdot 4C\cdot \frac{k}{\log^2 k}\leq \sum_{j=1}^{a}\deg_A(v_j)\leq 2C\cdot k\]
and consequently $a\leq (\log^2 k)/2\leq \log^2 k$.

So we have shown that if $A$ is interesting, then we must have $1\leq a\leq \log^2 k$.

Note that we always have
\[e(A)=e(\lbrace v_1,\dots,v_a\rbrace\cup \lbrace w_{a+1},\dots,w_k\rbrace)=\sum_{j=1}^{a}\deg_A(v_j)-e(\lbrace v_1,\dots,v_a\rbrace)+e(\lbrace w_{a+1},\dots,w_k\rbrace).\]
If $A$ is interesting, then we have $e(A)=\l$ and therefore
\[e(W\sm \lbrace w_{1},\dots,w_a\rbrace)=e(\lbrace w_{a+1},\dots,w_k\rbrace)=\l+e(\lbrace v_1,\dots,v_a\rbrace)-\sum_{j=1}^{a}\deg_A(v_j).\]
Now, (\ref{ineq-v-j-A-int}) and $a\leq \log^2 k$ yield
\begin{multline*}
e(W\sm \lbrace w_{1},\dots,w_a\rbrace)\geq \l-\sum_{j=1}^{a}\left(\frac{k-1}{n}\cdot \deg_G(v_j)+\sqrt{k\cdot \log k}\right)\\
\geq \l-\frac{k-1}{n}\cdot\left(\sum_{j=1}^{a}\deg_G(v_j)\right)-\log^2 k\cdot \sqrt{k\cdot \log k}\\
\geq \l-\frac{k-1}{n}\cdot\left(\sum_{j=1}^{a}\deg_G(v_j)\right)-\sqrt{k}\cdot \log^3 k
\end{multline*}
as well as
\begin{multline*}
e(W\sm \lbrace w_{1},\dots,w_a\rbrace)\leq \l+{a\choose 2}-\sum_{j=1}^{a}\left(\frac{k-1}{n}\cdot \deg_G(v_j)-\sqrt{k\cdot \log k}\right)\\
\leq \l+(\log^2 k)^2-\frac{k-1}{n}\cdot\left(\sum_{j=1}^{a}\deg_G(v_j)\right)+\log^2 k\cdot \sqrt{k\cdot \log k}\\
\leq \l-\frac{k-1}{n}\cdot\left(\sum_{j=1}^{a}\deg_G(v_j)\right)+\sqrt{k}\cdot \log^3 k.
\end{multline*}

All in all, if $A$ is interesting, we must have $1\leq a\leq \lfloor\log^2 k\rfloor$ and
\[ \l-\frac{k-1}{n}\left(\sum_{j=1}^{a}\deg_G(v_j)\right)-\sqrt{k}\cdot \log^3 k\leq e(W\sm \lbrace w_{1},\dots,w_a\rbrace)\leq \l-\frac{k-1}{n}\left(\sum_{j=1}^{a}\deg_G(v_j)\right)+\sqrt{k}\cdot \log^3 k.\]

For $t=1,\dots,\lfloor\log^2 k\rfloor$, let us define $\mathscr{E}_t$ to be the event
\[ \l-\frac{k-1}{n}\left(\sum_{j=1}^{t}\deg_G(v_j)\right)-\sqrt{k}\cdot \log^3 k\leq e(W\sm \lbrace w_{1},\dots,w_t\rbrace)\leq \l-\frac{k-1}{n}\left(\sum_{j=1}^{t}\deg_G(v_j)\right)+\sqrt{k}\cdot \log^3 k.\]
Note that this event $\mathscr{E}_t$ only depends on $v_1,\dots,v_t$ as well as $W$ and $w_1,\dots,w_t$. In particular, $\mathscr{E}_t$ is not influenced by the random choice of $a$ and the random choices of $w_{t+1},\dots,w_k$ inside $W$ (recall that we chose $w_{t+1},\dots,w_k$ after $W$ and $w_1,\dots,w_t$).

We saw above that in order for $A$ to be interesting, we must have $1\leq a\leq \lfloor\log^2 k\rfloor$ and the event $\mathscr{E}_a$ needs to hold. Since each of the events $\mathscr{E}_t$ is independent of the choice of $a$, this implies
\[\mathbb{P}\left[A\text{ interesting}\right]\leq \mathbb{P}\left[1\leq a\leq \lfloor\log^2 k\rfloor\text{ and }\mathscr{E}_a\text{ holds}\right]=\sum_{t=1}^{\lfloor\log^2 k\rfloor} \mathbb{P}[a=t]\cdot \mathbb{P}[\mathscr{E}_t]=\sum_{t=1}^{\lfloor\log^2 k\rfloor} p_t\cdot \mathbb{P}[\mathscr{E}_t]\]
By Claim \ref{claim-p-j-1-e} we have $p_t\leq e^{-1}+o_k(1)$ for $t=1, \dots, \lfloor\log^2 k\rfloor$. So we obtain
\begin{equation}\label{ineq-A-interesting-E-t}
\mathbb{P}\left[A\text{ interesting}\right]\leq \left(\frac{1}{e}+o_k(1)\right)\cdot \sum_{t=1}^{\lfloor\log^2 k\rfloor} \mathbb{P}[\mathscr{E}_t].
\end{equation}

Our goal is now to prove that $\sum_{t=1}^{\lfloor\log^2 k\rfloor} \mathbb{P}[\mathscr{E}_t]\leq 1+o_k(1)$. First, note that by the inclusion-exclusion principle we have
\[0\leq \mathbb{P}\left[\text{no }\mathscr{E}_t\text{ for }t=1, \dots, \lfloor\log^2 k\rfloor\text{ holds}\right]\\
\leq 1-\sum_{t=1}^{\lfloor\log^2 k\rfloor} \mathbb{P}[\mathscr{E}_t]+\sum_{1\leq t<t'\leq \lfloor\log^2 k\rfloor}\mathbb{P}[\mathscr{E}_t\wedge\mathscr{E}_{t'}].\]
Thus,
\begin{equation}\label{ineq-sum-e-t}
\sum_{t=1}^{\lfloor\log^2 k\rfloor} \mathbb{P}[\mathscr{E}_t]\leq 1+\sum_{1\leq t<t'\leq \lfloor\log^2 k\rfloor}\mathbb{P}[\mathscr{E}_t\wedge\mathscr{E}_{t'}].
\end{equation}

We will now find an upper bound for each of the terms $\mathbb{P}[\mathscr{E}_t\wedge\mathscr{E}_{t'}]$ for $1\leq t<t'\leq \lfloor\log^2 k\rfloor$.

So let us, for now, fix some indices $t$ and $t'$ with $1\leq t<t'\leq \lfloor\log^2 k\rfloor$. Let $W'=W\sm \lbrace w_1,\dots,w_t\rbrace$. If both $\mathscr{E}_t$ and $\mathscr{E}_{t'}$ hold, we must have
\[e(W')=e(W\sm \lbrace w_{1},\dots,w_t\rbrace)\geq \l-\frac{k-1}{n}\cdot\left(\sum_{j=1}^{t}\deg_G(v_j)\right)-\sqrt{k}\cdot \log^3 k\]
and
\[e(W'\sm \lbrace w_{t+1},\dots,w_{t'}\rbrace)=e(W\sm \lbrace w_{1},\dots,w_{t'}\rbrace)\leq \l-\frac{k-1}{n}\cdot\left(\sum_{j=1}^{t'}\deg_G(v_j)\right)+\sqrt{k}\cdot \log^3 k.\]
Hence
\begin{multline*}
\sum_{j=t+1}^{t'}\deg_{W'}(w_j)\geq e(W')-e(W'\sm \lbrace w_{t+1},\dots,w_{t'}\rbrace)\\
\geq \left(\l-\frac{k-1}{n}\cdot\left(\sum_{j=1}^{t}\deg_G(v_j)\right)-\sqrt{k}\cdot \log^3 k\right)-\left(\l-\frac{k-1}{n}\cdot\left(\sum_{j=1}^{t'}\deg_G(v_j)\right)+\sqrt{k}\cdot \log^3 k\right)\\
=\frac{k-1}{n}\cdot\left(\sum_{j=t+1}^{t'}\deg_G(v_j)\right)-2\cdot \sqrt{k}\cdot \log^3 k\geq 
\frac{k-1}{n}\cdot(t'-t)\cdot 10C\cdot \frac{n}{\log^2 k}-2\cdot \sqrt{k}\cdot \log^3 k\\
\geq (t'-t)\cdot 9C\cdot \frac{k}{\log^2 k}-2\cdot \sqrt{k}\cdot \log^3 k
\geq (t'-t)\cdot 8C\cdot \frac{k}{\log^2 k},
\end{multline*}
where we used $\deg_G(v_j)\geq 10C\cdot n/(\log^2 k)$ for $j=t+1,\dots,t'$ as $v_j\in \Vh$, as well as $(k-1)/k\geq 9/10$ and $2\sqrt{k}\cdot \log^3 k\leq k/(\log^2 k)$ since we assumed $k\geq 4\cdot \log^{10} k$ at the beginning of Section \ref{sect-proof-graph-1-e}. Thus, at least one of the vertices $w_j$ for $j=t+1,\dots,t'$ must satisfy $\deg_{W'}(w_j)\geq 8C\cdot k/(\log^2 k)$ if both of the events $\mathscr{E}_t$ and $\mathscr{E}_{t'}$ hold.

Let us now fix any choice of $v_1,\dots,v_k$, $W$ and $w_1,\dots,w_t$ such that the event $\mathscr{E}_t$ holds. We want to prove that then the probability for $\mathscr{E}_{t'}$ to hold as well is small (the choice of $w_{t+1},\dots,w_k$ in $W\sm \lbrace w_1,\dots,w_t\rbrace$ is still random as described above). As before, let $W'=W\sm \lbrace w_1,\dots,w_t\rbrace$. Since $\mathscr{E}_t$ holds we have
\[e(W')=e(W\sm \lbrace w_1,\dots,w_t\rbrace)\leq \l-\frac{k-1}{n}\cdot\left(\sum_{j=1}^{t}\deg_G(v_j)\right)+\sqrt{k}\cdot \log^3 k\leq C\cdot k+\sqrt{k}\cdot \log^3 k\leq 2C\cdot k,\]
where we used the assumption $k\geq 4\cdot \log^{10} k$ again. Thus,
\[\sum_{w\in W'}\deg_{W'}(w)=2e(W')\leq 4C\cdot k.\]
This in particular implies that the number of vertices $w\in W'$ with $\deg_{W'}(w)\geq 8C\cdot k/(\log^2 k)$ is at most $(\log^2 k)/2$. We saw above that in order for the event $\mathscr{E}_{t'}$ to hold in addition to $\mathscr{E}_{t}$, we must have $\deg_{W'}(w_j)\geq 8C\cdot k/(\log^2 k)$ for at least one index $j\in \lbrace t+1,\dots,t'\rbrace$. There are $t'-t$ choices for this index $j$ and for each fixed $j\in \lbrace t+1,\dots,t'\rbrace$ the vertex $w_j$ is a uniformly random element of $W'=W\sm \lbrace w_1,\dots,w_t\rbrace$. Hence for each fixed $j$ the probability of $\deg_{W'}(w_j)\geq 8C\cdot k/(\log^2 k)$ is at most
\[\frac{(\log^2 k)/2}{\vert W'\vert}=\frac{(\log^2 k)/2}{k-t}\leq \frac{(\log^2 k)/2}{k-\log^2 k}\leq \frac{(\log^2 k)/2}{k/2}=\frac{\log^2 k}{k}.\]
By the union bound, the probability for there to exist an index $j\in \lbrace t+1,\dots,t'\rbrace$ with $\deg_{W'}(w_j)\geq 8C\cdot k/(\log^2 k)$ is therefore at most
\[(t'-t)\cdot \frac{\log^2 k}{k}\leq (\log^2 k)\cdot \frac{\log^2 k}{k}=\frac{\log^4 k}{k},\]
where we used that $1\leq t<t'\leq \lfloor\log^2 k\rfloor$. This proves that for any fixed choice of $v_1,\dots,v_k$, $W$ and $w_1,\dots,w_t$ such that the event $\mathscr{E}_t$ holds, the probability for the event $\mathscr{E}_{t'}$ to hold as well is at most $(\log^4 k)/k$. Hence, the overall probability for both $\mathscr{E}_t$ and $\mathscr{E}_{t'}$ to hold is also at most $(\log^4 k)/k$.

So we have shown that for any fixed indices $t$ and $t'$ with $1\leq t<t'\leq \lfloor\log^2 k\rfloor$, we have
\[\mathbb{P}[\mathscr{E}_t\wedge\mathscr{E}_{t'}]\leq \frac{\log^4 k}{k}.\]
Thus, (\ref{ineq-sum-e-t}) yields
\[\sum_{t=1}^{\lfloor\log^2 k\rfloor} \mathbb{P}[\mathscr{E}_t]\leq 1+\sum_{1\leq t<t'\leq \lfloor\log^2 k\rfloor}\mathbb{P}[\mathscr{E}_t\wedge\mathscr{E}_{t'}]\leq 1+\sum_{1\leq t<t'\leq \lfloor\log^2 k\rfloor}\frac{\log^4 k}{k}\leq 1+\frac{\log^8 k}{k}=1+o_k(1).\]
Plugging this into (\ref{ineq-A-interesting-E-t}), we obtain
\[\mathbb{P}\left[A\text{ interesting}\right]\leq \left(\frac{1}{e}+o_k(1)\right)\cdot \sum_{t=1}^{\lfloor\log^2 k\rfloor} \mathbb{P}[\mathscr{E}_t]\leq \left(\frac{1}{e}+o_k(1)\right)\cdot (1+o_k(1))=\frac{1}{e}+o_k(1),\]
as desired. This finishes the proof of Lemma \ref{lemma-A-interesting}.

\appendix 
\section*{Appendix}

Here, we provide a proof of Claim \ref{claim-p-j-1-e}. As mentioned above, the proof is a straightforward calculation.

Let $1\leq j\leq k-1$. By definition, $p_j$ is the probability that a uniformly random $k$-vertex subset $A\su V(G)$ satisfies $\vert A\cap \Vh\vert =j$. Hence
\begin{multline*}
p_j={k\choose j}\cdot \frac{(\vert \Vh\vert)_j\cdot \vert (\Vl\cup \Vm\vert)_{k-j}}{(n)_k}\leq {k\choose j}\cdot \frac{\vert \Vh\vert^j\cdot \vert \Vl\cup \Vm\vert^{k-j}}{(n)_k}\\
=(1+o_k(1))\cdot {k\choose j}\cdot \frac{\vert \Vh\vert^j\cdot \vert \Vl\cup \Vm\vert^{k-j}}{n^k},
\end{multline*}
where in the last step we used that $n$ is large in terms of $k$.

Let $x$ be such that $\vert \Vh\vert=xn$, then $\vert \Vl\cup \Vm\vert=(1-x)n$. Recall that we assumed $\vert \Vh\vert\geq k$ and $\vert \Vl\cup \Vm\vert\geq k$ at the beginning of Section \ref{sect-A-interesting}, and consequently $0<x<1$. Now we have
\[p_j\leq (1+o_k(1))\cdot {k\choose j}\cdot \frac{(xn)^j\cdot ((1-x)n)^{k-j}}{n^k}=(1+o_k(1))\cdot \frac{k!}{j!\cdot (k-j)!}\cdot x^{j}\cdot (1-x)^{k-j}.\]
Note that the inequality between arithmetic and geometric mean implies
\begin{multline*}
x^j(1-x)^{k-j}=j^j\cdot (k-j)^{k-j}\cdot\left(\frac{x}{j}\right)^j\cdot\left(\frac{1-x}{k-j}\right)^{k-j}\\
\leq j^j\cdot (k-j)^{k-j}\cdot\left(\frac{j\cdot \frac{x}{j}+(k-j)\cdot \frac{1-x}{k-j}}{j+k-j}\right)^{j+k-j}=j^j\cdot (k-j)^{k-j}\cdot\left(\frac{1}{k}\right)^{k}.
\end{multline*}
Thus,
\[p_j\leq (1+o_k(1))\cdot \frac{k!}{j!\cdot (k-j)!}\cdot x^{j}\cdot (1-x)^{k-j}\leq (1+o_k(1))\cdot \frac{j^j}{j!}\cdot \frac{(k-j)^{k-j}}{(k-j)!}\cdot \frac{k!}{k^k}.\]
Hence it suffices to prove
\[\frac{j^j}{j!}\cdot\frac{(k-j)^{k-j}}{(k-j)!}\cdot \frac{k!}{k^k}\leq \frac{1}{e}+o_k(1)\]
for all $1\leq j\leq k-1$. For this, we can assume without loss of generality that $1\leq j\leq k/2$.

If $2\leq j\leq k/2$, then $j\cdot (k-j)\geq 2\cdot (k-2)$ and so Stirling's formula (see for example \cite{robbins}) yields
\begin{multline*}
\frac{j^j}{j!}\cdot\frac{(k-j)^{k-j}}{(k-j)!}\cdot \frac{k!}{k^k}\leq \left(\sqrt{2\pi}\cdot \sqrt{j}\cdot e^{-j}\right)^{-1} \left(\sqrt{2\pi}\cdot \sqrt{k-j}\cdot e^{-k+j}\right)^{-1} \left((1+o_k(1))\sqrt{2\pi}\cdot \sqrt{k}\cdot e^{-k}\right)\\
=(1+o_k(1)) \frac{1}{\sqrt{2\pi}}\cdot \sqrt{\frac{k}{j\cdot (k-j)}}\leq (1+o_k(1)) \frac{1}{\sqrt{2\pi}}\cdot \sqrt{\frac{k}{2\cdot (k-2)}}=(1+o_k(1)) \frac{1}{\sqrt{4\pi}}< \frac{1}{e}+o_k(1),
\end{multline*}
where in the last step we used $4\pi>12>9>e^2$. So it only remains to consider the case $j=1$. Then we simply have
\[\frac{1^1}{1!}\cdot\frac{(k-1)^{k-1}}{(k-1)!}\cdot \frac{k!}{k^k}=\frac{(k-1)^{k-1}}{k^{k-1}}=\left(1-\frac{1}{k}\right)^{k-1}\leq e^{-(k-1)/k}=e^{-1}\cdot e^{1/k}=\frac{1}{e}+o_k(1).\]
This finishes the proof of Claim \ref{claim-p-j-1-e}.

\section*{Acknowledgments} 
We would like to thank the anonymous referee for their helpful comments.
\\
\\
\noindent\textit{Remark.} After completing this work, we learned that Martinsson, Mousset, Noever, and Truji\'c \cite{martinsson-et-al}  also found a completion of the proof of the Edge-statistics Conjecture.

\bibliographystyle{amsplain}


\begin{aicauthors}
\begin{authorinfo}[jacob]
  Jacob Fox\\
  Department of Mathematics\\
  Stanford University\\
  Stanford, USA\\
  jacobfox\imageat{}stanford\imagedot{}edu\\
  \url{https://stanford.edu/~jacobfox}
\end{authorinfo}
\begin{authorinfo}[lisa]
  Lisa Sauermann\\
  Department of Mathematics\\
  Stanford University\\
  Stanford, USA\\
  lsauerma\imageat{}stanford\imagedot{}edu\\
  \url{https://stanford.edu/~lsauerma}
\end{authorinfo}
\end{aicauthors}

\end{document}